\documentclass[12pt,a4paper]{amsart}
\usepackage{graphicx,latexsym,amsfonts,amsmath,amssymb,rotating,txfonts,mathrsfs,enumerate, mathtools}
\usepackage[utf8]{inputenc}
\usepackage{epic}
\usepackage{curves}
\usepackage{pdfsync}
\usepackage{ytableau}
\usepackage{tikz}
\usepackage{calrsfs}

\input xy
\xyoption{all}

\newcommand{\Hom}{ \,{\rm Hom} \,}
\newcommand{\Sym}{ \,{\rm Sym} \,}

\newcommand{\im}{ \,{\rm Im} \,}

{\bf}{\rm}
\newtheorem{theorem}{Theorem}[section]
\newtheorem*{theorem*}{Theorem}
\newtheorem{proposition}[theorem]{Proposition}
\newtheorem{corollary}[theorem]{Corollary}
\newtheorem{lemma}[theorem]{Lemma}
\newtheorem{definition}[theorem]{Definition}
\newtheorem{remark}[theorem]{Remark}
\newtheorem{conjecture}[theorem]{Conjecture}
\newtheorem{example}[theorem]{Example}
{\bf}{\it}

\newcommand{\CC}{{\mathbb C }}

\newcommand{\ZZ}{{\mathbb Z }}
\newcommand{\PP}{ {\mathbb P }}
\newcommand{\QQ}{{\mathbb Q }}

\newcommand{\ff}{{\mathbf f }}

\newcommand{\calc}{\mathcal{C}}

\newcommand{\calH}{\mathcal{H}}

\newcommand{\calo}{\mathcal{O}}

\newcommand{\cale}{\mathcal{E}}

\newcommand{\frakm}{\mathfrak{m}}

\newcommand{\reg}{\mathrm{reg}}

\newcommand{\Euler}{\mathrm{Euler}}

\newcommand{\epd}[1]{\mathrm{eP}[#1]}
\newcommand{\mdeg}[1]{\mathrm{mdeg}[#1]}
\newcommand{\emu}{\mathrm{emult}}

\newcommand{\Span}{\mathrm{Span}}

\newcommand{\sg}[1]{\mathcal{S}_{#1}}
\newcommand{\res}{\operatornamewithlimits{Res}}

\newcommand{\ires}{\res_{z_1=\infty}\res_{z_{2}=\infty}\dots\res_{z_k=\infty}}
\newcommand{\sires}{\res_{\mathbf{z}=\infty}}
\newcommand{\siresw}{\res_{\mathbf{w}=\infty}}
\newcommand{\dbz}{\,d\mathbf{z}}
\newcommand{\dbw}{\,d\mathbf{w}}

\newcommand{\symdot}{\mathrm{Sym}^{\le k}\CC^n}
\newcommand{\symdotl}{\mathrm{Sym}^{\le \omega(\lambda)}\CC^n}

\newcommand{\grass}{\mathrm{Grass}}

\newcommand{\flag}{\mathrm{Flag}}
\newcommand{\diff}{\mathrm{Diff}}
\newcommand{\TT}{\mathrm{T}}

\newcommand{\bz}{\mathbf{z}}
\newcommand{\bw}{\mathbf{w}}
\newcommand{\bi}{\mathbf{i}}

\newcommand{\bj}{\mathbf{j}}

\newcommand{\bv}{\mathbf{v}}

\newcommand{\bd}{\mathbf{d}}

\newcommand{\bff}{\mathbf{f}}
\newcommand{\bA}{\mathbf{A}}
\newcommand{\bU}{\mathbf{U}}
\newcommand{\kt}{{K}}

\newcommand{\jetregl}[2]{J_{\lambda}^{\mathrm{reg}}({#1},{#2})}

\newcommand{\tc}{\hat T}

\newcommand{\GL}{\mathrm{GL}}
\newcommand{\sym}{\mathrm{Sym}}

\newcommand{\Hilb}{\mathrm{Hilb}}
\newcommand{\CHilb}{\mathrm{CHilb}}
\newcommand{\GHilb}{\mathrm{GHilb}}
\newcommand{\Curv}{\mathrm{Curv}}
\newcommand{\RHilb}{\mathrm{RHilb}}
\newcommand{\BHilb}{\mathrm{BHilb}}
\newcommand{\NAHilb}{\mathrm{NAHilb}}
\newcommand{\N}{\mathrm{N}}
\newcommand{\CN}{\mathrm{CN}}
\newcommand{\BN}{\mathrm{BN}}
\newcommand{\HC}{\mathrm{HC}}
\newcommand{\supp}{\mathrm{supp}}
\newcommand{\Thom}{\mathrm{Thom}}
\newcommand{\Stab}{\mathrm{Stab}}
\def\a{\alpha}
\def\b{\beta}
\def\g{\gamma}

\def\l{\lambda}

\def\s{\sigma}

\def\vp{\varphi}

\def\L{\Lambda}

\title[Tautological integrals on Hilbert scheme of points II]{Tautological integrals on Hilbert scheme of points II: Geometric subsets} 

\author{Gergely B\'erczi}
\address{Department of Mathematics, Aarhus University}
\email{gergely.berczi@math.au.dk}

\date{}
\begin{document}

\begin{abstract}
We develop a formula for tautological integrals over geometric subsets of the Hilbert scheme of points on complex manifolds. As an illustration of the theory, we derive a new iterated residue formula for the number of nodal curves in sufficiently ample linear systems. 
\end{abstract}

\maketitle
{\scriptsize
\tableofcontents}

\section{Introduction}\label{sec:intro}
%We take a new look at the curvilinear Hilbert scheme of points on a smooth projective variety $X$ as a projective completion of the %non-reductive quotient of holomorphic map germs from the complex line into $X$ by polynomial reparametrisations. Using an %algebraic model of this quotient coming from global singularity theory we develop an iterated residue formula for tautological %integrals over curvilinear Hilbert schemes.

%\subsection{Background}
The main goal of this paper is to extend the theory of integration we developed in \cite{bercziG&T, berczitau2} on the main component of the Hilbert scheme of points over a complex manifold to geometric subsets of the Hilbert scheme. This new approach has important applications in enumerative geometry, and we devote \cite{berczitau4} to study some of these problems. %such as (i) counting curves---or more generally, hypersurfaces---with given set of singularities in a sufficiently ample linear system, (ii) higher dimensional Nakajima calculus and (iii) higher dimensional Segre and Verlinde integrals.

Let $X$ be a smooth complex variety of dimension $n$, and let $\Hilb^k(X)$ be the Hilbert scheme of $k$ points on $X$, formed by length-k-subschemes on $X$. To a rank-r vector bundle $F$ over $X$ we can associate a rank $rk$ tautological bundle $F^{[k]}$ over $\Hilb^k(X)$,  whose fibre at $\xi \in \Hilb^k(X)$ is $H^0(\xi,F|_\xi)$. In \cite{berczitau2} we developed a formula for tautological integrals over the main (also called geometric) component 
\[\GHilb^k(X)=\overline{\{\xi \in \Hilb^k(X): \xi=p_1 \sqcup \ldots \sqcup p_k: p_i \neq p_j\}},\]
which is the closure of the locus of non-reduced subschemes. This is a singular component of dimension $k \cdot \dim(X)$, and can be considered as a canonical compactification of the configurations space of $k$ distinct points on $X$. It is the most  interesting and relevant part of $\Hilb^k(X)$ from an enumerative geometry viewpoint, because several classical enumerative geometry problems can be reformulated using its topology and intersection theory, see \cite{nakajima,berczitau4}. 

In this paper we generalise this integral formula to geometric subsets of $\Hilb^k(X)$. Let $A_1,\ldots, A_s$ be finite dimensional quotient algebras of $\CC[x_1, \ldots, x_n]$ of dimension $\dim_\CC(A_i)=k_i$ with $k=k_1+\ldots +k_s$. The corresponding geometric subset
\[\Hilb^{A_1,\ldots, A_s}(X)=\overline{\{\xi=\xi_1 \sqcup \ldots \sqcup \xi_s: \calo_{\xi_i} \simeq A_i\}}\subset \Hilb^k(\CC^n)\]
is the closure of the $(A_1,\ldots, A_s)$-locus fomed by subschemes supported on $s$ points with prescribed local algebra at each point. It was first observed by Gottsche \cite{gottsche} that the number of $\delta$-nodal curves in a sufficiently ample linear system over a surface can be rewritten as the top Chern number of a tautological bundle over the geometric subset where $A_i=\CC[x,y]/(x^2,xy,y^2)$ for $1\le i \le \delta$. His idea works in higher dimensions for hypersurface counts with given set of singularities. 

 %If $\frakm=(x_1,\ldots, x_n)$ is the maximal ideal at the origin then $N_i=A_i \cap \frakm$ is a nilpotent algebra of dimension $k-1$. 
A key object we introduce in this paper is the sum algebra $A_1+\ldots +A_s$, which is a monomial algebra of dimension $k=k_1+\ldots +k_s$. This is characterised by the property that a generic point of $\Hilb^{A_1+\ldots+A_s}(X)$ arises when points in $\Hilb^{A_1}(X),\ldots, \Hilb^{A_s}(X)$ collide along a smooth curve in $X$. We prove that $\Hilb^{A_1+\ldots+A_s}(X)$ is an irreducible component of the punctual part $\Hilb_0^{A_1,\ldots ,A_s}(X)=\Hilb^k_0(\CC^n) \cap \Hilb^{A_1,\ldots ,A_s}(X)$ and we call this the curvilinear component of the geometric subset. The codimension of $\Hilb^{A_1+\ldots+A_s}(X)$ in $\Hilb^{A_1, \ldots ,A_s}(X)$ is $s-1$, independent of the defining algebras. 

If $A_i=\CC[x_1,\ldots, x_n]/I_{\lambda_i}$ ($i=1,\ldots, s$) for some monomial ideals $I_{\lambda_i}$ indexed by the Young tableau $\lambda_i$ then we call $\Hilb^{A_1,\ldots, A_s}(X)$ a monomial geometric subset.
For monomial algebras $A_i$ their sum $A_1+\ldots +A_s$ is again a monomial algebra whose Young tableau can be explicitly determined from those of $A_1,\ldots, A_s$. 

Let $A=\CC[x_1,\ldots, x_n]/I$ be an algebra of dimension $k$. Choose a filtration on the nilpotent algebra $N=\frakm/I$ (whose dimension is $k-1$), that is, a decreasing sequence of subalgebras
\[N = N_1 \supset N_2 \supset \ldots \supset N_{m+1}=0\]
satisfying $N_i \cdot N_j \supset N_{i+j}$. Let $d_i = \dim(N_i/N_{i+1})$ and $\bd=(d_1,\ldots , d_m)$ be the dimension vector which satisfies $d_1+\ldots +d_m=k-1$. We call this filtration natural if any automorphism of the nilpotent algebra $N$ preserves this filtration.
Natural filtrations on the nilpotent algebra always exist but it is usually not unique. The formula for the tautological integral presented below depends on the choice of the natural filtration, but the iterated residue does not.  A canonical filtration is given by the powers of the maximal ideal, that is, $N_k = N^k$. In this case all components $d_i$ of the dimension vector are strictly positive. But in general, we allow for some components of the dimension vector to be vanishing.
Let $\mathrm{Alg}_\bd(N_\bullet)$ denote the vector space of filtered commutative algebra structures on the filtration $N_\bullet$, and let $Q(A) \subset \mathrm{Alg}_\bd(N)$ be the subset of those algebras which are isomorphic to $\mathfrak{m}/I$.  The group $\GL_\bd(N_\bullet)$ of filtration-preserving linear transformations of $N_\bullet$ preserves the subvariety $Q(A) \subset \mathrm{Alg}_\bd(N_\bullet)$ and we can define the equivariant Poincar\'e dual $\epd{Q(A) \subset \mathrm{Alg}_\bd(N_\bullet)}$, see \cite{kazarian}.

First we state the theorem for $s=1$, that is, geometric subsets with a single-point-support. This is a reformulation of Kazarian's result in \cite{kazarian3}, but we revisit its proof and study an alternative formula using test curve models. 

\begin{theorem}[\textbf{Tautological integrals on punctual geometric subsets}]\label{main1} Let $X$ be a smooth projective variety of dimension $n$ and $F$ a rank $r$ bundle on $X$ with Chern roots $\theta_1,\ldots, \theta_r$. Let  $A=\CC[x_1,\ldots, x_n]/I$ be an algebra of dimension $k$ and let $N_\bullet$ be arbitrary natural filtration on $N=\frakm/I$ with dimension vector $\bd=(d_1,\ldots, d_m)$. 
For an integer $i$ such that $d_1+ \ldots + d_{j-1}< i \le d_1+ \ldots +d_j$, we define its weight to be $w(i)=j$. Let $\Phi(F^{[k]})$ be a Chern polynomial of degree $\dim \Hilb^A(X)$ in the Chern classes of the tautological bundle $F^{[k]}$. Then  
\[\int_{\Hilb^A(X)} \Phi(F^{[k]})=\int_X \sires \frac{\prod_{\substack{1\le i,j\le k-1 \\ w(i)\le w(j)}}(z_i-z_j)\epd{Q(A) \subset \mathrm{Alg}_\bd(N_\bullet)}\Phi(F(\bz))d\bz}{\prod_{w(i)+w(j)\le w(m)}(z_i+z_j-z_m)(z_1\ldots z_{k-1})^n}\prod_{i=1}^{k-1} s_X\left(\frac{1}{z_i}\right).\]
where
\begin{itemize}
\item The formula contains $k-1$ residue variables $z_1,\ldots, z_{k-1}$, and the iterated residue $\sires$ is the coefficient of $(z_1\ldots z_{k-1})^{-1}$ in the expansion of the rational expression in the domain $z_1\ll \ldots \ll z_{k-1}$. 
\item $\Phi(F(\bz))=\Phi(\theta+\bz,\theta)$ stands for the symmetric polynomial in the twisted Chern roots $\{\theta_j+z_i,\theta_j:1\le i\le k-1,1\le j\le r\}$.   
\item $s_X\left(\frac{1}{z_i}\right)=1+\frac{s_1(X)}{z_i}+\frac{s_2(X)}{z_i^2}+\ldots +\frac{s_n(X)}{z_i^n}$ is the total Segre class of $X$ at $1/z_i$.
\end{itemize}
\end{theorem}

The main result of the present paper is the integral formula for $s>1$, that is, geometric subsets with multipoint support, we need further terminology. 
%Let $R(x_1,\ldots, x_m)$ be a homogeneous symmetric polynomial of degree $d$. For a partition $\{1,\ldots, m\}=\mu_1 \cup \ldots \cup \mu_t$ and $d=d_1+\ldots +d_t$ let $R(x_1,\ldots, x_m)^{[\mu_1,d_1],\ldots, [\mu_t,d_t]}$ denote the sum of those terms in $R$ which are homogeneous of degree $d_i$ in the variables in $\mu_i$. Then
%\[R(x_1,\ldots, x_m)=\sum_{d_1+\ldots +d_t=d}R(x_1,\ldots, x_m)^{[\mu_1,d_1],\ldots, [\mu_t,d_t]}.\]
%and $R(x_1,\ldots, x_m)^{[\mu_1,d_1],\ldots, [\mu_t,d_t]}$ can be written as a sum 
%\begin{equation}\label{division}
%R(x_1,\ldots, x_m)^{[\mu_1,d_1],\ldots, [\mu_t,d_t]}=\sum_{\gamma \in \Gamma(\mu,\bd)}\prod_{i=1}^t R^\g_{[\mu_i,d_i]}
%\end{equation}
%of products of homogeneous symmetric polynomials $R^\g_{[\mu_i,d_i]}$ of degree $d_i$ in the variables $\{x_j:j\in \mu_i\}$. The sum is over some finite indexing set $\G(\mu,d)$ and this decomposition is not unique. 
Let $\Hilb^{A_1,\ldots, A_s}(X) \subset \Hilb^{k}(X)$ be the geometric subset corresponding to the algebras $A_1,\ldots, A_s$ with $\dim A_i=k_i$ such that $k_1+\ldots +k_s=k$. For a subset $S \subset \{1,\ldots, s\}$ let $\bA_S=(A_i :i\in S)$ be the multiset and $A_S=\sum_{j\in S}A_j$ the sum algebra of dimension $k_S=\sum_{j\in S} k_j$. Let $\epd{Q(A_S) \subset \mathrm{Alg}_{\bd_S}(N_{S\bullet})}$ denote the equivariant dual in the Kazarian model, where $N_{S\bullet}$ is a natural filtration on $\mathfrak{m} \cap A_S$.

Let $\Pi(s)$ denote the set of partitions of $\{1,\ldots, s\}$ into nonempty subsets; an element $\a=\{\a_1,\ldots, \a_t\} \in \Pi(s)$ consists of subsets $\a_i \subset \{1,\ldots, s\}$ such that $\a_i\cap \a_j=\emptyset$ for $1\le i<j\le t$ and $\{1,\ldots, s\}=\cup_{i=1}^t \a_i$. Note that $\dim(\Hilb^{A_S}(X))=\sum_{j\in S}\dim(\Hilb^{A_j}(X))-|S|+1$ and therefore for a partition $(\a_1,\ldots, \a_t)\in \Pi(s)$ we have 
\[\sum_{i=1}^t \dim (\Hilb^{A_{\a_i}}(X))+|\a_i|-1=\dim \Hilb^{A_1,\ldots, A_s}(X).\] 
 We introduce a set of variables 
\[\bz^{\a_i}=\{z^{\a_i}_1,\ldots, z^{\a_i}_{k_{\a_i}-1}\}\]
for each element of the partition.

\begin{theorem}[\textbf{Tautological integrals on geometric subsets}]\label{main2} Let $X$ be a smooth projective variety of dimension $n$ and let $F$ be a rank $r$ bundle on $X$. Let $\Hilb^{A_1,\ldots, A_s}(X) \subset \Hilb^k(X)$ be the geometric subset corresponding to the algebras $A_1,\ldots, A_s$ with $\dim A_i=k_i$ such that $k_1+\ldots +k_s=k$. Let $\Phi(F^{[k]})$ be a Chern polynomial of degree $\dim \Hilb^{A_1,\ldots, A_s}(X)$ in the Chern classes of the tautological bundle $F^{[k]}$. Then 
\[\int_{\Hilb^{A_1,\ldots, A_s}(X)}\Phi(F^{[k]}) =\sum_{(\a_1,\ldots, \a_t) \in \Pi(s)} \int_{X^t} \res_{\bz^{\a_1}=\infty}\ldots \res_{\bz^{\a_t}=\infty} \mathcal{R}^\a(\theta_i, \bz) d\bz^{\a_1}\ldots d\bz^{\a_t}\]
where $\mathcal{R}^\a(\theta_i, \bz)$ stands for the rational form 
\[\Phi(F(\bz^{\a_1}) \oplus \ldots \oplus F(\bz^{\a_t})) 
\prod_{l=1}^t \left(\frac{\prod_{w^{\a_l}(i)\le w^{\a_l}(j)}(z_i^{\a_l}-z_j^{\a_l})\epd{Q(A_{\a_l}) \subset \mathrm{Alg}_{\bd_{\a_l}}(N_{\a_l\bullet})}d\bz^{\a_l}}
{\prod_{w^{\a_l}(i)+w^{\a_l}(j)\le w^{\a_l}(m) \le 
k_{\a_l}-1}(z_i^{\a_l}+z_j^{\a_l}-z_l^{\a_l})\epd{\a_l}}\prod_{i=1}^{k_{\a_l}-1} 
s_X\left(\frac{1}{z_i^{\a_l}}\right)\right).\] 
where 
\begin{itemize}
\item $\epd{\a_l}=\epd{\Hilb^{A_{\a_l}}(\CC^n), \Hilb^{\bA_{\a_l}}(\CC^n)}$ is the torus-equivariant dual of the punctual geometric subset corresponding to $\a_l$. This is is a homogeneous polynomial of degree equal to the codimension of $\Hilb^{\a_l}(X)$ in $\Hilb^{\bA_{\a_l}}(X)$ in the residue variables $z^{\a_l}_1,\ldots, z^{\a_l}_{k_{\a_|}-1}$. For the definition of this dual in case of singular $\Hilb^{\bA_{\a_l}}(\CC^n)$ see \S \ref{subsec:normalbundle}. If $k_1=k_2=\ldots =k_s$ (we call these balanced geometric subsets) then $\Hilb^{A_S}(X)$ is a local complete intersection in $\Hilb^{\bA_S}(X)$ given by the generalised Haiman bundle $B_S$ for any $S\subset \{1,\ldots , s\}$, and  
\[\epd{\Hilb^{A_{\a_l}}(X), \Hilb^{\bA_{\a_l}}(X)}=\Euler(B_{\a_l})=z_1\ldots z_{|\a_l|-1}\]
\item $F(z)$ stands for the bundle $F$ tensored by the line $\mathbb{C}_z$, which is the representation of a torus $T$ with weight $z$. Hence its Chern roots are $z+\theta_1,\ldots ,z+\theta_r$. Moreover 
	\[F(\bz^{\a_l})=F \oplus F(z^{\a_l}_1) \oplus \ldots \oplus F(z^{\a_l}_{k_{\a_l}-1})\]
	has rank $rk_{\a_l}$, and 
	\begin{equation*}\label{residueeuler} \Phi(F(\bz^{\a_1}) \oplus \ldots \oplus F(\bz^{\a_t})) =  \Phi(\theta_i^1,z^{\a_1}_1+\theta_i,\ldots z^{\a_1}_{k_{\a_1}-1}+\theta_i, \ldots, \theta_i^s,z^{\a_t}_1+\theta_i,\ldots z^{\a_t}_{k_{\a_t}-1}+\theta_i)
	\end{equation*}
	Note that we have $t$ copies of the roots $\theta_1,\ldots , \theta_r$ of $V$, and we think of the $i$th copy $\theta_1^i,\ldots , \theta_r^i$ as the Chern roots of $V=V^i$ sitting over the $i$th copy of $X$ in $X^t$.
\end{itemize}
\end{theorem}

%The equivariant class $\mathrm{ePD}(P_S \subset P(A_i:i\in S)$ is a homogeneous polynomial of degree $\mathrm{codim}(P_S \subset P(A_i:i\in S)$ in the residue variables $z_1,\ldots, z_{|S|-1}$; this is the equivariant Poincar\'e dual of $P_S$ sitting in $P(A_i:i\in S)$ under a residual $GL(|S|-1)$-action. For the definition and explanation see $\S \ref{}$. 
%\subsection{Xotivation} 

Let us make a few comments on the features of the residue formula.
First, the iterated residue gives a degree $n$ univdersal symmetric polynomial in Chern roots of $F$ and Segre classes of $X$ reproving the main result of \cite{rennemo}. Our formula in fact shows that the dependence on Chern classes of $X$ can be expressed via the Segre classes of $X$. 
The iterated residue is some linear combination of the coefficients of (the expansion of) $\mathcal{R}_k$ multiplied by Segre classes of $X$. By increasing the dimension, the iterated residue involves new terms of the expansion of $\mathcal{R}^\a$, and we can think of $\mathcal{R}^\a$ as a universal rational expression encoding the integrals for fixed $k$ but varying $n$.

The geometric component $\GHilb^k(X)$ is a geometric subset, and Theorem \ref{main2} gives back the main result of \cite{berczitau2} in this special case. Indeed, 
$\GHilb^k(X)=\Hilb^{A_0,\ldots, A_0}(X)$ where $A_0=\CC[x_1,\ldots, x_n]/\frakm \simeq \CC$ is the 1-dimensional algebra, hence $k_1=\ldots =k_s=1$ and $s=k$ in this case. 
The sum of $A_0$ algebras is $A_S=A_0+\ldots A_0=A_{|S|}=\CC[x]/x^{|S|}$ only depends on the size of $S$, and 
this monomial algebra corresponding to the tableau $\ytableausetup{smalltableaux,{aligntableaux=bottom}} \underbrace{\ydiagram{4}}_{|S|}$. The dimension is $P(A_{|S|})=n+(n-1)(|S|-1)$. 
The nilpotent algebra $N_t=\frakm/x^t$ has a natural filtration with dimension vector $\bd=(1,\ldots, 1)$ of length $t-1$ formed by the powers of $N_t$. The weights are given as $w(i)=i$ for this filtration.  
This is a balanced geometric subset and therefore $\epd{P_S \subset P(A_i:i\in S)}=z_1 \cdots z_{|S|-1}$. Hence Theorem \ref{main2} translates to the main theorem of \cite{berczitau2}:

\begin{corollary}[\textbf{Integrals over geometric Hilbert schemes} \cite{berczitau2}]
Let $X$ be a smooth projective variety and $\GHilb^k(X) \subset \Hilb^k(X)$ denote the main component of the Hilbert scheme which is the closure of the open set consisting of $k$ different points on $X$. 
Let $F$ be a rank $r$ vector bundle over $X$ with Chern roots $\theta_1, \ldots, \theta_r$. Let $\Phi$ be a Chern polynomial of degree $nk$. Then  
\[\int_{\GHilb^{k}(X)}\Phi(F^{[k]}) =\sum_{(\a_1,\ldots, \a_s) \in \Pi(k)} \int_{X^s} \res_{\bz^{\a_1}=\infty}\ldots \res_{\bz^{\a_s}=\infty} \mathcal{R}^\a(\theta_i, \bz) d\bz^{\a_s}\ldots d\bz^{\a_1}\]
where $\mathcal{R}^\a(\theta_i, \bz)$ stands for the rational form 
\[\Phi(F(\bz^{\a_1}) \oplus \ldots \oplus F(\bz^{\a_s})) 
\prod_{l=1}^s \left(\frac{(-1)^{|\a_l|-1} \prod_{1\le i<j \le |\a_l|-1}(z_i^{\a_l}-z_j^{\a_l})Q_{|\a_l|-1}d\bz^{\a_l}}
{\prod_{i+j\le l\le 
|\a_l|-1}(z_i^{\a_l}+z_j^{\a_l}-z_l^{\a_l})(z_1^{\a_l}\ldots z_{|\a_l|-1}^{\a_l})^{n+1}}\prod_{i=1}^{|\a_l|-1} 
s_X\left(\frac{1}{z_i^{\a_l}}\right)\right).\]
\end{corollary}

Tautological integrals have important applications in enumerative geometry: in \cite{berczitau4} we will focus on three of these. We will derive new residue formulas for counting singular hypersurfaces in sufficiently ample linear systems with given set of singularity classes. This formula is higher dimensional generalisation of curve-counting formulas on surfaces, and in particular, the count of nodal curves. The second main application is higher dimensional Nakajima calculus. Finally, we will introduce generating functions for tautological integrals, and state a general structural theorem about them. In particular, we revisit  Segre and Verlinde integrals. 
These implementations of our general formula provide enumerative evidence that this formula is indeed correct. 

Although we discuss the main applications in \cite{berczitau4}, in the final section of this paper we provide toy examples, and we revisit G\"ottsche's nodal curve counting formula on surfaces \cite{gottsche} and generalise it to higher dimensions. 

\textbf{Acknowledgments}. The author gratefully acknowledges useful discussions 
with Andr\'as Szenes. This research was supported by Aarhus University Research Foundation grant AUFF 29289.

\section{Smoothable and curvilinear components of Hilbert schemes}\label{sec:curvilinear}
Let $X$ be a smooth projective variety of dimension $n$ and let  
\[\Hilb^k(X)=\{\xi \subset X:\dim(\xi)=0,\mathrm{length}(\xi)=\dim H^0(\xi,\calo_\xi)=k\}\]
denote the Hilbert scheme of $k$ points on $X$ parametrizing all length-$k$ subschemes of $X$. For $p\in X$ let 
\[\Hilb^k_p(X)=\{ \xi \in \Hilb^k(X): \mathrm{supp}(\xi)=p\}\]
denote the punctual Hilbert scheme consisting of subschemes supported at $p$. If $\rho: \Hilb^k(X) \to S^kX$, $\xi \mapsto \sum_{p\in X}\mathrm{length}(\calo_{\xi,p})p$ denotes the Hilbert-Chow morphism then $\Hilb^k(X)_p=\rho^{-1}(kp)$.

Despite intense study in the last 40 years, not much is known about the geometry and topology of $\Hilb^k(X)$. There is a distinguished component of $\Hilb^k(X)$ which we call the main or geometric component and denote by $\GHilb^k(X)$. It is the closure of the open set consisting of subschemes supported at $k$ different points on $X$. Its dimension is $kn$ and \cite{ekedahl} constructs the good component as a certain blow-up along an ideal of $\Sym^n X$--this is the generalisation of the classical result of Haiman \cite{haiman} on surfaces. 

Iarrobino \cite{iarrobino} found new components for $n=3$ and $k\ge 8$. These Iarrobino-type components are supported at one point, and their dimensions are typically larger $nk$. They consist of ideals 
\[\mathrm{Iarr}^{[k]}_p=\{I:\mathfrak{m}^s \subseteq I \subseteq \mathfrak{m}^{s+1},\dim_\CC \mathfrak{m}/I=k-1\}\]
where $\mathfrak{m} \subset \calo_p(X)$ is the maximal ideal and $s$ is determined by $\dim (\mathfrak{m}/\mathfrak{m}^s) <k <\dim (\mathfrak{m}/\mathfrak{m}^{s+1})$. In fact, a subspace of $V=\mathfrak{m}^s/\mathfrak{m}^{s+1}$ of dimension $k=\dim (\mathfrak{m}/\mathfrak{m}^s)-1$ determines the ideal uniquely and 
\[\mathrm{Iarr}^{[k]}_p=\grass_{k-\dim (\mathfrak{m}/\mathfrak{m}^s)-1}(V).\]

For surfaces ($n=2$) the punctual geometric set $\GHilb^k_p(X)=\Hilb_p^{k}(X) \cap \GHilb^k(X)$
is irreducible and equal to the punctual Hilbert scheme $\Hilb_p^{k}(X)$, however, for $n\ge 3$ $\GHilb^k_p(X)$ is typically not irreducible; its components are called smoothable components of $\Hilb_p^{k}(X)$. Again, description of the smoothable components is a hard problem and is unknown in general. The Iarrobino-type punctual components are clearly not smoothable: their dimension can be higher than $kn$. It was an open question until recently whether there exist divisorial components of $\GHilb^k_p(X)$, that is, components $C_p \subset R_p^{[k]}$ whose sweep $\cup_{p\in X} C_p$ over $X$ is a divisor of the good component. The dimension of this $C_p$ is $k(n-1)$.  Erman and Velasco in \cite{erman} give affirmative answer to this question when $d\ge 4, k\ge 11$. In the other direction, Satriano and Staal \cite{satriano} found small punctual components, whose dimension is smaller than the curvilinear component. 

A subscheme $\xi \in \Hilb_p^k(X)$ is called curvilinear if $\xi$ is contained in some smooth curve $\calc \subset X$. Equivalently, $\xi$ is curvilinear if $\calo_\xi$ is isomorphic  to the $\CC$-algebra $\CC[z]/z^{k}$.
The punctual curvilinear locus at $p\in X$ is the set of curvilinear subschemes supported at $p$: 
\begin{equation}\nonumber 
\mathrm{Curv}^k_p(X)=\{\xi \in \Hilb^k_p(X): \xi \subset \mathcal{C}_p \text{ for some smooth curve } \mathcal{C} \subset X\}
=\{\xi:\calo_\xi \simeq \CC[z]/z^{k}\}
\end{equation}

Note that any curvilinear subscheme contains only one subscheme for any given smaller length and any small deformation of a curvilinear subscheme is again locally curvilinear.

For surfaces ($n=2$) $\mathrm{Curv}^k_p(X)$ is an irreducible quasi-projective variety of dimension $k-1$ which is an open dense subset in $\Hilb^k_p(X)$ and therefore its closure is the full punctual Hilbert scheme at $p$. However, if $n\ge 3$, the punctual Hilbert scheme $\Hilb^k_p(X)$ is not necessarily irreducible or reduced, and it has smoothable and non-smoothable components. The closure of the curvilinear locus is one of its smoothable irreducible components as the following lemma shows. 

\begin{lemma} $\overline{\mathrm{Curv}^k_p(X)}$ is an irreducible component of the punctual Hilbert scheme $\Hilb^k_p(X)$ of dimension $(n-1)(k-1)$. 
\end{lemma}

\begin{proof}
Note that $\xi \in \Hilb^{k}_0(\CC^n)$ is not curvilinear if and only if $\calo_\xi$ does not contain elements of degree $k-1$, that is, after fixing some local coordinates $x_1,\ldots, x_n$ of $\CC^n$ at the origin we have
\[\calo_\xi \simeq \CC[x_1,\ldots, x_n]/I \text{ for some } I\supseteq (x_1,\ldots, x_n)^{k-1}.\]
This is a closed condition and therefore curvilinear subschemes can't be approximated by non-curvilinear subschemes in $\Hilb^{[k]}_0(\CC^n)$. The dimension will follow from the description of it as a non-reductive quotient in the next subsection.  
\end{proof}

\begin{definition} We call $\CHilb^k_p(X)=\overline{\mathrm{Curv}^k_p(X)}$ the punctual curvilinear component supported at $p$ and 
\[\CHilb^k(X)=\cup_{p\in X} \CHilb^k_p(X)\]
is the curvilinear Hilbert scheme which is the fibrewise closure of the curvilinear locus. 
\end{definition}

The natural torus action on $\CC^n$ induces an action on these Hilbert schemes and the curvilinear component $\CHilb^k_0(\CC^n)$ is distinguished among the smoothable component of $\Hilb^k_0(\CC^n)$ because it conjecturally contains all torus fixed points of $\Hilb^k(\CC^n)$. 

\begin{conjecture}\label{curvfixedpoint}
All torus fixed points on $\Hilb^k(\CC^n)$ sit in the punctual curvilinear component $\CHilb^k_0(\CC^n)$.
\end{conjecture}

%In fact less would be enough for us. A trick of our computations is that we increase first the dimension $n$ above the number of points $k$. So we need only the following, which is plausible as all examples of irreducibility happens for $n<k$. 

%\begin{conjecture}
%If $n\gg k$ then all torus fixed points on $\Hilb^k(\CC^n)$ sit in the punctual curvilinear component $\CHilb^k_0(\CC^n)$. 
%\end{conjecture}

\begin{remark}\label{remark:embed}
Fix coordinates $x_1,\ldots, x_n$ on $\CC^n$. Recall that the defining ideal $I_\xi$ of any subscheme $\xi \in \Hilb^{k+1}_0(\CC^n)$ is a codimension $k$ subspace in the maximal ideal $\mathfrak{m}=(x_1,\ldots, x_n)$. The dual of this is a $k$-dimensional subspace $S_\xi$ in $\mathfrak{m}^*\simeq \symdot$ giving us a natural embedding $\varphi: \Hilb^{k+1}(X)_p \hookrightarrow \grass_k(\symdot)$. In what follows, we give an explicit  parametrization of this embedding using an algebraic model coming from global singularity theory.  
\end{remark}

\section{Geometric subsets and their curvilinear components}\label{sec:geomsubsets}

Following Rennemo \cite{rennemo} we define punctual geometric subsets to be the constructible subsets of the punctual Hilbert scheme containing all $0$-dimensional schemes of given isomorphism types.
\begin{definition} A type is a constructible subset $Q\subseteq \Hilb^k_0(\CC^n)$ which is the union of isomorphism classes of subschemes, that is, if $\xi \in Q$ and $\xi' \in \Hilb^k_0(\CC^n)$ are isomorphic schemes then $\xi' \in Q$. 
\end{definition}

\begin{definition} For an $s$-tuple $\mathbf{Q}=(Q_1,\ldots, Q_s)$ of types such that $Q_i\subseteq \Hilb^{k_i}_0(\CC^n)$ and $k=\sum k_i$ define the subset of type $\mathbf{Q}$ as 
\[\mathcal{H}^\mathbf{Q}(X)=\{\xi \in \Hilb^k(X):\xi=\xi_1 \sqcup  \ldots \sqcup \xi_s \text{ where } \xi_i\in X^{[k_i]}_{p_i}\cap Q_i \text{  for distinct } p_1,\ldots, p_s\}\subseteq \Hilb^k(X).\]
A subset of $\Hilb^k(X)$ is geometric if it can be expressed as finite union, intersection and complement of sets of the form $\mathcal{H}^\mathbf{Q}$. We let 
\[\Hilb^{\mathbf{Q}}(X)=\overline{\mathcal{H}^\mathbf{Q}(X)}\]
denote the closure in $\Hilb^k(X)$.
\end{definition}

A straightforward way to produce types is by taking a complex algebra $A$ of complex dimension $k$ on $n$ generators and define the corresponding punctual set
\[Q_A=\{\xi \in \Hilb^k_0(\CC^n): \calo_\xi \simeq A\}.\]  
In this paper we will focus on monomial types and geometric subsets where the algebra $A$ is defined by a monomial ideal. 

\begin{definition} Let $A_1,\ldots, A_s$ be finite dimensional algebras on $n$ generators. The locus of type $\bA=(A_1,\ldots, A_s)$ is 
\[\mathcal{H}^{A_1,\ldots, A_s}(X)=\{\xi\in \Hilb^k(X)|\xi=\xi_1\sqcup \ldots \sqcup \xi_s, \calo_{\xi_i} \simeq A_i\}\]
We call its closure 
\[\Hilb^{A_1,\ldots, A_s}(X)=\overline{\mathcal{H}^{A_1,\ldots, A_s}(X)}\]
the $\bA$-Hilbert scheme. If for $1\le i\le s$ $A_i=\CC[x_1,\ldots, x_n]/I_{\lambda_i}$ for some monomial ideal $I_{\lambda_i}$ indexed by the $n$-diimensional Young tableau $\lambda_i$, then we call this a monomial geometric subset, and also use the notation $\Hilb^{\l_1,\ldots, \l_s}(X)$. 
\end{definition}

We follow the French convention in drawing Young-diagrams, e.g in dimension $2$ for $\lambda=(\lambda_1 \ge \ldots \ge \l_n)$ we stack each row on top of the previous one,  e.g.  $\mu=(2,1)$ corresponds to $\ytableausetup{smalltableaux} \ydiagram{1,2}$.

\begin{example}\label{example:monsubsets} \begin{enumerate} 
\item The geometric component $\GHilb^k(X)$ is itself a monomial geometric subset; indeed, it corresponds to the trivial partitions $\l_1=\ldots =\l_k=(1)$ and the corresponding monomial ideal is the maximal ideal $(x_1,\ldots, x_n)$:
\[\GHilb^k(X)=\GHilb^{(1),\ldots, (1)}(X).\]
\item The Xorin algebra $A_k=\CC[z]/z^{k}$ is monomial corresponding to $\l=(k)=\underbrace{\ydiagram{4}}_k$. The punctual set $Q_A=\Curv^k_0(\CC^n)$ is the curvilinear locus in $\Hilb^k_0(\CC^n)$ and the corresponding closed set $\Hilb^{A_k}(X)=\CHilb^k(X)$ is the curvilinear component. 
\item The G\"ottsche geometric subset corresponds to the square of the maximal ideal $\mathfrak{m}^2 \subset \CC[x_1,\ldots, x_n]$ at each point of the support.  On a surface $S$ this means that $\lambda_1=\ldots =\lambda_s=\ytableausetup{smalltableaux} \ydiagram{1,2}$ and $A=\CC[x,y]/\mathfrak{m}^2=\CC[x,y]/(x^2,xy,y^2)$. Then $\Hilb^{s \cdot A}(S) \subset \Hilb^{3s}(S)$ is a $2s$-dimensional closed variety and the count of $s$-nodal curves in the $5s$-ample linear system $\PP(|L|)$ is 
$\int_{P(s\cdot \lambda)} c_{2s}(L^{[3s]})$, see \cite{gottsche}.
\end{enumerate}
\end{example} 
% In \S\ref{sec:resolutions} we construct an embedding $\overline{CX}^{[k+1]}_p \subset \grass_k(\sym^{\le k}T_pX )$ into a smooth Grassmannian and for $k\le n$ we construct a partial resolution $\widetilde{CX}^{[k]}_p \to \overline{CX}^{[k]}_p$. In \S\ref{sec:locsnowman} we develop the iterated residue formula of Theorem %\ref{main} using equivariant localisation to compute $\int_{\widetilde{CX}_p^{[k]}}P(c_i(F^{[n]}))$ which is according to the remark above equal to $\int_{\overline{CX}_p^{[k]}}P(c_i(F^{[n]}))$.

\subsection{Curvilinear components of geometric subsets}
 
Let $\Hilb^{\l_1,\ldots, \l_s}(X)$ be the monomial geometric subset defined by the partitions $\l_1,\ldots, \l_s$. In this section we study the the punctual geometric subset 
 \[\Hilb^{\l_1,\ldots, \l_s}_0(X)=\Hilb^{\l_1,\ldots, \l_s}(X) \cap \Hilb^{k}_0(X)\]
which consists of subschemes supported at some point of $X$. When we fix the support to be a specific point $p \in X$, we write 
\[\Hilb^{\l_1,\ldots, \l_s}_p(X) \simeq \Hilb^{\l_1,\ldots, \l_s}_{\{0\}}(\CC^n)\]
The components of $\Hilb^{\l_1,\ldots, \l_s}_{\{0\}}(\CC^n)$ are often called smoothable components. Their description is far beyond of our current knowledge on Hilbert schemes. 

The goal of this section is to define a distinguished component of $\Hilb^{\l_1,\ldots, \l_s}(X)$ which contains all subschemes which arise when the support of a generic point collides along a smooth curve. It is a multisingularity version of the curvilinear component $\CHilb^{k}(X)$ and hence we will call this the curvilinear geometric subset. 

Let 
\[\mathcal{G}(\l_1,\ldots, \l_s) \subset  \Hilb^{\l_1,\ldots, \l_s}(X) \times \Hilb^{s}(X)\]
denote the closure of the graph of the support map 
\[\mathrm{supp}: \mathcal{H}^{\l_1,\ldots, \l_s}(X) \to \Hilb^{s}(X),  \xi \mapsto \mathrm{supp}(\xi).\]
We denote by $\pi_{\l_1.\ldots, \l_s}$ and $\pi_{[s]}$ the projections of $\mathcal{G}(\l_1,\ldots, \l_s)$ to the first and second factor. Note that 
\begin{itemize}
\item The image of $\pi_{[s]}$ is the main component $\GHilb^s(X)$, and the fibre $\pi_{[s]}^{-1}(\eta)$ over a point $\eta \in \GHilb^s(X)$ consists of all subschemes $\xi \in \Hilb^{\l_1,\ldots, \l_s}(X)$ which arises as a limit $\xi=\lim_{i\to \infty} \xi_i$ for some sequence $\xi_i \in \mathcal{H}^{\l_1,\ldots, \l_s}(X)$ satisfying $\lim_{i\to \infty}\mathrm{supp}(\xi_i)=\eta$. 
\item $\pi_{\l_1.\ldots, \l_s}$ is surjective and the fibre $\pi_{\l_1.\ldots, \l_s}^{-1}(\xi)$ over a point $\xi \in R_p$ consists of all subschemes $\eta \in \Hilb^{s}(X)$ which arises as a limit $\eta=\lim_{i\to \infty} \mathrm{supp}(\xi_i)$ for some sequence $\xi_i \in \mathcal{H}^{\l_1,\ldots, \l_s}(X)$ whose limit is $\xi$.  
\end{itemize}

The following conjecture is in fact a corollary of Conjecture \ref{curvfixedpoint} and the fact that $\pi_{[s]}$ is torus-equivariant and therefore the image of a fixed-point is fixed in $\CHilb^{[k]}_0(\CC^n)$. 

\begin{conjecture}
The torus fixed points of $\Hilb^{\l_1,\ldots, \l_s}(\CC^n) \subset \Hilb^{[k]}(\CC^n)$ sit in the set 
\[\pi_{[s]}^{-1}(\CHilb_0^{[s]}(\CC^n)).\]  
\end{conjecture}

Our next goal is to %show that (i) $\pi_{[s]}^{-1}(\CHilb_0^{[s]}(\CC^n))$ is irreducible and (ii)
identify a distinguished irreducible component of $\pi_{[s]}^{-1}(\CHilb^{[s]}(X)$ as a mono-geometric subset $\Hilb^{\l}(X)$ for some partition $\l$ cooked up from $\l_1,\ldots, \l_s$ such that $|\l|=|\l_1|+\ldots +|\l_s|$.  

\begin{definition}[Well-oriented partitions]\label{partialorder}
Fix a basis of $\CC^n$. We associate to a partition $\l$ consisting of $k$ boxes in $\CC^n$ a sequence $r(\l)=(r_1,\ldots, r_n)$ of integers defined as 
\[r_j(\l)=\max_{(i_1,\ldots, i_n) \in \l} i_j =\text{Length of $\l$ in the } j \text{th coordinate direction}\]
We call $\l$ well-oriented if $r_1\ge \ldots \ge r_n$ holds and in this case we call $r_1$ the length of $\l$, denoted by $l(\l)$. Any partition can be moved into a well-oriented position using reflections.  
\end{definition}
%The following well-known lemma will play a crucial role in the collision punctual geometric subsets of Hilbert schemes but also provides a natural phrasing of the residue vanishing theorem below.  
%\begin{lemma}\label{collisionlemma} If $\xi$ is a deformation of $\xi'$, that is  
%\[\xi \in \overline{\{\eta \in \Hilb^{[k]}_0(\CC^n): \calo_\eta \simeq \calo_{\xi'}\}} \setminus \{\eta \in \Hilb^{[k]}_0(\CC^n): \calo_\eta \simeq \calo_{\xi'}\}.\]
%then $\xi \succ \xi'$.  
%\end{lemma}

%In particular, this lemma implies the following for nodal curves.

%\begin{proof}
%The $\supseteq$ part is clear: when $\delta$ monomial ideals come together along a line the limit ideal is monomial corresponding to the partition $(2\delta, \delta)$, see this schematic picture:
%\[\ytableausetup{smalltableaux}
%\ydiagram{1, 2} \leftarrow \ydiagram{1,2} \leftarrow \ydiagram{1,2} = \ydiagram{3,6}\]

%When they come together from different directions, the Wall-index of the limit scheme will be higher, and therefore this limit scheme sits in $P(\underbrace{\ydiagram{3, 6}}_{2\delta})$.
%This is a schematic picture when 2 points hits a third one at the origin along the axes:  
%\[\ytableausetup{smalltableaux,{aligntableaux=bottom}}
%\xymatrix{
%\ydiagram{1, 2} \ar[d] &  & & \\
%\ydiagram{1,2}   & \ydiagram{1,2} \ar[l] & = &  \ydiagram{1,2,2,4}}\]
%\end{proof}

\begin{definition}[Sum of partitions along the first axis]\label{addition} Let $\l$ be $n$-dimensional partition. For fixed $(i_2,\ldots, i_n)$ let 
\[\l^{(i_2,\ldots, i_n)}=\max\{m: (m,i_2,\ldots, i_n) \in \l\}\]
denote its length along the slice $(i_2,\ldots, i_n)$. We define the sum of the $n$-dimensional partitions $\l_1,\ldots, \l_s$ along the first axis by setting 
\[(\l_1+_1 \ldots +_1 \l_s)^{(i_2,\ldots, i_n)}=\l_1^{(i_2,\ldots, i_n)}+\ldots +\l_s^{(i_2,\ldots, i_n)}.\]
\end{definition}

\begin{definition}[Curvilinear sum of partitions] Let $\l_1,\ldots, \l_s$ be $n$-dimensional partitions. Their curvilinear sum, denoted by simply $\l_1+\ldots +\l_s$ is their sum along the first axis after moving all of them into well-oriented position. This addition is clearly commutative and associative.
\end{definition}

\begin{example} The curvilinear sum of $\l_1=\ldots =\l_s=(1)$ is $(1)+\ldots +(1)=s\cdot (1)=(s)$. Schematically:
\[s \cdot \ydiagram{1} = (s)= \underbrace{\ydiagram{4}}_{s}.\]
The curvilinear sum of the G\"ottsche partitions $\l_1=\ldots =\l_s=(2,1)$ is $(2,1)+\ldots +(2,1)=s\cdot (2,1)=(2s,s)$:
\[s \cdot \ydiagram{1,2} = (2s,s)= \underbrace{\ydiagram{3, 6}}_{2s}.\]
\end{example}

Because of the analogy with curvilinear components, we introduce the following terminology.

\begin{definition}[Curvilinear geometric subsets]\label{def:curvgeometricsubset}
We call the component $\Hilb^{\l_1+\ldots +\l_s}(X)$ the curvilinear component of the geometric subset $\Hilb^{\l_1,\ldots, \l_s}(X)$. With fixed support $p\in X$ we call $\Hilb^{\l_1+\ldots +\l_s}_p(X)$ the curvilinear component supported at $p$.
\end{definition}

An example where we can prove this conjecture is the case when the types $Q_{\l_i}=\mathcal{H}^{\l_i}(\CC^n)$ are zero dimensional for $i=1,\ldots, s$. 
Lemma \ref{curvfixedpoint} gives a natural analog of Conjecture \ref{curvfixedpoint} for regular geometric subsets. Namely, for regular geometric subsets all fixed points sit in $\Hilb^{\lambda_1+ \ldots + \lambda_s}(\CC^n)$.

%\begin{proof} For (1) we can use the results of Lederer \cite{lederer} and the Groebner filtration of the punctual Hilbert scheme.  For (2) note that $\l_i$ is some power of the maximal ideal for all $i$ and I need to explain that the way how the support of $\xi \in \Hilb^{\l_1,\ldots, \l_s}(X)$ comes together at one point determines the limit subscheme, but this is is pretty much clear.  
%To see (1) note that $\CHilb_0^{[s]}(\CC^2))=\Hilb_0^{[s]}(\CC^2))$ and therefore $\Hilb_0^{[s]}(\CC^2))$ is irreducible. A point 
%\[\xi \in \pi_{[s]}^{-1}(\Hilb_0^{[s]}(\CC^2)) \setminus \overline{\pi_{[s]}^{-1}(CL\Hilb_0^{[s]}(\CC^2))}=(\Hilb^{\l_1,\ldots, \l_s}(X) \cap \Hilb_0^{[s]}(\CC^2)) \setminus \overline{\pi_{[s]}^{-1}(CL\Hilb_0^{[s]}(\CC^2))}\]
%by definition would be a limit $\xi=\lim_{m\to \infty}\xi_m$ such that 
%\end{proof}

The following characterisation of the curvilinear locus plays a crucial role in our argument. In fact, this is one of the main geometric features of the curvilinear locus why our construction works.
\begin{lemma}\label{characterisation} Let $\calo_p$ denote the local ring at $p\in X$ with maximal ideal $\mathfrak{m}_p \subset \calo_p$. Let $\xi \in \Hilb^{\l_1,\ldots, \l_s}_p(X)$ be a point in the geometric subset such that $\mathrm{supp}(\xi)=p$. Then 
\[\xi \in \calH^{\l_1+\ldots +\l_s}_p(X) \Leftrightarrow \calo_\xi/\mathfrak{m}_p^{l(\l_1)+\ldots +l(\l_s)} \nsubset \calo_\xi\]
In other words, if $I_\xi \subset \mathfrak{m}_p$ denotes the ideal of $\xi$ in some local coordinates at $p$ then 
\[\xi \in \calH^{\l_1+\ldots +\l_s}_p(X) \Leftrightarrow \mathfrak{m}_p^{l(\l_1)+\ldots +l(\a_s)}\nsubset I_\xi,\]
that is, there is a homogeneous degree $l(\l_1)+\ldots +l(\a_s)$ polynomial which is not in $I_{\xi}$. 
\end{lemma} 
\begin{proof}
This automatically follows from the definition. This is why we add the partitions along their longest direction, i.e in well-oriented position.
\end{proof}

\begin{corollary}\label{prop:components} For arbitrary $n$-dimensional partitions $\lambda_1,\ldots, \lambda_s$ of size $k_1,\ldots, k_s$ with $k=k_1+\ldots +k_s$ and $p\in X$ we have
\begin{enumerate} 
\item  The punctual geometric set $\Hilb^{\lambda_1+\ldots +\lambda_s}_{p}{X}$ is an irreducible component of $\Hilb^{\l_1,\ldots, \l_s}_p(X)$. 
\item The codimension of $\Hilb^{\l_1+\ldots +\l_k}(X)$ in $\Hilb^{\l_1,\ldots, \l_k}(X)$ is equal to $s-1$, the size of the support minus $1$. 
%\item  $\Hilb^{\lambda_1+\ldots +\lambda_s}_{\{0\}{\CC^n}$ is a maximal dimensional irreducible component of $\pi_{[s]}^{-1}(\CHilb_0^{[s]}(\CC^n))$.
\end{enumerate}
\end{corollary}

\begin{proof} For (1) we will see in \S \ref{subsec:testjetmodel} that $\Hilb^{\l_1+\ldots +\l_s}_p(X)$ is the closure of a $\GL(N)$-orbit for some $N$ and therefore irreducible. On the other hand, by Lemma \ref{characterisation}, the set of non-curvilinear subschemes in $\Hilb^{\l_1,\ldots ,\l_s}_p(X)$ is closed.
We will prove the codimension part (2) in \S\ref{subsec:dim} after introducing a suitable algebraic model for monomial geometric subsets. This model will provide a combinatorial characterisation of the dimension of $\Hilb^{\l}(\CC^n)$ for any partition $\l$. 
\end{proof}

We finish this section with a structural theorem on the curvilinear locus.

\begin{proposition} The generalised curvilinear locus 
\[\{\xi=\xi_1 \sqcup \ldots \sqcup \xi_m \in \Hilb^{\l_1,\ldots, \l_s}(X) :\xi_i \text{ is curvilinear for } 1\le i \le m\} \]
which consist of subschemes which is curvilinear at each point, is locally closed in $\Hilb^{\l_1,\ldots, \l_s}(X)$. 
\end{proposition}

\begin{proof} Let $\xi_i \in \Hilb_{p_i}^{k}(\CC^n)$ and $\eta_i \in \Hilb_{q_i}^{m}(\CC^n)$ ($i=1,\ldots$) be two sequences supported at $p_i,q_i \in \CC^n$ such that  
\begin{enumerate}
\item $\lim_{i\to \infty} p_i=\lim_{i\to \infty} q_i=0 \in \CC^n$ and
\item $\mathfrak{m}_{p_i}^{N}\subset I_{\xi_i}$ for $i=1,\ldots $, $\mathfrak{m}_{q_i}^{X}\subset I_{\eta_i}$ for $i=1,\ldots $ for some  
fixed $N,X$. 
\end{enumerate}
Then $\mathfrak{m}_a^N \cdot \mathfrak{m}_b^X \subset I_{\xi_i} I_{\eta_i} \subset I_{\xi_i \sqcup \eta_i}$
holds for any $i$ and therefore the same is true for the limit:
\begin{equation}\label{limitscheme}
\mathfrak{m}_0^{N+X} \subset I_{\lim_{i\to \infty} \xi_i \sqcup \eta_i}.
\end{equation}
In particular, assume that $\xi_i \in \Hilb^{\l_1,\ldots, \l_s}(X)$ is a sequence whose limit is curvilinear:
\[\xi:=\lim_{m\to \infty} \xi \in \calH^{\l_1+\ldots +\l_s}(X).\]
Let $\xi=\xi_i^1 \sqcup \ldots \sqcup \xi_i^{s_i}$ be supported at $s_i$ points. We claim that all (but finite) $\xi_i^j$ is curvilinear. Indeed, if say, $\xi_i^1$ was not curvilinear for all $i$ then $\mathfrak{m}_{p_i}^{l(\xi_i^1)-1} \subset I_{\xi_i^1}$ would hold where $l(\xi_i^1)$ is the sum of the lengths of those $\l_j$'s whose points correspond to $\xi_i^1$. Then by \eqref{limitscheme}
\[\mathfrak{m}_0^{l(\xi_i^1)+ \ldots +l(\xi_i^{s_i})-1} = \mathfrak{m}_0^{\l_1+\ldots +\l_s-1} \subset I_{\xi_i}\]
would hold for all $i$, and therefore  Therefore the same is true for the limit point:
\[\mathfrak{m}_0^{\l_1+\ldots +\l_s-1} \subset I_{\xi}\]
By Lemma \ref{characterisation} $\xi$ would not be curvilinear. This proves the following key property of curvilinear subschemes.
\end{proof}

Finally, we fix some terminology regarding the relation of $\Hilb^{\l_1+ \ldots +\l_s}(X)$ and $\CHilb^k(X)$. 
\begin{definition}\label{def:curvilinear} We call the geometric subset $\Hilb^{A_1,\ldots, A_s}(X)$ curvilinear, if $\Hilb^{A_i}(X) \subset \CHilb^{\dim(A_i)}(X)$ for $1\le i \le s$. That is, the algebra at each point sits in the curvilinear component.
\end{definition}
By Conjecture \ref{curvfixedpoint} all monomial algebras are curvilinear.  For curvilinear tuples $(A_1,\ldots ,A_s)$ with $k=\dim(A_1)+\ldots +\dim(A_s)$ 
\[\Hilb^{A_1+\ldots +A_s}(X) \subset \CHilb^k(X)\] 
holds. 
\begin{definition}\label{def:regular} We call $(A_1,\ldots, A_s)$ regular if 
\[\Hilb^{A_1+\ldots +A_s}(X) = \CHilb^k(X) \cap \Hilb^{A_1,\ldots, A_s}(X).\] 
\end{definition}

We conjecture that being regular is not restrictive.

\begin{conjecture}
All monomial geometric subsets are regular. 
\end{conjecture}

\begin{remark}
An other natural question about the geometric subset $\Hilb^{A_1,\ldots, A_s}(X)$ is how deformations of the support determine the limit points. Namely, is it true that 
\[\Hilb^{\lambda_1+ \ldots + \lambda_s}(X)=\pi_{[s]}^{-1}(\CHilb^s(X))\]
holds for regular geometric subsets?
\end{remark}

\section{Two models for punctual geometric subsets}
Let $I \subset \mathfrak{m} \subset \CC[x_1,\ldots, x_n]$ be an ideal which defines the nilpotent complex algebra $A=\CC[x_1,\ldots, x_n]/I$ of dimension $k$. 
In this section we develop a model for the punctual geometric locus (we called this a type in the introduction)
\[\calH^{A}(\CC^n)=\{\xi \in \Hilb^k(\CC^n): \calo_\xi \simeq A\}\]
and the correponding punctual geometric subset $\Hilb^A(\CC^n)=\overline{\calH^{A}(\CC^n)}$. For historical accuracy we start in \S \ref{subsec:testjetmodel} with a model developed in \cite{bsz} for monomial geometric subsets based on ideas of Damon \cite{damon} and then we describe the work of Kazarian \cite{kazarian} on the refinement of this model for any nilpotent algebra $A$. For the sake of the vanishing theorem described in the strategy Kazarian's model turns out to be more efficient and we use this model in the rest of the paper. 

In the first approach the main geometric idea is the description of $Q(A)$ as certain moduli of map germs from $\CC^d$ to $\CC^n$ where $d$ is the minimal number of generators is a presentations of $A$. Its closure, $P(A)$ is a projective completion of a quotient by a non-reductive reparametrisation group. This way we construct a parametrisation of the embedding $P(A) \hookrightarrow \grass_k(\mathfrak{m}/\mathfrak{m}^{K+1})$ into the traditional embedding space for the Hilbert scheme, which is the Grassmannian (or flag) variety of subspaces of fixed dimension in the vector space $\mathfrak{m}/\mathfrak{m}^{K+1}$ of jets of functions on $\CC^n$. Equivariant localisation on $P(A)$ can be shifted to the smooth ambient space using equivariant duals and this way we can transform the Atiyah-Bott localisation into a residue and prove the vanishing property directly.

In Kazarian's approach $\grass_k(\mathfrak{m}/\mathfrak{m}^{K+1})$ is replaced by a new embedding space which he calls the nonassociative Hilbert scheme. 

\subsection{The test jet model for $\Hilb^{\lambda}(\CC^n)$}\label{subsec:testjetmodel}

\begin{definition}[$\lambda$-jets]
Let $\lambda$ be a $d$-dimensional partition whose boxes are numbered with $d$-tuples $(a_1,\ldots, a_d)\in \ZZ_{\ge 0}^d$ indicating the integer coordinates of the corresponding box. E.g for $\lambda=(\delta,2\delta)$ we have
\[\ytableausetup{centertableaux,boxsize=normal}
\mathcal{B}=\begin{ytableau}
 {\scriptscriptstyle 0,1} & {\scriptscriptstyle 1,1}  &  & {\scriptscriptstyle \delta-1,1} \\
 {\scriptscriptstyle 0,0}        & {\scriptscriptstyle 1,0} &  & {\scriptscriptstyle \delta,0} & & & & {\scriptscriptstyle 2\delta-1,0} 
\end{ytableau}\]
If $u,v$ are positive integers let $J(u,v)$ denote the (infinite dimensional ) vector space of holomorphic map germs $(\CC^u,0)\to (\CC^v,0)$ at the origin. Let $J_\lambda(u,v)$ denote the vector space of $\lambda$-jets of holomorphic maps $(\CC^u,0)\to (\CC^v, 0)$ at the origin, that is, the set of equivalence classes of maps $f:(\CC^u,0)\to (\CC^v,0)$ with nonzero first partial derivatives, where $f\sim g$ if and only if $f^{(\bj)}(0)=g^{(\bj)}(0)$ for all $\bj\in \lambda$.
\end{definition}

One can compose map-jets via substitution:
\[J(u,v) \times J(v,w) \to J(u,w).\]
However, the elimination of terms of degree outside $\lambda$ does not commute with this map, that is, the diagram 
\[\xymatrix{J(u,v) \times J(v,w) \ar[r] \ar[d]^{(f,g)\mapsto (f_\l,g_\l)} & J(u,w) \ar[d]^{f \mapsto f_\l} \\
J_\l(u,v) \times J_\l(v,w) \ar[r] & J_\l(u,w)
}\]
does not commute in general: $(f \circ g)_\l \neq (f_\l \circ g_\l)_\l$. %Indeed, take for instance $\l=(1,2)=\ydiagram{2,4}$, $f=,g=$...finish this.

Let $\tilde{J}_(u,v) \subset J(u,v)$ denote the subset whose elements result a commutative diagram
\[\xymatrix{\tilde{J}(u,v) \times J(v,w) \ar[r] \ar[d]^{(f,g)\mapsto (f_\l,g_\l)} & J(u,w) \ar[d]^{f \mapsto f_\l} \\
\tilde{J}_\l(u,v) \times J_\l(v,w) \ar[r] & J_\l(u,w)
}\]
that is, 
\[(f \circ g)_\l = (f_\l \circ g_\l)_\l \text{ for all } f\in \tilde{J}_\l(u,v), g\in J_\l(v,w)\]
This leads to the composition map
\[\tilde{J}_\lambda(u,v) \times J_\lambda(v,w) \to J_\lambda(u,w).\]

Let $\xi \in \Xi(\lambda)$ be a subscheme. Then by definition $\calo_\xi \simeq \CC[x_1,\ldots, x_n]/I_\lambda$ where $I_\lambda$ is the monomial ideal corresponding to $\lambda$ and $\xi$ is contained in the $\lambda$-jet of a nonsingular $\CC^d$-germ at the origin in $\CC^n$: 
\[\xi \subset \mathcal{S}^\lambda_0 \subset X.\]
Let $f:(\CC^d,0)\to (\CC^n,0)$ be a holomorphic germ whose $\lambda$-jet $f_\l$ parametrises $\mathcal{S}_0$ in some local coordinates on $\CC^n$ at $\{0\}$.  
Then $f_{\l}\in J_\lambda(d,n)$ is determined only up to polynomial reparametrisation germs $\phi: (\CC^d,0)\to (\CC^d,0)$. More precisely, let $\tilde{J}_\l(d,d) \subset J_\l(d,d)$ denote the subset as above such that the diagram 
\[\xymatrix{\tilde{J}(d,d) \times J(d,n) \ar[r] \ar[d]^{(f,g)\mapsto (f_\l,g_\l)} & J(d,n) \ar[d]^{f \mapsto f_\l} \\
\tilde{J}_\l(d,d) \times J_\l(d,n) \ar[r] & J_\l(d,n)
}\] 
The subset $\tilde{J}^\reg_\l(d,d) \subset \tilde{J}_\l(d,d)$ consisting of jets with regular linear part form a subgroup acting on $J_\l(d,n)$.
\begin{lemma}\label{plquotient} $\calH^{\l}(\CC^n)$ is equal to the set of regular $\lambda$-jets of germs modulo reparametrisation: 
\[\calH^{\lambda}(\CC^n)=J^\reg_\l(d,n)/\tilde{J}^\reg_\l(d,d)\]

\end{lemma}

%We will explicitely write out this reparametrisation action in 
%let $f_{\xi}(z)=z f'(0)+\frac{z^2}{2!}f''(0)+\ldots +\frac{z^k}{k!}f^{(k)}(0) \in \jetreg 1n$ a $k$-jet of germ at the origin (i.e no constant term) in $\CC^n$ with $f^{(i)}\in \CC^n$ such that $f' \neq 0$ and let  
%$\varphi(z)=\alpha_1z+\alpha_2z^2+\ldots +\alpha_k z^k \in \jetreg 11$ with $\alpha_i\in \CC, \alpha_1\neq 0$.  
 %Then 
 %\[f \circ\varphi(z)
%=(f'(0)\alpha_1)z+(f'(0)\alpha_2+
%\frac{f''(0)}{2!}\alpha_1^2)z^2+\ldots
%+\left(\sum_{i_1+\ldots +i_l=k}
%\frac{f^{(l)}(0)}{l!}\alpha_{i_1}\ldots \alpha_{i_l}\right)z^k=\]
%\begin{equation}\label{jetdiffmatrix}
%=(f'(0),\ldots, f^{(k)}(0)/k!)\cdot 
%\left(
%\begin{array}{ccccc}
%\alpha_1 & \alpha_2   & \alpha_3          & \ldots & \alpha_k \\
%0        & \alpha_1^2 & 2\alpha_1\alpha_2 & \ldots & 2\alpha_1\alpha_{k-1}+\ldots \\
%0        & 0          & \alpha_1^3        & \ldots & 3\alpha_1^2\alpha_ {k-2}+ \ldots \\
%0        & 0          & 0                 & \ldots & \cdot \\
%\cdot    & \cdot   & \cdot    & \ldots & \alpha_1^k
%\end{array}
% \right)
% \end{equation}
%where the $(i,j)$ entry is $p_{i,j}(\bar{\alpha})=\sum_{a_1+a_2+\ldots +a_i=j}\alpha_{a_1}\alpha_{a_2} \ldots \alpha_{a_i}.$

Let $l=\omega(\lambda)$ denote the width of $\l \subset \mathbb{Z}^d$, that is, the maximal integer coordinate appearing in any of the boxes belonging to $\l$. Fix an integer $N\ge 1$ and define
\[\Theta_\lambda=\left\{\Psi\in J_\omega(n,N):\exists \g \in J_\lambda^\reg(d,n) \text{ such that }  (\Psi \circ \g)_\lambda=0
\right\},\]
that is, $\Theta_\lambda$ is the set of those $\lambda$-jets of germs on $\CC^n$ at the origin which vanish on some regular map. By definition, $\Theta_\lambda$ is the image
of the closed subvariety of $J_\omega(n,N) \times J_\lambda^\reg(d,n)$ defined by
the algebraic equations $(\Psi \circ \g)_\lambda=0$, under the projection to
the first factor. If $\Psi \circ \gamma=0$, we call $\g$ a {\em test
map} of $\Theta$. 

Test maps of germs are generally not unique. A basic but crucial observation is the following. If $\g$ is a test
map of $\Psi \in \Theta_\lambda$, and $\vp \in \tilde{J}_\lambda^\reg(d,d)$ is a 
holomorphic reparametrisation of $\CC^d$, then $\g \circ \vp$ is,
again, a test map of $\Psi$:
\begin{displaymath}
\label{basicidea}
\xymatrix{
  \CC \ar[r]^\vp & \CC \ar[r]^\g & \CC^n \ar[r]^{\Psi} & \CC^N}
\end{displaymath}
\[\Psi \circ \g=0\ \ \Rightarrow \ \ \ \Psi \circ (\g \circ \vp)=0\]

In fact, we get all test maps of $\Psi$ in this way if the
following property (open and dense in $\theta_\lambda$) holds: the linear part of $\Psi$ has
$d$-dimensional kernel. Before stating this in Theorem 
\ref{embedgrass} below, let us write down the equation $\Psi \circ
\g=0$ in coordinates in an illustrative example. 

\begin{example} Let $\lambda=(\delta,2\delta)$ which is a $d=2$-dimensional partition of width $\omega=2\delta$. Let $\xi \in \calH^{(\delta,2\delta)}(\CC^n)=\calH(\underbrace{\ydiagram{3, 6}}_{2\delta})$ be a generic subscheme supported at $p\in X$. Then by definition $\calo_\xi\simeq \CC[u,v]/(u^{2\delta},u^{\delta}v,v^2)$.
 Let 
\[\g=\{\frac{\partial \g}{\partial^i u\partial^j v}  : (i,j)\in (\delta,2\delta)\} \in J_{(\delta,\delta)}^\reg(d,n)\]
be the $\lambda$-jet of the test map $\g$ and 
\[\Psi=(\Psi',\Psi'',\ldots ,\Psi^{(2\delta)})\]
the $2\delta$-jet of $\Psi$. Using the chain rule and the shorthand notation $v_{ij}=\frac{\partial \g}{\partial^i u\partial^j v}$ for $(i,j)\in (\delta,2\delta)$ 
the equation $(\Psi \circ \g)_{(\delta,2\delta)}=0$ reads as follows for $\delta=2$:
\begin{eqnarray}\label{nodalequations} \nonumber
& \Psi'(v_{10})=0    \\ \nonumber
& \Psi'(v_{20})+\Psi''(v_{10},v_{10})=0 \\ \nonumber
& \Psi'(v_{30})+2\Psi''(v_{10},v_{20})+\Psi'''(v_{10},v_{10},v_{10})=0 \\ \nonumber
& \Psi'(v_{01})=0 \\ \nonumber
& \Psi'(v_{11})+\Psi''(v_{10},v_{01})=0 \\ \nonumber
\end{eqnarray}

\end{example}

\begin{lemma}\label{lambdaequations}  Let 
\[\g=\{\frac{\partial \g}{\partial^i u\partial^j v}  : (i,j)\in \lambda\} \in J_\lambda^\reg(d,n)\]
be the $\lambda$-jet of a test map $\g$ and 
\[\Psi=(\Psi',\Psi'',\ldots ,\Psi^{\omega(\lambda)})\]
the $\omega(\lambda)$-jet of $\Psi$.
Then substituting $v_{ij}=\frac{\partial \g}{\partial^i u\partial^j v}$, the equation $(\Psi\circ \g)_{\lambda}=0$ is equivalent to the following system of $|\lambda|-1$ linear equations with values in
  $\CC^N$:
\begin{equation}
  \label{modeleqlambda}
\sum_{\tau \in \mathcal{P}(\mu)} \Psi(\bv_\tau)=0,\quad \mu \in \lambda, \mu \neq (0,0)
\end{equation}
Here $\mathcal{P}(\mu)$ denotes the set of partitions $\tau=\tau_1+\ldots +\tau_s$ of $\mu$ into elements in $\lambda$ and $\bv_\tau=v_{\tau_1}\cdots v_{\tau_s}$. 
\end{lemma}
\begin{remark}

\end{remark}
For a given $\g \in \jetregl dn$ let $\mathcal{S}^{N}_{\g}$ denote the set of 
solutions of the equations in \eqref{modeleqlambda}, that is,
\begin{equation}\label{solutionspacelambda}
\mathcal{S}^{N}_\g=\left\{\Psi \in J_{\omega(\lambda)}(n,N):\ (\Psi \circ \g)_\lambda=0\right\}
\end{equation}
The equations \eqref{modeleqlambda} are linear in $\Psi$, hence
\[\mathcal{S}^{N}_\g \subset J_{\omega(\l)}(n,N)\]
is a linear subspace of codimension $(|\l|-1)N$, i.e a point of $\grass_{\mathrm{codim}=(|\l|-1)N}(J_{\omega(\l)}(n,N))$, whose orthogonal, $(\mathcal{S}^{N}_{\g})^\perp$, is an $(|\l|-1)N$-dimensional subspace of $J_{\omega(\l)}(n,N)^*$. These subspaces are invariant under the reparametrization of $\g$ and again 
\[(\mathcal{S}^{N}_{\g})^\perp=(\mathcal{S}^{1}_{\g})^\perp \otimes \CC^N\]
holds. 

For $\Psi \in J_{\omega(\l)}(n,N)$ let $\Psi^1 \in \Hom(\CC^n,\CC^N)$ denote the linear part. When $N\ge n$ then the subset 
\[\tilde{\mathcal{S}}^{N}_{\g}=\{\Psi \in \mathcal{S}^{N}_{\g}: \dim \ker \Psi^1=d\}\]
is an open dense subset of the subspace $\mathcal{S}^{N}_{\g}$. In fact it is not hard to see that the complement $\tilde{\mathcal{S}}^{N}_{\g}\setminus \mathcal{S}^{N}_{\g}$ where the kernel of $\Psi^1$ has dimension at least two is a closed subvariety of codimension $>1$.

%Furthermore, for $\g \in J_k(1,n)$ by putting $\g^{(i)}/i! \in \CC^n$ into the $i$th column of a matrix we can identify elements of $\jetreg 1n$ as $k$-by-$n$ matrices, that is, elements of $\Hom(\CC^k,\CC^n)$ with nonzero first column.   

\begin{theorem}\label{embedgrass} \begin{enumerate}
\item The map
\[\phi: \jetregl dn \rightarrow \grass_{(|\l|-1)}(J_{\omega(\l)}(n,1)^*)\]
defined as  $\gamma  \mapsto (\mathcal{S}^{1}_\g)^\perp$
is $\diff_{\l}=\tilde{J}^\reg_\l(d,d)$-invariant and induces an injective map on the $\diff_{\l}$-orbits into the Grassmannian 
\[\phi^\grass: \jetregl dn /\diff_{\l} \hookrightarrow \grass_{(|\l|-1)}(J_{\omega(\l)}(n,1)^*).\]
Moreover, $\phi$ and $\phi^\grass$ are $\GL(n)$-equivariant with respect to the standard action of $\GL(n)$ on $\jetregl dn \subset \Hom(\CC^{|\l|-1},\CC^n)$ and the induced action on $\grass_{(|\l|-1)}(J_{\omega(\l)}(n,1)^*)$.
\item Recall that $J_{\omega(\l)}(n,1)^*=\symdotl$. The image of $\phi$ and the image of $\varphi$ defined in Remark \ref{remark:embed} coincide in $\grass_{(|\l|-1)}(\symdotl)$.
\[\mathrm{im}(\phi)=\mathrm{im}(\varphi)\subset \grass_{(|\l|-1)}(\symdotl).\]
\item The closure of the image is the punctual geometric subset $\Hilb_0^{\lambda}(\CC^n)$:
\[\Hilb_0^{\lambda}(\CC^n)=\overline{\phi^\grass(\jetregl dn)}\]

\end{enumerate}
\end{theorem}

\begin{proof}
For the first part it is enough to prove that for $\Psi \in \Theta_\l$ with $\dim \ker \Psi^1=d$ and $\gamma,\delta \in J_\l^\reg(d,n)$
\[(\Psi \circ \gamma)_\l=(\Psi \circ \delta)_\l=0 \Leftrightarrow \exists
\Delta \in \diff_{\l}\text{ such that } \gamma_\l=\delta_\l
\circ \Delta.\]
We prove this statement by induction on the length of $\l$. Let $\gamma=X_1t+\dots +t^k$ and
$\delta=w_1t+\dots+ w_kt^k$. Since $\dim \ker \Psi^1=1$, $v_1=\lambda w_1$, for some
$\lambda\neq0$. This proves the $k=1$ case. 

Suppose the statement is true for $k-1$. Then, using the appropriate
order-($k-1$) diffeomorphism, we can assume that $v_m=w_m$, $m=1\ldots
k-1.$ Then from the equation
$\Psi\circ\gamma=0$ we get that  $\Psi^1(v_k)=\Psi^1(w_k)$, and hence
$w_k=v_k-\lambda v_1$ for some $\lambda\in\CC$. Then
$\gamma=\Delta\circ\delta$ for $\Delta=t+\lambda t^k$, and the proof
is complete.

The second and third part immediately follow from the definition of $\varphi$ and $\phi$.
\end{proof}

\begin{remark}\label{remark:orbit} \begin{enumerate}
%\item In particular the second part of Theorem \ref{embedgrass} tells us that the curvilinear component $\overline{CX}^{[k+1]}_p=\overline{\mathrm{im}(\phi^\grass_p)}$ has a $\GL(n)$-equivariant embedding into the Grassmannian $\grass_k(\cald^k_{X,p})$ as the closure of the image of $\phi^\grass_p$. 
\item From Lemma \ref{lambdaequations} immediately follows that 
\begin{equation}\label{sgammal}
(\mathcal{S}^{1}_\g)^\perp=\Span_\CC \left(\sum_{\tau \in \mathcal{P}(\mu)} \bv_\tau:\  \mu \in \lambda \setminus \{(0,0)\}\right) \subset \symdotl.
\end{equation} 
\item Let $\{e_1,\ldots, e_n\}$ be a basis of $\CC^n$. Pick any $|\l|-1$ elements of this basis and index them with the boxes of $\l$. Since $\phi$ is $\GL(n)$-equivariant, for $|\l|-1 \le n$ the $\GL(n)$-orbit of 
\[p_{\l,n}=\phi(e_{\mu}: \mu \in \lambda)=\mathrm{Span}_\CC \left(\sum_{\tau \in \mathcal{P}(\mu)} e_\tau:\  \mu \in \lambda \setminus \{(0,0)\}\right),\] 
forms a dense subset of $\Hilb_0^{\lambda}(\CC^n)$ and therefore 
\[\overline{\phi(\jetregl dn)}=\overline{\mathrm{GL_n} \cdot p_{\l,n}}.\] 
\end{enumerate}
\end{remark}

%\begin{theorem}[Test map model]
%Let $\lambda$ be a $d$-dimensional partition and let
%\[\Sigma_\lambda^N=\{f\in J_\lambda(\CC^n,\CC^N):\CC[z_1,\ldots, z_n]/(f_1,\ldots f_{N}) \simeq \CC[u_1,\ldots, u_d]/I_\lambda\]
%denote the monomial singularity corresponding to $\lambda$. Then 
%\[f\in \Sigma_{\lambda}^N \Leftrightarrow \exists \g \in J_{\max(\lambda)}(\CC^d,\CC^n) \text{ such that } \frac{\partial (\g \circ f)}{\partial^i x\partial^j y}=0 \text{ for } (i,j)\in \lambda.\]
%\end{theorem}

%This model imposes $|\lambda|-1$ equations and induces an embedding 
%\[\rho_\lambda: \Sigma_{\lambda}^1 \hookrightarrow \grass_{|\lambda|-1}(\Sym)\]
%for some space $\Sym$ which is a direct sum of symmetric powers of $\CC^n$ coming from the order of these equations. The closure of the image in the Grassmannian is $P(\lambda)$...

%ADD DETAILS HERE 
%Then we can apply equivariant localisation, turn the formula into iterated residue and prove a vanishing theorem. 
%ADD DETAILS LATER

\begin{example}\label{example:gottsche} For $\lambda=(\delta, 2\delta)$ the embedding $\phi^\grass$ has the following form: 
\[\phi^\grass: \Hilb^{(\delta,2\delta)}_0(\CC^n) \hookrightarrow \grass_{3\delta-1}(\symdotl)\]
\[(v_{i0},v_{j1}:1\le i \le 2\delta-1,0\le j \le \delta-1) \mapsto \mathrm{Span}\left(\begin{cases} v_{10} \\ v_{20}+v_{10}^2 \\ \ldots \\ v_{2\delta-1,0}+\ldots +v_{10}^{2\delta-1} \end{cases},\begin{cases} v_{01} \\ v_{11}+2v_{10}v_{01} \\\ldots \\ v_{\delta-1,1}+\ldots +v_{10}^{\delta-1}v_{01} \end{cases}\right)\]

In particular, for $\delta=2$ we get  $\lambda=(2,4)$ and the equations above provide this embedding immediately as 
\[\Sigma_{2}^1 \hookrightarrow \grass(5, \Sym)\]
\[(v_{10},v_{20},v_{30},v_{01},v_{11}) \mapsto \mathrm{Span}(v_{10},v_{20}+v_{10}^2,v_{30}+2v_{10}v_{20}+v_{10}^3,v_{01},v_{11}+2v_{10}v_{01})\]
Note that in this example the unipotent radical of the stabiliser $\Stab(p_{\l,n})=\diff_\l$ is generated by the actions 
\[v_{20}+tv_{10},\  v_{30}+tv_{10},\ v_{20}+tv_{01},\  v_{11}+tv_{01},\ v_{01}+tv_{10}\]
that is, 
\[\diff_\l=\]
Hence we actually have an embedding into a smaller Grassmannian: we can replace $\sym^{\le \omega(\l)}$ with 
\[\widetilde{\sym}=\{v{10}\wedge (v_{20}+v_{10}^2)\wedge (v_{30}+v_{10}v_{20}+v_{10}v_{01})\wedge v_{01} \wedge (v_{11}+v_{10}v_{01}+v_{10}^2): v_{ij} \in \CC^n\}\subset \symdotl\]
We will see in \S \ref{sec:severi} that this test curve model, and the localisation over that can give better formula for the integrals, because $\grass(\symdotl)$ is a smaller dimensional smooth ambient space of the geometric subset than the set of all filtered commutative algebra structures in the Kazarian model. 
\end{example}

\subsection{The Kazarian model of punctual geometric subsets}\label{subsec:kazarianmodel} This section is a short overview of the model introduced by Kazarian in \cite{kazarian3}. 
Let $V$ be an $n$-dimensional vector space, and let $\frakm$ denote the maximal ideal in the ring of function germs at the origin of $V$. 
Let $I \subset \frakm$ be an ideal of finite codimension $k$, and let $N=\frakm/I$ be the nilpotent quotient algebra. Then we have two natural linear maps:
\[\psi_1: V^* \to N\] 
is the restriction of the projection $\frakm \to \frakm/I=N$ to the dual $V^*$ consisting of linear functions. 
The second map is 
\[\psi_2: \Sym^2 N \to N.\]
given by the multiplication law in the algebra $N$. 
\begin{lemma} The induced linear map $\psi_1 \oplus \psi_2 : V^* \oplus \Sym^2 N \to N$ is surjective.
\end{lemma}
%\begin{proof} Consider a polynomial representing any element of $N$. This polynomial can be represented as a sum of linear terms which are in the image of $\psi_1$ and terms of order greater or equal to two which are in the image of $\psi_2$.

%Conversely, let $N$ be a $\mu$-dimensional commutative associative nilpotent algebra with the structure morphism $\psi_2 : \Sym^2 N \to N$, and $\psi_1: V^* \to N$ be an arbitrary linear map such that $\psi_1 \oplus \psi_2$ is surjective. Then $\psi_1$ extends uniquely to an algebra homomorphism $\psi_1: \frakm \to N$ for $V^*$ generates $\frakm$. The homomorphism $\psi_1$ is surjective since $\psi_1 \oplus \psi_2$ is surjective. It follows that $N$ can be identified with $\frakm/I$ where $I=\ker(\psi_1)$. 
%\end{proof}

This lemma implies the following.

\begin{proposition}\label{idealcharacterisation}
Let $V$ be an $n$-dimensional vector space, $\frakm \subset \calo_{V,0}$ the maximal ideal at the origin, and let $N$ be a vector space of dimension $k$. There is a one-to-one correspondence 
 \begin{equation}\label{1-1}
 \left\{I \subset \frakm: \dim(\frakm/I)=k\right\} \longleftrightarrow \left\{\substack{(\psi_1,\psi_2)\ |\ \  \psi_1: V^* \to N \text{ linear }\\ \psi_2 : \Sym^2 N \to N \text{ is an associative commutative nilpotent algebra structure on $N$} \\
 \psi_1 \oplus \psi_2 \text{ is surjective}}\right\}
  \end{equation} 
 Let $Y_r \subset \Hom(V^* \oplus \Sym^2 N,N)$ denote the Zariski open set on the right hand side of \eqref{1-1}.
\end{proposition}

Next we describe a filtered version of Proposition \ref{idealcharacterisation}. Assume that $N$ is endowed with the filtration
\[N=N_1 \supset N_2 \supset \ldots \supset N_{r+1}=0,\ \  \dim N_k/N_{k+1}=d_k.\]
We call a linear map $\psi: \Sym^2 N \to N$ satisfying $\psi(N_k \otimes N_m) \subset N_{k+m}$ a \textit{filtered commutative algebra structure} on $N$. We denote by $\mathrm{Alg}(d,N)=\mathrm{Alg}(d_1, \ldots , d_r)$ the vector space of filtered commutative algebra structures on $N$. Then $\dim N=k=\sum_{i=1}^r d_i$ and we call $d=(d_1,\ldots, d_r)$ the dimension vector of the filtration. Let $e_1,\ldots ,e_k$ be a basis of $N$ such that the vectors $e_{d_1+\ldots +d_{k-1}+1},\ldots ,e_{d_1+\ldots +d_{k-1}+d_k}$ span $N_k/N_{k+1}$. Then the multiplication is given by the structure equations  
\[e_i \cdot e_j=\psi(e_i,e_j)=\sum_{w(i)+w(j) \le w(\ell)} q_{i,j}^\ell e_\ell\]
where we define the weight $w(i)$ for an integer $i$ satisfying 
\[d_1+ \ldots +d_{s-1} < i \le  d_1+\ldots+d_s\]
to be equal to $w(i)=k$. The structure coefficients
\[q^k_{ij}=q^k_{ji} ,\ \  w(i)+w(j) \le w(k)\]
form coordinates in the vector space $\mathrm{Alg}(d)$.

Filtered commutative algebras are not necessarily associative: we denote by $\mathrm{Assoc}(d,N) \subset \mathrm{Alg}(d,N)$ the subvariety of \textit{filtered associative algebras} on the filtration $N_\bullet$. The ideal of $\mathrm{Assoc}(d,N)$ in $\mathrm{Alg}(d,N)$ is generated by the quadratic associativity equations: the identity which expresses the equality of the degree-$t$ components of the products $(e_i \cdot e_j) \cdot e_k$ and $e_i \cdot (e_j \cdot e_k)$ for a quadruple $(i,j,s,t)$ with $w(i)+w(j)+w(s) \le w(t)$ is 
\begin{equation}\label{assoceq}
\sum_m q_{i,j}^mq_{m,s}^t-\sum_mq_{i,m}^tq_{j,s}^m=0.
\end{equation}
 
Let $\GL(d)\subset \GL(N)$ denote the linear
transformations of $N$ preserving the flag $N_\bullet$. This subgroup is determined by the dimension vector $d=(d_1,\ldots, d_{r})$ and $\mathrm{Assoc}(d,N)$ is invariant under $\GL(d)$. For a fixed nilpotent associative algebra $A \in \mathrm{Assoc}(d,N)$ of dimension $k$ we let 
\[Q(A)=\mathrm{Assoc}_A(d,N)=\overline{\{\xi \in \mathrm{Assoc}(d,N): A_\xi \simeq A\}}\]
denote the closure of the subset formed by the algebras isomorphic to $A$. 

Now we discuss the connection of $\mathrm{Assoc}(d,N)$ and $\mathrm{Assoc}_A(d,N)$ with Hilbert schemes.  Let $V=\CC^n$ and denote by $\Hilb_d(V)$ the nested Hilbert scheme parametrising flags of ideals 
\begin{equation}\label{idealflag}
\mathfrak{m}=I_1 \supset I_2 \supset \ldots \supset I_{r+1}=I
\end{equation}
such that $\dim I_s/I_{s+1}=d_s$ for $s=1,\ldots ,r$ and $I_i \cdot I_j \subset I_{i+j}$ for all $i$ and $j$ with $i+j \le r+1$. Since $\frakm^{r+1} \subset I_{r+1}$, the variety $\Hilb_d(V)$ is equipped with the bundles 
\[0 \to I_{r+1}/\frakm^{K+1} \to (\frakm/\frakm^{K+1}) \times \Hilb_d(\CC^n) \to \frakm/I_{r+1} \to 0 \]
where $I_{r+1}/\frakm^{K+1}$ is the canonical subbundle of the trivial bundle $(\frakm/\frakm^{K+1}) \times \Hilb_d(V)$ and quotient bundle $N=\frakm/I_{r+1}$ of rank $\mu$. Let $D=N^*$ denote its dual. 

The algebra $A_i=\CC[x_1,\ldots, x_n]/I_{I}$ has dimension $1+d_1+\ldots +d_i$, and we denote by 
\[\Hilb^{A_\bullet}(\CC^n)=\{\mathfrak{m}=I_1 \supset I_2 \supset \ldots \supset I_{r+1}: \mathfrak{m}/I_i \simeq A_i\} \subset \Hilb_d(\CC^n)\] 
the closure of the subvariety formed by flags of ideals with prescribed isomorphism types of the quotient filtered algebras. $\Hilb^A(V)$ and $\Hilb_d(V)$ are not necessarily smooth spaces. However, using Proposition \ref{idealcharacterisation} we now construct an embedding of $\Hilb_d(V)$ to some compact nonsingular variety $X$ such that 1) the action of $\GL(n)$ and the bundle $D$ extend to $X$ and 2) Equivariant localisation on $X$ can be transformed into iterated residue and the vanishing theorem holds. 

The ideal flag \eqref{idealflag} defines a filtration
\[N=\frakm/I_{r+1}=N_1 \supset N_2 \supset \ldots \supset N_{r+1}=0, \ \ \dim N_k/N_{k+1}\]
which is compatible with the nilpotent algebra structure on $N=\frakm/I_{r+1}$.

\begin{definition}\label{def:nonasshilbert} Let $X_r \subset \Hom(V^* \oplus \Sym^2 N,N)$ the subset on the right hand side of \eqref{1-1} without the associativity condition: 
\[X_r=\left\{\substack{(\psi_1,\psi_2)\ |\ \  \psi_1: V^* \to N \text{ linear }\\ \psi_2 : \Sym^2 N \to N \text{ is commutative nilpotent algebra structure on $N$} \\
 \psi_1 \oplus \psi_2 \text{ is surjective}}\right\}\]
The surjectivity determines a Zariski open subset in the variety in $\Hom(V^* \oplus \Sym^2 N,N)$ satisfying the first two conditions. The group $\GL(d)\subset \GL(N_\bullet)$ acts naturally on $\tilde{X}_r$. 
\end{definition}

\begin{proposition}\label{freeaction} \begin{enumerate}
\item The action of $\GL(d)$ on $X_r$ is free, hence the orbit space 
\[\NAHilb^d(V)=X_r/\!/ \GL(d)\]
is smooth and compact. We call this the nonassociative Hilbert scheme, and $q:X_r \to \NAHilb^d(V)$ the quotient map. 
\item The Hilbert scheme $\Hilb_d(V)$ is isomorphic to the locus in $\NAHilb^d(V)$ corresponding to associative algebras:
\[\Hilb_d(V)=q(Y_r)=Y_r/\!/\GL_d \subset \NAHilb^d(V).\] 
\item  The geometric subset $\Hilb^A_0(V) \subset \Hilb_d(V)$ consists of the points of $\NAHilb^d(V)$ for which the canonical algebra on the corresponding fiber of $N_r$ is isomorphic to $A$:
\[\Hilb^{A_\bullet}_0(\CC^n)=\{\psi \in Y_r : \frakm/I_s \simeq A_s\}//\GL(d) \subset \Hilb_d(V) \subset \NAHilb^d(V).\]  
\end{enumerate}
\end{proposition}

%Note that $\GL_d$ is not reductive and therefore Mumford's GIT does not apply directly. Still, we can use the following terminology.

%\begin{definition}\label{defxr}
%We call $\tilde{X}_r=\mathrm{Alg}(d,N)^{s}$ the GIT stable locus for the $\GL_d$ action on $\mathrm{Alg}(d,N)^{s}$ and define the GIT quotient to be $X_r$: 
%\[\mathrm{Alg}(d,N) /\!/ \GL(d):=X_r\]
%with the quotient map $q: \mathrm{Alg}(d,N) \to X_r$ sending a point to its orbit. We define  
%\[\mathrm{Assoc}(d,N)/\!/\GL_d:=q(\mathrm{Assoc}(d,N))\]
%as the image of $q$ restricted to $\mathrm{Assoc}(d,N) \subset \mathrm{Alg}(d,N)$. Similar definition stands for \\$\mathrm{Assoc}_A(d,N)/\!/\GL_d$. 
%\end{definition}

We finish this overview of \cite{kazarian3} by presenting a construction of the non-associative Hilbert scheme as a tower of Grassmannian bundles. We use the dual commutative coalgebra structure on $D=N^*$. For a filtration
\[0=D_0 \subset D_1 \subset \ldots \subset D_r=D\]
with $\dim(D_m/D_{m-1})=d_s$ and $d_1+\ldots +d_r=k$ a filtered commutative coalgebra structure is a linear map $\psi: D \to \sym^2 D$ mapping $D_m$ to the subspace 
\[S_m \subset D \otimes D \text{ generated by } D_i \otimes D_j \text{ with } i + j \le m.\]
Then naturally sits in $S_m \subset \sym^2 D_{m-1}$. We define 
\[X^*_r=\left\{\substack{(\psi_1,\psi_2)\ |\ \  \psi_1: D \to V \text{ linear }\\ \psi_2 : D \to \Sym^2 D \text{ is commutative nilpotent algebra structure on $D$} \\
 \psi_1 \oplus \psi_2 \text{ is injective}}\right\}\]
and the non-associative Hilbert scheme is again the orbit space
\[\NAHilb^d(V)=X_r/\!/ \GL(d)=X^*_r/\!/ \GL(d)\]
for the $\GL(d)$ action. 

Forming the quotient can be done iteratively, by induction on $r$. Let $\psi=\psi_1\oplus \psi_2: D \hookrightarrow V \oplus \sym^2 D$ denote the embedding. For $r=1$ we have $S_1 = 0$, therefore $\psi(D_1) \subset V$ and $\NAHilb^{d_1}(V)=\grass_{d_1}(V)$ is the Grassmannian manifold with the tautological rank $d_1$ bundle $D_1$ over it.

Assume that we already contructed the tower $\NAHilb^{d_1,\ldots, d_{r-1}}(V)$ with the corresponding tautological flag of bundles $D_1 \subset  \ldots  \subset D_{r-1}$ over it and with the embedding of subbundles $\psi(D_{r-1}) \subset V \oplus S_{r-1}$. Since $S_r$ is determined by $D_1, \ldots , D_{r-1}$ the bundle $S_r$ sits over $X_{r-1}$. The manifold $X^*_r$ parametrise subspaces $D_r$ in $V \oplus S_r$ containing the subspaces $D_{r-1}$ constructed on the previous step. Hence 
\[X^*_r=\grass_{d_r}((V \oplus S_r)/D_{r-1})\]
is the Grassmannian bundle over $X^*_{r-1}$.
This way we constructed $\NAHilb^d(V)$ as the total space of a tower of Grassmannian fibrations with smooth compact fibers:
\begin{equation}\label{mrtower}
\xymatrixcolsep{5pc} \xymatrix{\NAHilb^{d_1,\ldots ,d_r}(V) \ar[r]^{\grass_{d_r}(V_r/D_{r-1})} & \NAHilb^{d_1,\ldots ,d_{r-1}}(V) \ar[r]^-{\grass_{d_{r-1}}(V_{r-1})/D_{r-2})} \ar[r] & \ldots  \ar[r]^{\grass_{d_1}(V)} & pt  
}
\end{equation}
where $V_m=V\oplus S_m$. The non-associative Hilbert scheme is endowed with a diagram of bundles and subbundles
\begin{equation}\label{mrbundles} 
\xymatrix{D_1 \ar@{^{(}->}[r] \ar@{^{(}->}[d] & D_2 \ar@{^{(}->}[r] \ar@{^{(}->}[d]  & \ldots \ar@{^{(}->}[r] & D_r \ar@{^{(}->}[d] \\
V=V_1 \ar@{^{(}->}[r] & V_2 \ar@{^{(}->}[r] & \ldots \ar@{^{(}->}[r] & V_r 
}
\end{equation}

Proposition \ref{idealcharacterisation} can be reformulated as 
\begin{proposition} The Hilbert scheme $\Hilb_d(V)$ is isomorphic to the sublocus in $X_r$ consisting of points for which the canonical commutative filtered algebra on the corresponding fiber of $N_r$ is associative:
\[\Hilb_d(V)=\mathrm{Assoc}(d,N)/\!/\GL_d \subset X_r.\] 
 The geometric subset $\Hilb^A_0(\CC^n) \subset \Hilb_d(V)$ consists of the points of $X_r$ for which the canonical algebra on the corresponding fiber of $N_r$ is isomorphic to $A$:
\[\Hilb^A_0(\CC^n)=\mathrm{Assoc}_A(d,N)/\!/\GL_d=\{\xi \in \Hilb_d(V): A_\xi \simeq A\} \subset X_r.\] 
\end{proposition}

%\subsection{Comparison of the two models: examples}

\subsection{The dimension of punctual geometric subsets}\label{subsec:dim} In this section we use the test-jet model of punctual monomial geometric subsets to describe a combinatorial formula for the dimension of $\Hilb^{\l}(\CC^n)\subset \Hilb^k(X)$ for any monomial ideal. Using this we prove Proposition \ref{prop:components} on curvilinear subsets. %In fact, according to Corollary \ref{cor:irreducible} we only have to show that the codimension of $\Hilb^{\l_1+\ldots +\l_s}(\CC^n)$ in $\Hilb^{\l_1,\ldots, \l_s}(X)$ is $s-1$. 

Recall from Theorem \ref{embedgrass} that $\Hilb^{\l}(\CC^n)$ can be identified with the closure of $\jetregl dn /\diff_{\l}$ in $\grass_{(|\l|-1)}(J_{\omega(\l)}(n,1)^*)$. The embedding is given as (see Remark \ref{remark:orbit}) 
\[\phi: \jetregl dn \to \mathrm{Hom}^{\reg}(\CC^{(|\l|-1)},\symdotl)\]
\[(\bv_\tau:\tau \in \l) \mapsto \left(\sum_{\tau \in \mathcal{P}(\mu)} \bv_\tau:\  \mu \in \lambda \setminus \{(0,0)\}\right)\]
where the right hand side is a vector of dimension $|\l|-1$ with entries in $\symdotl$ and $\phi$ induces an embedding
\[\jetregl dn/\diff_{\l} \hookrightarrow \grass_{(|\l|-1)}(\symdotl).\]
Therefore  
\[\dim(\Hilb^{\l}(\CC^n))=\dim(\jetregl dn)-\dim \mathrm{Stab}(\phi)\]
where $\mathrm{Stab}\phi \subset \GL((|\l|-1))$ 
is the stabiliser of the map $\phi$.

\subsection{Fibration over the flags in $TX$.}
Let $A=\CC[x_1,\ldots, x_n]/I$ be a local nilpotent algebra on $n$ generators and $\dim(A)=k$. In this section we define a partial resolution of $\Hilb^A(X)$ which fibers over the tangent flag manifold $\flag_{k-1}(TX)$. 
Let $k-1\le n$ and let $P_{n,k-1} \subset \GL_n$ denote the parabolic subgroup which preserves the flag 
\[\mathbf{f}=(\mathrm{Span}(e_1)   \subset \mathrm{Span}(e_1,e_2) \subset \ldots \subset \mathrm{Span}(e_1,\ldots, e_{k-1})=\CC_{[k-1]} \subset \CC^n).\] 
\begin{definition}\label{def:xktilde}
Define the partial desingularization 
\[\widehat{\Hilb}^{A}_0(\CC^n)=\GL_n \times_{P_{n,k-1}} \overline{P_{n,k-1} \cdot p_{n,k-1}}\]
with the resolution map $\rho: \widehat{\Hilb}^{A}_0(\CC^n) \to \Hilb^{A}_0(\CC^n)$ given by $\rho(g,x)=g\cdot x$. The test curve model defines the embedding
\[\xymatrix{\widehat{\Hilb}^{A}_0(\CC^n) \ar@{^{(}->}[r]^-\rho \ar[d]^\rho & \widehat{\grass}_{k-1}(\sym^{\le k-1} \CC^n) \ar[ld] \\ \flag_{k-1}(\CC^n) & }\]
where $\widehat{\grass}_k(\sym^{\le k} \CC^n)=\GL_n \times_{P_{n,k}} \grass_k(\sym^{\le k}\CC_{[k]})$ and $\flag_k(\CC^n)=\GL(n)/P_{n,k}$ is the flag variety.
\end{definition} 
\begin{remark} The fiber $\rho^{-1}(\bff)$ consists of limit subschemes $\lim_{l \to \infty} \xi^l$ where $\xi^l=\{p_1^l \sqcup \ldots \sqcup p_{k+1}^l\}$ 
is non-reduced such that $\lim_{l \to \infty} \Span(p_1,\ldots, p_i)=\Span(e_1,\ldots, e_{i-1})$ for $2\le i \le k+1$. 
\end{remark}

The fiberwise resolution results in a fibration over the flag tangent bundle:

\begin{definition}\label{def:widetilde} Define the partial desingularisation 
\[\widehat{\CHilb}^{k+1}(X)=\GL_X \times_{P_{n,k}}  \overline{P_{n,k} \cdot p_{k,m}}\]
and 
\[\widehat{\grass}_k(\sym^{\le k} TX)=\GL_X \times_{P_{n,k}} \grass_k(\sym^{\le k}\CC_{[k]})\]
with the partial resolution map 
\[\rho: \widehat{\CHilb}^{k+1}(X)=\GL_X \times_{\GL(n)} \left(\GL(n) \times_{P_{n,k}} \overline{P_{n,k}\cdot \mathfrak{p}_{m,k}}\right) \to \GL_X \times_{\GL(n)} \left(\overline{\GL(n)\cdot \mathfrak{p}_{m,k}}\right)=\CHilb^{k+1}(X)^{\GL}.\]
This results the diagram 
\[\xymatrix{\widehat{\CHilb}^{k+1}(X) \ar@{^{(}->}[r]^-\rho \ar[d]^\rho & \widehat{\grass}_k(\sym^{\le k} TX) \ar[ld] \\ \flag_k(TX) \ar[d]^\pi & \\ X & }\]
\end{definition}

\subsection{Neighborhood of $\Hilb^A(X)$ in $\GHilb^k(X)$}

Finally, the same argument works over a smooth manifold $X$: we let $\bU \subset TX$ be a small $\GL_X$-invariant tubular neighborhood of the zero section, with the exponential map $\exp: \bU \to X$, which identifies $U_x$ with $\exp(\bU_x) \subset X$. We may assume that $\bU=\GL_X \times_{\GL(n)} U$ for some $U \subset \CC^n$. We define 
\[\widehat{\BHilb}^{k+1}(\bU)=\GL_X \times_{P_{n,k}} \BHilb^{k+1}_\bff(U)\]
We get the diagram
\begin{equation}\label{diagramfour}
\xymatrix{
\widehat{\Hilb}^{A}(X) \ar@{^{(}->}[r] \ar[d]^\rho &  \widehat{\BHilb}^{A}(\bU) \ar[ld]^{\rho_{\BHilb}} \ar@{^{(}->}[r]^-{\iota} & \grass_{k+1}(S^\bullet TX) \ar[lld]^{\rho_\grass}\\
\flag_{k-1}(TX) \ar[d]^\mu & &\\
X &&} 
\end{equation}
with a birational morphism $\widehat{\BHilb}^{k+1}(\bU) \to \BHilb^{k+1}(\bU)=\cup_{p\in X} \BHilb^{k+1}(U_p)$. 
\begin{remark} The fiber $\rho^{-1}(f)$ over a flag 
\[f=(F_1 \subset F_2 \subset \ldots \subset F_k \subset T_pX)\]
sitting over the point $p\in X$ consists of limit subschemes $\lim_{l \to \infty} \xi^l$ where $\xi^l=\{p_1^l \sqcup \ldots \sqcup p_{k+1}^l\}$ is non-reduced such that $p_i \in U_p$ and $\lim_{l \to \infty} \Span(p_1,\ldots, p_i)=F_{i-1}$ for $2\le i \le k+1$. 
\end{remark}

\section{Approximating geometric subsets}

In this section we define approximating sets of geometric subsets following Rennemo \cite{rennemo} and Li \cite{junli}, and formulate a motivic localisation formula which presents integrals over geometric subsets as a certain sum of integrals over boundary components supported on diagonals via the Hilbert-Chow morphism. 

\subsection{Tautological bundles and integrals}\label{subsec:tautologicalbundles} Let $X$ be a smooth projective variety of dimension $n$ and let $F$ be a rank $r$ bundle (loc. free sheaf) on $X$.  
Let $\Hilb^k(X)$ denote the Hilbert scheme of $k$ points on $X$ parametrising length $k$ subschemes of $X$ and $F^{[k]}$ the corresponding rank $rk$ bundle on $\Hilb^k(X)$ whose fibre over $\xi \in \Hilb^k(X)$ is $F \otimes \calo_{\xi}=H^0(\xi,F|_\xi)$. 
  
Equivalently, $F^{[k]}=q_*p^*(F)$ where $p,q$ denote the projections from the universal family of subschemes $\mathcal{U}_k$ to $X$ and $\Hilb^k(X)$ respectively: 
 \[\Hilb^k(X)\times X \supset \xymatrix{\mathcal{U} \ar[r]^-q \ar[d]^-p & \Hilb^k(X) \\ X & }.\]
We call $F^{[k]}$ the tautological bundle corresponding to $F$. 

In this paper we work with singular homology and cohomology with rational coefficients. For a smooth manifold $X$ the degree of a class $\eta \in H_*(X)$ means its pushforward to $H_*(pt)=\QQ$. By choosing  $\alpha_\eta \in \Omega^{top}(X)$,  a closed compactly supported differential form representing the cohomology class $\eta$ this degree is equal to the integral  
\[\eta \cap [X]=\int_{X} \alpha_\eta.\]

Let $\Hilb^{A_1,\ldots, A_s}(X) \subset \Hilb^k(X)$ be a geometric subset and $\Phi(F^{[k]})$ be a monomial in the Chern classes $c_i=c_i(F^{[k]})$ of weighted degree equal to $\dim(\Hilb^{A_1,\ldots, A_s}(X))$. Then 
\begin{equation}\label{integral}
[\Hilb^{A_1,\ldots, A_s}(X)]\cap \Phi(F^{[k]})=\int_{\Hilb^{A_1,\ldots, A_s}(X)} \a_{\Phi}
\end{equation}
is called a tautological integral of $F^{[k]}$, where $\a_{\Phi}$ is a closed compactly supported differential form representing $\Phi$.
\begin{remark}\label{remark:pullback}
\begin{enumerate}
\item In \eqref{integral} the integral of $\alpha_X$ on the smooth part of $]\Hilb^{A_1,\ldots, A_s}(X)$ is absolutely convergent and by definition we denote this by $\int_{\Hilb^{A_1,\ldots, A_s}(X)} \alpha_X$.
\item Recall that if $f:X\to Y$ is a smooth proper
map between connected oriented manifolds such that $f$ restricted to
some open subset of $X$ is a diffeomorphism, then for a compactly
supported form $\mu$ on $Y$, we have $\int_X   f^*\mu= \int_Y\mu$.
The analogous statement for singular varieties is the following. Let $f:X\to N$ be a 
smooth proper map between smooth quasi-projective
varieties and assume that $X\subset X$ and $Y\subset N$ are possibly
singular closed subvarieties, such that $f$ restricted
to $X$ is a birational map from $X$ to $Y$. If $\mu$ is a closed differential form on $N$ then the integral of $\mu$ on the smooth part of
$Y$ is absolutely convergent; we denote this by $\int_Y\mu$. With this
convention we again have
$\int_X   f^*\mu= \int_Y\mu$.

In particular this means means that the integral $\int_Y \mu$ of the compactly supported form $\mu$ on $N$ is the same as the integral $\int_{\tilde{Y}}f^*\mu$ of the pull-back form $f^*\mu$ over any (partial) resolution $f:(\tilde{Y},\tilde{X}) \to (Y,X)$.
\end{enumerate}
\end{remark}

%The key idea then is the evaluation of integrals over the bundary components using equivariant localisation on the universal model. 

\subsection{Ordered Hilbert schemes and geometric subsets}

The Hilbert scheme of ordered points $\Hilb^{[k]}(X)$ is defined by the commutative diagram
\begin{equation*}
\xymatrix{\Hilb^{[k]}(X) \ar[r] \ar[d] &  \Hilb^k(X) \ar[d] \\
 X^k \ar[r]  & \Sym^k(X) } 
\end{equation*}
where the right hand arrow is the Hilbert-Chow morphism taking a subscheme $Z$ to its support cycle. We denote by $(Z,(x_1,\ldots, x_k))$ the point in $\Hilb^{[k]}(X)$ mapping to $Z\in \Hilb^k(X)$ and $(x_i) \subset X^k$. Define the bundle $F^{[[k]]}$ on $\Hilb^{[k]}(X)$ as the pullback of $F^{[k]}$ along $\Hilb^{[k]}(X) \to \Hilb^k(X)$. 

\begin{definition}[Geometric subsets in ordered Hilbert schemes]
Let $\mathbf{Q}=(Q_1,\ldots, Q_s)$ be a set of types such that $Q_i \subset \Hilb^{k_i}_0(\CC^n)$ and $k=k_1+\ldots +k_s$. Let $\mu=(\mu_1,\ldots, \mu_s)$ be a partition of $\{1,\ldots, k\}$ such that $|\mu_i|=k_i$. We define 
\[\calH^{\mathbf{Q},\mu}(X)=\{(Z,(x_i)): Z=Z_1 \sqcup \ldots \sqcup Z_s \text{ where }  Z_j \in Q_j, x_i=\mathrm{supp}(Z_j) \text{ if } i\in \mu_j\}.\]
A subset of $\Hilb^{[k]}(X)$ is geometric if it can be expressed as finite union, intersection and comple-
ment of sets of the form $\calH^{\mathrm{Q},\mu}(X)$. The closure in $\Hilb^{[k]}(X)$ is
\[\Hilb^{\mathbf{Q},\mu}(X)=\overline{\calH^{\mathbf{Q},\mu}(X)}\] 
Note that we will often drop $\mu$ from the notation, and since we will always work with ordered Hilbert schemes, we hope this will not cause any confusion. 

\end{definition}

Note that $\Hilb^{[k]}(X) \to \Hilb^k(X)$ is a branched cover of degree $k!$, and it restricts to a branched cover $\Hilb^{\mathbf{Q},\mu}(X) \to \Hilb^{\mathbf{Q}}(X)$ of degree ${k \choose k_1 \ldots k_s}$. The projection formula then gives 
\[{k \choose k_1 \ldots k_s} \int_{\Hilb^{\mathbf{Q}}(X)} \Phi(F^{[k]})=\int_{\Hilb^{\mathbf{Q},\mu}(X)} \Phi(F^{[[k]]})\]
for any Chern polynomial $\Phi$.

\begin{definition}\label{def:curvordgeometricsubset}
Let $A_1,\ldots, A_s$ be nilpotent algebras with $n$ generators, and let $\Hilb^{A_1,\ldots, A_s ;\mu}(X) \subset \Hilb^{[k]}(X)$ denote the geometric subset of the ordered Hilbert scheme corresponding to a fixed partition $\mu=(\mu_1,\ldots, \mu_s)$ of the points $\{1,\ldots, k\}$ into groups of size $|\mu_i|=\dim(A_i)$:
\[\{1,\ldots, k\}=\sqcup_{i=1}^s \mu_i\]

We call the component $\Hilb^{\lambda_1+ \ldots + \lambda_s,\L}(X)=\Hilb^{\lambda_1+ \ldots + \lambda_s}(X)$ the curvilinear component of the geometric subset $\Hilb^{\l_1,\ldots, \l_s;\mu}(X)$. When we fix the support to be at $p\in X$ we get the punctual curvilinear component $\Hilb^{\lambda_1+ \ldots + \lambda_s,\L}_p(X)$.
\end{definition}

\begin{example}[The G\"ottsche geometric subset] The G\"ottsche geometric subset corresps to the ordered geometric subset $\Hilb^{\mathfrak{m}/\mathfrak{m}^2,\ldots , \mathfrak{m}/\mathfrak{m}^2; \mu}(X) \subset \Hilb^{[3\delta]}(X)$ with the partition $\mu=\{\mu_1,\ldots, \mu_\delta\}$ where $\mu_i=\{3i-2,3i-1,3i\}$. The curvilinear component of the G\"ottsche subset is $\Hilb^{\mathfrak{m}/\mathfrak{m}^2 + \ldots +\mathfrak{m}/\mathfrak{m}^2}(X)=\Hilb^{(2\s,\s)}(X)$.
%Hence the image of $T$ under the Hilbert-Chow morphism is $\Delta_\mu=\Delta_{\{1,2,3\},\{4,5,6\},\ldots, \{3\d-2,3\d-1,3\d\}}$. 
\end{example}

\subsection{Approximating sets} 

In this section we slightly modify the construction of Li \cite{junli} and Rennemo \cite{rennemo}, and define approximating sets of the geometric subset $\Hilb^{A_1,\ldots, A_s}(X)$.We will use the rational support map $\supp: \Hilb^{A_1,\ldots, A_s}(X) \to \GHilb^{[s]}(X)$ to the ordered main (geometric) component, and define a universal blow-up of $\Hilb^{A_1,\ldots, A_s}(X)$ with a rational morphism to the fully nested Hilbert scheme $\N^s(X)$ defined in \cite{berczitau2}. 

Hence we start with recalling from \cite{berczitau2} the construction of the fully nested Hilbert scheme. There we defined for any partition $\{1,\ldots, s\}=\a_1 \sqcup \a_2  \sqcup \ldots  \sqcup \a_t$ the scheme $\Hilb^{[\a_1,\ldots, \a_t]}(X)$ which is a certain approximation of  $\Hilb^{[s]}(X)$. We fix some notation and conventions first. For simplicity, let $\RHilb^{[s]}(X) \subset \GHilb^{[s]}(X)$ denote the open set of reduced subschemes of the form $\xi=x_1 \sqcup \ldots \sqcup x_s$, formed by $s$ different points on $X$.

\begin{definition}[Approximating sets of $\GHilb^{[s]}(X)$]\label{def:approximatingsets} Let $\a=(\a_1,\ldots, \a_t) \in \Pi(s)$ be a partition of $\{1,\ldots, s\}$. 
\begin{enumerate}
\item We let $\sim_\alpha$ be the equivalence relation on $\{1,\ldots, s\}$ given by letting the
elements of $\alpha$ form the equivalence classes.
We introduce the partial order by setting $\alpha \le \beta$ if $\sim_\a$ is a refinement of $\sim_\beta$.
%We denote by $\Lambda$ the maximal partition under this ordering, that is, $\Lambda=\{\{1,\ldots, s\}\}$.
%Given two partitions $\alpha,\beta$ we denote by $[\a,\b]$ the set of partitions $\g$ such that $\a\le \g \le \b$. Define $[\a, \b)$ etc. similarly.
Let 
%Finally, if $\a$ is a partition of $\{1,\ldots, s\}$, denote by $\Delta_\a$ the (closed) diagonal 
\[\Delta_\alpha=\{(x_1,\ldots, x_s) \in X^s | x_i=x_j \text{ if } i\sim_\a j\}.\]
denote the closed diagonal, then $\Delta_\b \subseteq \Delta_\a$ whenever $\a \le \b$.
\item Let 
\[\Hilb^{[\a]}(X)=\prod_{i=1}^s \Hilb^{[\a_i]}(X),\]
where for a subset $S \subset \{1,\ldots, s\}$, $\Hilb^{[S]}(X)$ denotes the ordered Hilbert scheme of $|S|$ points labeled by $S$. The punctual part sits over the corresponding diagonal:
\[\Hilb^{[\a]}_0(X)=\prod_{i=1}^s \Hilb^{[\a_i]}_0(X)=\HC^{-1}(\Delta_\a).\]
\item We define
\[\GHilb^{[\a]}(X) = \prod_{i=1}^s \GHilb^{[\a_i]}(X)\]
which is the the closure of $\RHilb^{[s]}(X)$ in $\Hilb^{[\a]}(X)$. Like above, the punctual part is $\GHilb^{[\a]}_0(X)=\HC^{-1}(\Delta_\a)$. We will often use the simpler notation $\GHilb^{s}(X)=\GHilb^{[\Lambda]}(X)$ where $\Lambda=\{1,\ldots, s\}$ is the trivial partition, that is, the main component of the ordered Hilbert scheme on $s$ points. 
\end{enumerate}
\end{definition}
%We fix a closed, irreducible, geometric subscheme $T=T(Q_1,\ldots, Q_s,\mu_T) \subseteq \Hilb^{[s]}(X)$. The image of $T$ under the Hilbert-Chow morphism is $\Delta_{\mu_T}$. 
The fully nested Hilbert scheme $\N^s(X)$ parametrises ordered collections of $s$ points in $X$, with the additional data that when $l$ points with labels in the same set in the partition $\a$ come together at $x$, one must specify a length $l$ subscheme supported at $x$. The master blow-up space encodes information on how different subsets of points collide. 

\begin{definition}[Fully nested Hilbert scheme of $\GHilb^{[s]}(X)$]  Let $X$ be a complex nonsingular manifold and $s\ge 1$. \begin{enumerate} 
\item The fully nested Hilbert scheme is $\N^s(X)=\overline{im(h)}$, the closure of the image of the natural map
\[h: \RHilb^{[s]}(X) \to \prod_{\a \in \Pi(s)} \GHilb^{[\a]}(X).\]
The Hilbert-Chow morphism extends to $\N^s(X)$ and gives a morphism $\HC: \N^s(X) \to X^s$, hence obtain the $\a$-punctual locus $\N^\a_0(X)=\HC^{-1}(\Delta_\a)$. The projection $\pi_\a:\N^s(X) \to \GHilb^{[\a]}(X)$ fits into the diagram
\begin{equation}\label{commdiagram}
\xymatrix{\N^s(X) \ar[r]^-{\pi_\a} \ar[d]^{\HC} & \GHilb^{[\a]}(X) \ar[dl]^{\HC_\a} \\
X^s 
}
\end{equation}
\item Let $\a=(\a_1,\ldots, \a_s) \in \Pi(s)$ be a partition. Pull-back along the projection map $\pi_\a: \N^s(X)  \to \GHilb^{[\a]}(X)$ defines an approximating bundle $\pi_\a^*(F^{[\a]})$ over $\N^s(X)$, which satisfies 
 \[\pi_\a^*F^{[\a]}=\pi_{\a_1}^*F^{[[\a_1]]} \oplus \ldots \oplus \pi_{\a_s}^*F^{[[\a_s]]}.\]
We will loosely use the shorthand notation $F^{[\a]}$ for the bundle $\pi_\a^*(F^{[\a]})$. 
\item The punctual part, which sits over the $\a$-diagonal, fits into the diagram 
\begin{equation}\label{commdiagram2}
\xymatrix{\N^\a_0(X) \ar[r]^{\pi_\a} \ar[d]^{\HC} & \GHilb^{[\a]}_0(X) \ar[dl]^{\HC_\a} \\
\Delta_\a 
}
\end{equation}
and  $F^{[\a]}_0=\pi_\a^*(F^{[\a]}|_{\GHilb^{[\a]}_0(X)})$ its restriction to the punctual part $\N^\a_0(X)$. 
\end{enumerate}
\end{definition}

The terminology is self-explanatory: $\N^s(X)$ can be considered as a nested Hilbert scheme, but nested with respect to the full partially ordered net of subsets of $\{1,\ldots, s\}$. Note that $\N^s(X)$ is irreducible, being the closure of the image under $h$ of an irreducible variety. It is a blow-up of $\GHilb^s(X)$ via the natural projection map 
\[\pi_\L: \N^s(X) \to \GHilb^s(X)\] 
where $\L=\{1,\ldots, s\}$ is the trivial partition. However, the geometry of the restriction $\pi_\L: \N^s_0(X) \to \GHilb^s_0(X)$
to the punctual part is more subtle. In particular, the preimage of the curvilinear component $\CHilb^s(X)$ is not necessarily irreducible which is a delicate part of our integration argument, addressed in the next section.  

\begin{definition}\label{def:curvgeometricsubsetQ} The curvilinear part of $\N^s(X)$ is $\CN^s(X)=\pi_\L^{-1}(\CHilb^{[s]}(X))$ where $\L=\{1,\ldots, s\}$ is the trivial partition. The punctual curvilinear component supported at $p\in X$ is 
\[\CN^s_p(X)=\pi_\L^{-1}(\CHilb^{[s]}_p(X))=\CN^s(X) \cap \HC^{-1}(\{p,\ldots, p\}).\]
%\item The curvilinear component is the closure of $h(\CHilb^{[[s]]}(X))$ in $\prod_{\a \in \Pi(s)} \GHilb^{[\a]}(X)$. This is an irreducible component of $CN^s(X)$, and we denote this by $CN^{main}(X)$. The fiber over $p\in X$ is denoted by $CN^{main}_p(X)$.
%\end{enumerate}
\end{definition}

\begin{example} We have seen in \cite{berczitau2} that $CN^3_0(\CC^2)$ is not irreducible, and  it has $2$ components: the curvilinear component $\CN^3_{main}(\CC^2)$ of dimension $1$, and an other component $CN^3_{Por}=\PP^1 \times \PP^1$ sitting over the Porteous point $I=\mathfrak{m}^2 \in \CHilb^3_0(\CC^2)$. 
\end{example}

Despite these extra components, the projection $\pi^\L$ is isomorphism over the curvilinear locus in $\CHilb^{s}(X)$. Recall this is defined as
\[\mathrm{Curv}^s(X)=\{\xi \in \Hilb^s_0(X): \calo_\xi \simeq \CC[t]/t^s\},\]
and it is a dense open subset of $\CHilb^s(X)$. 

Next, we define approximating sets and fully nested Hilbert scheme for geometric subsets. These spaces will all admit a map to the corresponding approximating sets of $\Hilb^{[s]}(X)$.

\begin{definition} For a partition $\mu=(\mu_1,\ldots, \mu_s)$ of $\{1,\ldots, k\}$ and a partition $\a=(\a_1,\ldots, \a_t)$ of $\{1,\ldots, s\}$ let $\mu|\a=(\beta_1,\ldots, \beta_t)$ denote the partition of $\{1,\ldots, k\}$ which we get by merging elements of $\mu$ if they are in the same element of $\a$:
\[\beta_i=\cup_{j \in \a_i} \mu_j \text{ for } 1\le i \le t\]
\end{definition}

\begin{definition}[Approximating sets of geometric subsets] Let $\Hilb^{A_1,\ldots, A_s; \mu}(X) \subset \GHilb^{[k]}(X)$ be a geometric subset with a fixed partition $\mu=(\mu_1,\ldots, \mu_s)$ of the labeled points $\{1,\ldots, k\}$ into the points of the support. 
\begin{enumerate}
\item For any $\a=(\a_1,\ldots \a_t) \in \Pi(s)$ let $f_\a:\Hilb^{[k]}(X) \dasharrow \Hilb^{[\mu|\a]}(X)$ be the natural birational map whose domain consists of points $(Z,(x_i))$ where $x_i \neq x_j$ if $i \nsim_{\mu|\a} j$. We define 
\[\Hilb^{A_1,\ldots, A_s; \mu|\a}(X)=\overline{f_\a(\Hilb^{A_1,\ldots, A_s; \mu}(X))}\]
to be the closure in $\Hilb^{[\mu|\a]}(X)=\prod_{i=1}^t \Hilb^{[\beta_i]}(X)$ where $(\b_1,\ldots, \b_t )=\mu|\a$. 
In short, $\Hilb^{A_1,\ldots, A_s; \mu|\a}(X)$ parametrises tuples $(\xi_1,\ldots, \xi_t)$ where $\xi_i$ is a length $\sum_{j\in \a_i}|\mu_j|$ which records the collision of the subschemes supported at points labeled by $\a_i$. 
%For this to make sense we only take those $\nu_A \in T$ where the support of the subscheme labeled by $A$ is disjoint from the other points. The point is that the supports of the coordinates $\xi_A$ are not necessarily disjoint. 
%\item Let $\N^{A_1,\ldots, A_s; \mu|\a}(X)$  be the closure of the graph of the rational map 
%\[(\supp,f_\a): \Hilb^{A_1,\ldots, A_s; \mu}(X) \dashrightarrow \GHilb^{[\a]}(X) \times \Hilb^{[\mu|\a]}(X)\]
%with the natural projection maps $\pi_\a: \N^{A_1,\ldots, A_s; \a}(X) \to \GHilb^{[\a]}(X)$. 
\item The punctual part sits over the corresponding diagonal:
\[\Hilb^{A_1,\ldots, A_s;\mu|\a}_0(X)=\HC^{-1}(\Delta_{\mu|\a}).\]
\end{enumerate}
\end{definition}

\begin{definition}[Fully nested Hilbert schemes of geometric subsets] The fully nested geometric subset $N^{A_1,\ldots, A_s; \mu}(X)$ is the closure of the graph of the rational map 
\[(\supp,(f_\a)_{\a \in \Pi(s)}): \Hilb^{A_1,\ldots, A_s; \mu}(X) \dashrightarrow \Pi_{\a \in \Pi(s)} \GHilb^{[\a]}(X) \times \Pi_{\a\in \Pi(s)} \Hilb^{[\mu|\a]}(X)\]
endowed with the following natural projections:
\begin{itemize}
\item support map:  $\supp: \N^{A_1,\ldots, A_s; \mu}(X) \to \N^{s}(X)$. 
\item projections to the approximation sets: $\pi_\a: \N^{A_1,\ldots, A_s;\mu}(X) \to \Hilb^{A_1,\ldots, A_s;\mu|\a}(X)$.
\end{itemize}
We will work with a fixed $\mu$, and drop it from the upper index, simply writing $\N^{A_1,\ldots, A_s}(X)$. The $\a$-punctual part sits over the corresponding diagonal:
\[\N^{A_1,\ldots, A_s;\mu|\a}_0(X)=\HC^{-1}(\Delta_{\mu|\a}).\]
\end{definition}

\begin{definition}[Approximating bundles] Let $\a=(\a_1,\ldots, \a_s) \in \Pi(s)$ be a partition. Pull-back along the projection map $\pi_\a: \N^{A_1,\ldots, A_s; \mu}(X)  \to \GHilb^{[\mu|\a]}(X)$ defines an approximating bundle $\pi_\a^*(F^{[\mu|\a]})$ over $\N^{A_1,\ldots, A_s; \mu}(X)$, which satisfies 
 \[\pi_\a^*F^{[\mu|\a]}=\pi_{\mu|\a_1}^*F^{[[\mu|\a_1]]} \oplus \ldots \oplus \pi_{\mu|\a_s}^*F^{[[\mu|\a_s]]}.\]
We will loosely use the shorthand notation $F^{[\mu|\a]}$ for the bundle $\pi_\a^*(F^{[\mu|\a]})$. 
\end{definition}

%\begin{remark}
%Note that all coordinates of $T^0_{\a}$ but the one corresponding to $\alpha$ are nonzero. In fact, $T^0_{\a}$ can be characterised as the set of points in $Q$ with only one nonzero coordinate corresponding to $\alpha$. 
%\end{remark}

\subsection{Fibration of the fully nested Hilbert scheme over the flag manifold}

Recall the curvilinear part of $\N^{k}(X)$, defined as $\CN^{k}(X)=\pi_\L^{-1}(\CHilb^{k}(X))$ where $\L=\{0,1,\ldots, k\}$ is the trivial partition. The punctual curvilinear component supported at $p\in X$ is 
\[\CN^{k}_p(X)=\pi_\L^{-1}(\CHilb^{k}_p(X))=\CN^{k}(X) \cap \HC^{-1}(\{p,\ldots, p\}).\]

We follow the same argument which resulted in diagram \eqref{diagramfour}. Let $\bU \subset TX$ be a small $\GL_X$-invariant tubular neighborhood of the zero section, with the exponential map $\exp: \bU \to X$, which identifies $U_x$ with $\exp(\bU_x) \subset X$. We may assume that $\bU=\GL_X \times_{\GL(n)} U$ for some $U \subset \CC^n$. We define 
\[\widehat{\BN}^{k}(\bU)=\pi_\Lambda^{-1}(\widehat{\BHilb}^{k}(\bU)).\]
%\[\widehat{\GHilb}^{k}(\bU)=\GL_X \times_{\GL(n)} \widetilde{\BHilb}^{k}_0(\CC^n)\]%\left(\GL(n) \times_{P_{n,k}} \BHilb^{k}_\ff(U)\right)\] 
%By adding constants, we have again 
%\[\iota: \grass_k(\sym^{\le k}TX) \subset \grass_{k}(\calo_X\oplus TX \oplus \ldots \oplus \sym^kTX)\]
We get the following extension of diagram \eqref{diagramfour}:
\begin{equation}\label{diagramfive}
\xymatrix{
\widehat{\CN}^{k}(X) \ar@{^{(}->}[r]  \ar[d]^{\pi_\Lambda} & \widehat{\BN}^{k}(\bU) \ar[d]^{\pi_\Lambda}  \ar@{^{(}->}[r]^-{\iota} & \prod\limits_{(\a_1,\ldots, \a_s) \in \Pi(k)} \prod\limits_{i=1}^s \widetilde{\grass}_{|\a_i|}(S^\bullet TX) \ar[d]^{\pi_\Lambda}\\ 
\widehat{\CHilb}^{k}(X) \ar@{^{(}->}[r] \ar[d]^\rho &  \widehat{\BHilb}^{k}(\bU) \ar[ld]  \ar@{^{(}->}[r]^-{\iota} & \widehat{\grass}_{k}(S^\bullet TX) \ar[lld]^{\rho_{\grass}} \\
\flag_{k-1}(TX) \ar[d]^\mu & \\
X &} 
\end{equation}
and $\widehat{\BN}^{k}(\bU) \to \widehat{\BHilb}^{k}(\bU)$ is a birational morphism. When $X=\CC^n$, all vertical maps in diagram \eqref{diagramfive} are $\GL(n)$-equivariant, and we can take $\bU=T\CC^n$. 

Finally, we extend this diagram with the nested geometric subsets. Let $(A_1,\ldots, A_s)$ be a curvilinear regular tuple as in Definition \ref{def:curvilinear} and Definition \ref{def:regular}, hence 
\[\Hilb^{A_1+\ldots +A_s}(X) = \CHilb^k(X) \cap \Hilb^{A_1,\ldots, A_s}(X).\]
Then we can define the partial blow-up $\widehat{\CHilb}^{A_1+\ldots +A_s}(X) \subset \widehat{\CHilb}^{k}(X)$ by sweeping out the fiber over the distinguished flag $\ff$ as before:
\[\rho^{-1}(\ff)=\widehat{\Hilb}^{A_1+\ldots +A_s}(X)_\ff=\widehat{\CHilb}^{k}(X)_\ff \cap \Hilb^{A_1+\ldots +A_s}(X)\]
which parametrises subschemes compatible with the distinguished flag $\ff$. We define 
\[\widehat{\CN}^{A_1+\ldots +A_s}(X)=\pi_{\Lambda}^{-1}(\widehat{\Hilb}^{A_1+\ldots +A_s}(X))\]
and these spaces fit into the diagram

\begin{equation}\label{diagramsix}
\xymatrix{
\widehat{\CN}^{A_1+\ldots +A_s}(X) \ar@{^{(}->}[r]  \ar[d]^{\pi_\Lambda} & \widehat{\CN}^{k}(X) \ar@{^{(}->}[r] \ar[d]^{\pi_{\Lambda}}  & \widehat{\BN}^{k}(\bU) \ar[d]^{\pi_\Lambda}  \\
%\widehat{\CN}^{s}(X) \ar@{^{(}->}[r]  \ar[d]^{\pi_\Lambda} & \widehat{\CN}^{k}(X) & \widehat{\BN}^{k}(\bU) \ar[d]^{\pi_\Lambda} \\ 
\widehat{\Hilb}^{A_1+\ldots +A_s}(X) \ar@{^{(}->}[r] \ar[dr]^\rho & \widehat{\CHilb}^{k}(X) \ar@{^{(}->}[r] \ar[d]^\rho & \widehat{\BHilb}^{k}(\bU) \ar[ld]  \\
& \flag_{k-1}(TX) \ar[d]^\mu & \\
& X &} 
\end{equation}

\section{Equivariant localisation and multidegrees}\label{sec:equiv}

This section is a brief of equivariant cohomology and localisation. For
more details, we refer the reader to Berline--Getzler--Vergne \cite{bgv} and B\'erczi--Szenes \cite{bsz}. 

Let $\kt\cong U(1)^m$ be the maximal compact subgroup of
$T\cong(\CC^*)^m$, and denote by $\mathfrak{t}$ the Lie algebra of $\kt$.  
Identifying $T$ with the group $\CC^n$, we obtain a canonical basis of the weights of $T$:
$\lambda_1,\ldots ,\l_n\in\mathfrak{t}^*$. 

For a manifold $X$ endowed with the action of $\kt$, one can define a
differential $d_\kt$ on the space $S^\bullet \mathfrak{t}^*\otimes
\Omega^\bullet(X)^\kt$ of polynomial functions on $\mathfrak{t}$ with values
in $\kt$-invariant differential forms by the formula:
\[   
[d_\kt\alpha](X) = d(\alpha(X))-\iota(X_X)[\alpha(X)],
\]
where $X\in\mathfrak{t}$, and $\iota(X_X)$ is contraction by the corresponding
vector field on $X$. A homogeneous polynomial of degree $d$ with
values in $r$-forms is placed in degree $2d+r$, and then $d_\kt$ is an
operator of degree 1.  The cohomology of this complex--the so-called equivariant de Rham complex, denoted by $H^\bullet_T(X)$, is called the $T$-equivariant cohomology of $X$. Elements of $H_T^\bullet (X)$ are therefore polynomial functions $\mathfrak{t} \to \Omega^\bullet(X)^K$ and there is an integration (or push-forward map) $\int: H_T^\bullet(X) \to H_T^\bullet(\mathrm{point})=S^\bullet \mathfrak{t}^*$ defined as  
\[(\int_X \alpha)(X)=\int_X \alpha^{[\mathrm{dim}(X)]}(X) \text{ for all } X\in \mathfrak{t}\]
where $\alpha^{[\mathrm{dim}(X)]}$ is the differential-form-top-degree part of $\alpha$. The following proposition is the Atiyah-Bott-Berline-Vergne localisation theorem in the form of \cite{bgv}, Theorem 7.11. 
\begin{theorem}[(Atiyah-Bott \cite{atiyahbott}, Berline-Vergne \cite{berlinevergne})]\label{abbv} Suppose that $X$ is a compact complex manifold and $T$ is a complex torus acting smoothly on $X$, and the fixed point set $X^T$ of the $T$-action on X is finite. Then for any cohomology class $\a \in H_T^\bullet(X)$
\[\int_X \alpha=\sum_{f\in X^T}\frac{\a^{[0]}(f)}{\mathrm{Euler}^T(T_fX)}.\]
Here $\mathrm{Euler}^T(T_fX)$ is the $T$-equivariant Euler class of the tangent space $T_fX$, and $\alpha^{[0]}$ is the differential-form-degree-0 part of $\alpha$. 
\end{theorem}

The right hand side in the localisation formula considered in the fraction field of the polynomial ring of $H_T^\bullet (\mathrm{point})=H^\bullet(BT)=S^\bullet \mathfrak{t}^*$ (see more on details in Atiyah--Bott \cite{atiyahbott} and \cite{bgv}). Part of the statement is that the denominators cancel when the sum is simplified.

\subsection{Equivariant Poincar\'e duals and multidegrees}
\label{subsec:epdmult} 

Restricting the equivariant de Rham complex to compactly supported (or quickly
decreasing at infinity) differential forms, one obtains the compactly
supported equivariant cohomology groups $ H^\bullet_{\kt,\mathrm{cpt}}(X)
$. Clearly $H^\bullet_{\kt,\mathrm{cpt}}(X) $ is a module over
$H^\bullet_\kt(X)$. For the case when $X=W$ is an $N$-dimensional
complex vector space, and the action is linear, one has
$H^\bullet_\kt(W)= S^\bullet\mathfrak{t}^*$ and $ H^\bullet_{\kt,\mathrm{cpt}}(W) $ is
a free module over $H^\bullet_\kt(W)$ generated by a single element of
degree $2N$:
\begin{equation}
  \label{thomg}
   H^\bullet_{\kt,\mathrm{cpt}}(W) = H^\bullet_{\kt}(W)\cdot\Thom_{\kt}(W)
\end{equation}

Fixing coordinates $y_1,\dots,y_N$ on $W$, in which the $T$-action is
diagonal with weights $\eta_1,\ldots,  \eta_N$, one can write an explicit
representative of  $\Thom_{\kt}(W)$ as follows:
\[   \Thom_{\kt}(W) = 
e^{-\sum_{i=1}^N|y_i|^2}\sum_{\sigma\subset\{1,\ldots , N\}}
\prod_{i\in\sigma}\eta_i/2\cdot\prod_{i\notin \sigma}dy_i\,d\bar y_i
\]

We will say that an algebraic variety has dimension $d$ if its
maximal-dimensional irreducible components are of dimension $d$.  A
$T$-invariant algebraic subvariety $\Sigma$ of dimension $d$ in $W$
represents $\kt$-equivariant $2d$-cycle in the sense that
\begin{itemize}
\item a compactly-supported equivariant form $\mu$ of degree $2d$ is
  absolutely integrable over the components of maximal dimension of
  $\Sigma$, and $\int_\Sigma\mu\in S^\bullet \mathfrak{t}$;
\item if $d_\kt\mu=0$, then $\int_\Sigma\mu$ depends only on the class
  of $\mu$ in $ H^\bullet_{\kt,\mathrm{cpt}}(W) $,
\item and $\int_\Sigma\mu=0$  if $\mu=d_\kt\nu$ for a
  compactly-supported equivariant form $\nu$.
\end{itemize}

\begin{definition} \label{defepd} Let $\Sigma$ be an $T$-invariant algebraic
  subvariety of dimension $d$ in the vector space $W$. Then the
  equivariant Poincar\'e dual of $\Sigma$ is the polynomial on $\mathfrak{t}$
  defined by the integral
\begin{equation}
 \label{vergneepd}
 \epd\Sigma = \frac1{(2\pi)^d}\int_\Sigma\Thom_{\kt}(W).
\end{equation}  
\end{definition}
\begin{remark}
  \begin{enumerate}
  \item An immediate consequence of the definition is that for an equivariantly
closed differential form $\mu$ with compact support, we have
\[  \int_\Sigma\mu = \int_W \epd\Sigma\cdot\mu.
\]
This formula serves as the motivation for the term {\em equivariant
  Poincar\'e dual.}
\item This definition naturally extends to the case of an analytic
  subvariety of $\CC^n$  defined in the neighborhood of the origin, or
  more generally, to any $T$-invariant cycle in $\CC^n$.
  \end{enumerate}
\end{remark}

Another terminology for the equivariant Poincar\'e dual is {\em multidegree}, which is close in spirit to the original
construction of Joseph \cite{joseph}. Let  $\Sigma \subset W$ be a $T$-invariant
subvariety. Then we have
\[       \epd{\Sigma,W}_T=\mdeg{I(\Sigma),\CC[y_1,\ldots , y_N]}.
\] 

Some basic properties of the equivariant Poincar\'e dual are listed in \cite{bsz}, these are: Positivity, Additivity, Deformation invariance, Symmetry and a formula for complete intersections. Using these properties one can easily describe an algorithm for
computing $\mdeg{I,S}$ as follows (see Xiller--Sturmfels \cite[\S8.5]{milsturm}, Vergne \cite{voj} and \cite{bsz} for details). %

\subsection{The Rossman formula} \label{subsec:rossman} 

The Rossmann equivariant localisation formula is a variant of the Atiyah-Bott/Berline-Vergne localisation for singular varieties sitting in a smooth ambient space. 
Let $Z$ be a complex manifold with a holomorphic $T$-action, and let
$X\subset Z$ be a $T$-invariant analytic subvariety with an isolated
fixed point $p\in X^T$. Then one can find local analytic coordinates
near $p$, in which the action is linear and diagonal. Using these
coordinates, one can identify a neighborhood of the origin in $\TT_pZ$
with a neighborhood of $p$ in $Z$. We denote by $\tc_pX$ the part of
$\TT_pZ$ which corresponds to $X$ under this identification;
informally, we will call $\tc_pX$ the $T$-invariant {\em tangent cone}
of $X$ at $p$. This tangent cone is not quite canonical: it depends on
the choice of coordinates; the equivariant dual of
$\Sigma=\tc_pX$ in $W=\TT_pZ$, however, does not. Rossmann named this
 the {\em equivariant multiplicity of $X$ in $Z$ at $p$}:
\begin{equation}\label{emult}
   \emu_p[X,Z] \overset{\mathrm{def}}= \epd{\tc_pX,\TT_pZ}.
\end{equation}

\begin{remark}
In the algebraic framework one might need to pass to the {\em
tangent scheme} of $X$ at $p$ (cf. Fulton \cite{fulton}). This is canonically
defined, but we will not use this notion.
\end{remark}
The analog of the Atiyah-Bott formula for singular subvarieties of smooth ambient manifolds is the following statement.
\begin{proposition}[Rossmann's localisation formula \cite{rossmann}]\label{rossman} Let $\mu \in H_T^*(Z)$ be an equivariant class represented by a holomorphic equivariant map $\mathfrak{t} \to\Omega^\bullet(Z)$. Then 
\begin{equation}
  \label{rossform}
  \int_X\mu=\sum_{p\in X^T}\frac{\emu_p[X,Z]}{\mathrm{Euler}^T(\TT_pZ)}\cdot\mu^{[0]}(p),
\end{equation}
where $\mu^{[0]}(p)$ is the differential-form-degree-zero component
of $\mu$ evaluated at $p$.  
\end{proposition}

\section{Proof of Theorem \ref{main1}}

Integration on punctual geometric subsets is based on the Kazarian model, and we use equivariant localisation over the non-associate Hilbert scheme, which is a smooth ambient space for the punctual Hilbert scheme. 

\subsection{Localisation on nonassociate Hilbert schemes}

Let $\frakm \subset \CC^n$ be the maximal ideal at the origin and let $A=\frakm/I$ be a nilpotent algebra of dimension $k-1$. As a vector space, $N=\frakm/I$ can be endowed with several filtrations $N_\bullet$; we fix one of length $r$ and dimension vector $d=(d_1,\ldots, d_r)$. We will work only with natural filtrations of $N$ where automorphisms of the algebra $A$ preserves also the filtration.

Recall form \S \ref{subsec:kazarianmodel} (see \eqref{mrtower}) the nonassociative Hilbert scheme $X_r$ constructed as a tower of Grassmannian bundles
\[X_r \to X_{r-1} \to \ldots \to X_1 \to pt\]
which forms a smooth ambient space of the punctual geometric subset $Q(A)=\Hilb^A(\CC^n)$:
\[Q(A) \subset \Hilb_d(\CC^n) \subset X_r=\tilde{X}_r/\GL_d=\mathrm{Alg}(d,N)/\!/ \GL_d\]
where $\tilde{X}_r$ is an open dense subset (stable set) of the vector space of filtered commutative algebra structures $\mathrm{Alg}(d,N)$ on $N$. This stability condition comes from the surjectivity of a certain map, see \S \ref{subsec:kazarianmodel}. The Hilbert scheme is cut out by the associativity equations:
\[\Hilb_d(\CC^n)=\mathrm{Assoc}(d,N)/\!/\GL_d \subset X_r\]
and $Q(A)$ is cut out by further equations ensuring that the corresponding filtered associative algebra is isomorphic to $A$. The tower $X_r$ comes with a flag of canonical bundles $(D_i \subset V_i)$, see \eqref{mrbundles}. 

Assume that $\a$ is a closed form in the Chern roots of the tautological bundles $D_i/D_{i-1}$ restricted to $Q(A) \subset X_r$. By definiton of the Poincar\'e dual  
\[\int_{Q(A)}\a=\int_{X_r}\a \cdot \epd{Q(A)\subset X_r}\]
holds. Since $X_r=\mathrm{Alg}(d,N)/\!/\GL_d$ is a tower of Grassmannian bundles, $H^*(X_r)$ is generated by the Chern roots of the tautological bundles $D_i/D_{i-1}$ and therefore 
\[\epd{Q(A)\subset X_r}=\mathrm{PD}(z_1,\ldots, z_{k-1})\]
is a polynomial in these. From the GIT description it also clear that in fact 
\[\mathrm{PD}(z_1,\ldots, z_{k-1})=\mathrm{ePD}[\mathrm{Assoc}_A(d,N) \subset \mathrm{Alg}(d,N)]\]
is the equivariant Poincar\'e dual with respect to the maximal torus $(\CC^*)^{k-1} \subset \GL(d)$. The weight of this torus action on the coordinate $q_{i,j}^\ell$ of $\mathrm{Alg}(d,N)$ is 
\[w(z_{i,j}^\ell)=z_i+z_j-z_\ell.\]
and therefore $\mathrm{PD}(z_1,\ldots, z_{k-1})$ is in fact a polynomial in the linear forms $z_i+z_j-z_\ell$ with $w(i)+w(j)\le w(\ell)$.

We use the residual $\GL_n$ action to develop equivariant localisation on $X_r$. The $\GL_n$ action on $X_1=\grass_{d_1}(\CC^n)$ induces a natural $\GL_n$ action on $X_r$ and the displayed embeddings of $Q(A)$ and $\Hilb_d(\CC^n)$ are equivariant with respect to this action. Various alternative methods are known for the computation of the Gysin homomorphism, in particular, Kazarian \cite{kazarian} uses the machinery developed in \cite{kazarian2}. Although the final residue formula is the same, it is crucial to us to apply equivariant localisation in order to study fixed point contributions and to derive the vanishing theorem.

\subsection{From Atiyah-Bott to residues on $X_1=\grass_{d_1}(\CC^n)$}\label{subsec:m1}

Let $\l_1,\ldots, \l_n \in \mathfrak{t}^*$ denote the weights of the diagonal torus $T^{n}\subset \GL_n$ action on $\CC^n$ with respect to the basis $\{e_1,\ldots, e_n\}$. This induces an action on $X_1=\grass_{d_1}(\CC^d)$. Fixed points are coordinate subspaces $\mathrm{Span}(e_{\s(1)},\ldots ,e_{\s(d_1)})$ indexed by pemutations $\s \in \sg n/(\sg {d_1} \times \sg {n-d_1})$. Let $\a$ be a Chern polynomial of the tautological rank $d_1$ bundle over $\grass_{d_1}(\CC^n)$ The Atiyah-Bott localisation gives us 
\begin{equation}
\int_{\grass_{d_1}(\CC^n)} \a=\sum_{\sigma\in\sg n/(\sg {d_1} \times \sg {n-d_1})}
\frac{Q(\lambda_{\sigma(1)},\ldots ,\lambda_{\sigma(d_1)})}
{\prod_{1\leq m\leq d_1}\prod_{i=d_1+1}^n(\lambda_{\sigma\cdot
    i}-\lambda_{\sigma\cdot m})}
\end{equation}

The right hand side can be transformed into an iterated residue. To describe this formula, we recall the notion of an {\em iterated
  residue} (cf. e.g. \cite{szenes}) at infinity.  Let
$\omega_1,\dots,\omega_N$ be affine linear forms on $\CC^k$; denoting
the coordinates by $z_1,\ldots, z_k$, this means that we can write
$\omega_i=a_i^0+a_i^1z_1+\ldots + a_i^kz_k$. We will use the shorthand
$h(\bz)$ for a function $h(z_1\ldots z_k)$, and $\dbz$ for the
holomorphic $n$-form $dz_1\wedge\dots\wedge dz_k$. Now, let $h(\bz)$
be an entire function, and define the {\em iterated residue at infinity}
as follows:
\begin{equation}
  \label{defresinf}
 \ires \frac{h(\bz)\,\dbz}{\prod_{i=1}^N\omega_i}
  \overset{\mathrm{def}}=\left(\frac1{2\pi i}\right)^k
\int_{|z_1|=R_1}\ldots
\int_{|z_k|=R_k}\frac{h(\bz)\,\dbz}{\prod_{i=1}^N\omega_i},
 \end{equation}
 where $1\ll R_1 \ll \ldots \ll R_k$. The torus $\{|z_m|=R_m;\;m=1 \ldots
 k\}$ is oriented in such a way that $\res_{z_1=\infty}\ldots
 \res_{z_k=\infty}\dbz/(z_1\cdots z_k)=(-1)^k$.
We will also use the following simplified notation: $\sires \overset{\mathrm{def}}=\ires.$

In practice, one way to compute the iterated residue \eqref{defresinf} is the following algorithm: for each $i$, use the expansion
 \begin{equation}
   \label{omegaexp}
 \frac1{\omega_i}=\sum_{j=0}^\infty(-1)^j\frac{(a^{0}_i+a^1_iz_1+\ldots
   +a_{i}^{q(i)-1}z_{q(i)-1})^j}{(a_i^{q(i)}z_{q(i)})^{j+1}},
   \end{equation}
   where $q(i)$ is the largest value of $m$ for which $a_i^m\neq0$,
   then multiply the product of these expressions with $(-1)^kh(z_1\ldots
   z_k)$, and then take the coefficient of $z_1^{-1} \ldots z_k^{-1}$
   in the resulting Laurent series.

\begin{proposition}[{\rm B\'erczi--Szenes \cite{bsz}, Proposition 5.4}]\label{prop:ABtoresidue} For any homogeneous polynomial $\a(\bz)$ on $\CC^d$ we have
\begin{equation}\label{flagres}
\sum_{\sigma\in\sg n/\sg{d}. \times \sg{n-d}}
\frac{\a(\lambda_{\sigma(1)},\ldots ,\lambda_{\sigma(d)})}
{\prod_{1\leq m\leq d}\prod_{i=d+1}^n(\lambda_{\sigma\cdot
    i}-\lambda_{\sigma\cdot m})}=\sires
\frac{\prod_{1\leq m, l\leq d}(z_m-z_l)\,\a(\bz)\dbz}
{\prod_{l=1}^d\prod_{i=1}^n(\lambda_i-z_l)}.
\end{equation}
Moreover, the terms on the left hand side correspond to torus fixed points which in turn correspond to poles of the rational expression on the right hand side. 
\end{proposition}

\begin{remark}
  Changing the order of the variables in iterated residues, usually,
  changes the result. In this case, however, because all the poles are
  normal crossing, formula \eqref{flagres} remains true no matter in
  what order we take the iterated residues.
\end{remark}
 
\subsection{From Atiyah-Bott to residues on $X_2$}\label{subsec:m2}

Recall that $D_1$ denotes the tautological rank $d_1$ bundle over $X_1$. Next we take $X_2=\grass_{d_2}((\CC^n \oplus S_2)/D_1)$ where $S_2=\Sym^2 D_1 \subset \Sym^2 D$ is a rank $d_1^2$ bundle over $X_1$. The $T^n$-weights of $(\CC^n \oplus S_2)/D_1$ over the fixed point $\s=\mathrm{Span}(e_{\s(1)},\ldots, e_{\s({d_1})}) \in X_1$ are 
\[\{\underbrace{\l_1,\ldots, \l_n}_{\CC^n}\}  \cup \{\underbrace{\l_{\s(i)}+\l_{\s(j)}: 1\le i\le j \le d}_{D_1^2}\} \setminus \{\underbrace{\l_{\s(1)},\ldots ,\l_{\s(d)}}_{D_1}\}.\]
In the residue formula \eqref{flagres} this fixed point corresponds to the pole where $\l_{\s(i)}=z_i$ for $i=1,\ldots, d$. Therefore applying Proposition \ref{prop:ABtoresidue} fibrewise we get the following residue formula for $\int_{X_2}\a$ where 
\[\a=\a(t_1,\ldots, t_{d_1},u_1,\ldots, u_{d_2})\]
is a polynomial in the Chern roots of the tautological bundles $D_1, D_2/D_1$ over $X_2$:
\[\int_{X_2}\a=\sires \siresw
\frac{\overbrace{\prod (z_m-w_l)}^{D_1} \prod_{1\leq m, l\leq d_2}(w_m-w_l)\prod_{1\leq m, l\leq d_1}(z_m-z_l)\,\a(\bz,\bw)\dbw \dbz}
{\underbrace{\prod_{l=1}^{d_2} \prod_{1\le i\le j \le d_1}(z_i+z_j-w_l)}_{D_1^2} \underbrace{\prod_{l=1}^{d_2}\prod_{i=1}^n(\lambda_i-w_l)}_{\CC^n} \prod_{l=1}^{d_1}\prod_{i=1}^n(\lambda_i-z_l)}.\]
We briefly indicated with the underbrackets which terms belong to which weight spaces. With the substitution $z_{d_1+i}=w_i$ we can rewrite this as 
\[\int_{X_2}\a=\sires 
\frac{\prod_{1\leq m, l\leq d_1+d_2}(z_m-z_l)\,\a(\bz)\dbz}
{\prod_{l=d_1+1}^{d_2} \prod_{1\le i\le j \le d_1}(z_i+z_j-z_l) \prod_{l=1}^{d_1+d_2}\prod_{i=1}^n(\lambda_i-z_l)}\]

\begin{remark} The iterated residue on the right hand side is an integral on the contour $\calc$ defined by the inequalities $|z_i| \ll |z_j|$ for $i<j$. In particular, this means that the poles corresponding to the vanishing of the linear forms $z_i+z_j-z_l$ do not contribute to the residue, they sit outside $\calc$. In other words, when we expand the rational expression on the contour $\calc$ the coefficient of $(z_1 \ldots z_{d_1+d_2})^{-1}$ is equal to the some of the pole contributions for those poles where $z_i=\l_{\s(i)}$ for some $\s \in \sg n$.
\end{remark}

\subsection{Localisation on $X_r$} The same argument works for any $r$. For an integer $i$ such that $d_1+ \ldots + d_{k-1}< i \le d_1+ \ldots +d_k$, we define its weight to be $w(i)=k$. The weights sequences has the form
\[\bw=(\underbrace{1,\ldots, 1}_{d_1},\underbrace{2,\ldots, 2}_{d_2},\ldots, \underbrace{r,\ldots, r}_{d_r})\]

\begin{proposition}\label{prop:mr} Let $d_1+\ldots +d_r=k-1$. Then 
\begin{equation}\label{mr} \int_{X_r}\a=\sires 
\frac{\prod_{\substack{1\leq m, l\leq k-1\\ w(m)<w(l)}}(z_m-z_l)\,\a(\bz)\dbz}
{\prod_{\substack{1\le i \le j <l\le k-1\\ w(i)+w(j)\le w(l)}}(z_i+z_j-z_l) \prod_{l=1}^{k-1}\prod_{i=1}^n(\lambda_i-z_l)}
\end{equation}
 The iterated residue on the right hand side is an integral on the contour $\calc$ defined by the inequalities $|z_i| \ll |z_j|$ for $i<j$. %In particular, this means that the poles corresponding to the vanishing of the linear forms $z_i+z_j-z_l$ do not contribute to the residue. When we expand the rational expression on the contour $\calc$ the coefficient of $(z_1 \ldots z_{d_1+d_2})^{-1}$ is equal to the sum of the pole contributions for those poles where $z_i=\l_{\s(i)}$ for some $\s \in \sg n$.
\end{proposition}

\begin{proof} We take $X_r=\grass_{d_r}((\CC^n \oplus S_r)/D_{r-1})$ where $S_r=\oplus_{i+j\le r}D_i \otimes D_j \subset \Sym^2 D_{r-1}\subset \Sym^2 D$ is a rank $\sum_{i+j\le r}d_id_j$ bundle over $X_{r-1}$. 

In the residue formula for $r-1$ the $T^n$-weights on $D_i/D_{i-1}$ are $z_{d_1+\ldots +d_{i-1}+1},\ldots z_{d_1+\ldots +d_{i-1}+d_i}$. The $T^n$-weights of $(\CC^n \oplus S_r)/D_{r-1}$ over the fixed point $(z_1,\ldots, z_{d_1+\ldots +d_{r-1}}) \in X_{r-1}$ are 
\[\{\underbrace{\l_1,\ldots, \l_n}_{\CC^n}\}  \cup \{\underbrace{z_i+z_j: w(i)+w(j) \le w(r)}_{D_i \otimes D_j}\} \setminus \{\underbrace{z_{d_1+\ldots +d_{r-1}+1},\ldots z_{d_1+\ldots +d_{r-1}+d_r}}_{D_{r-1}}\}.\]
Therefore applying Proposition \ref{prop:ABtoresidue} fibrewise again, we get the desired residue formula \eqref{mr} for $\int_{X_r}\a$ where 
\[\a=\a(t_1,\ldots, t_{d_1+\ldots+d_r})\]
is a polynomial in the Chern roots of the tautological bundles $D_1, D_2/D_1,\ldots, D_r/D_{r-1}$ over $X_r$.
\end{proof}

Note that by definition of the equivariant Pointcar\'e dual and Definition \ref{defxr} we have 
\[\int_{\Hilb^A(\CC^n)} \Phi(V^{[k]}) =\int_{X_k} \Phi(V^{[k]}) \cdot \epd{\mathrm{Assoc}_A(d,N_\bullet),\mathrm{Alg}_d(N_\bullet)}\]
and hence Theorem \ref{main1} follows from Proposition \ref{prop:mr}.
We also proved the following geometric feature of the localisation process on $X_r$, which is  summarised as

\begin{proposition}[\textbf{The Vanishing Theorem}] Let $I \subset \mathfrak{m} \subset \CC[x_1,\ldots, x_n]$ be an ideal which defines the nilpotent complex algebra $A=\frakm/I$ of dimension $k-1$ and the corresponding punctual geometric subset
\[P(A)=\overline{\{\xi \in \Hilb^k_0(\CC^n): \calo_\xi \simeq A\}}\]
Let $X_r=\tilde{X}_r/\!/\GL_d$ denote the nonassociative Hilbert scheme corresponding to $A$ where 
\[\tilde{X}_r=\{\psi_1 \oplus \psi_2 \in \Hom((\CC^n)^* \oplus \Sym^2 N,N): \psi_2(N_i \otimes N_j) \subset N_{i+j}, \psi_1 \oplus \psi_2 \text{ is surjective}.\}\]
Let $\a=\a(t_1,\ldots, t_{k-1})$ be a Chern polynomial of the tautological bundles over $X_r$ and write write the integral 
\[\int_{X_r}\a=\sum_{(\psi_1,\psi_2)\in X_r^{T_n}} \mathrm{AB}_{(\psi_1,\psi_2)}\] 
as the sum of fixed point (Atiyah-Bott) contributions. Then 
\[\sum_{(\psi_1,\psi_2)\in X_r^{T_n}} \mathrm{AB}_{(\psi_1,\psi_2)}=\sum_{\substack{(\psi_1,\psi_2)\in X_r^{T_n}\\ \psi_2=0}} \mathrm{AB}_{(\psi_1,\psi_2)}.\]

\end{proposition}

\section{Proof of Theorem \ref{main2}}\label{sec:prooftauintegrals}

Going to multipoint-supported geometric subsets comes with deep technical difficulties. The strategy is to reduce integration to the cirvilinear geoemtric subset by combining a) a stability trick with b) a sieve method, then c) apply equivariant localisation and finally d) prove a deep vanishing theorem for iterated residues. 

\subsection{The stability trick} Assume $X$ is a complex manifold, and in fact, we can assume that $X=\CC^n$, due to the reduction to equivariant integrals below. Let $V$ be a rank-r equivariant bundle over $X=\CC^n$, and $\Phi(V^{[k]})$ an Chern polynomial in the equivariant Chern classes of the tautological bundle. Let $f: X\to X$ be a stable map, see \cite{berczitau2} for definition. Stable maps are dense in the space of holomorphic maps, and we will pick one in the homotopy class of the identity if that exists, otherwise we approximate the identity map with stable maps. The map $f$ induces the $k$th Hilbert extension map
\[\mathrm{hf}^{[k]}: \GHilb^{k}(X) \to \GHilb^{k}(X \times X)\]
which sends the subscheme $\xi_I$ to $\xi_{(I,I_{\Gamma(f)})}$ where $I_{\Gamma(f)}$ is the ideal of the graph $\Gamma(f)$ of $f$. $\mathrm{hf}^{[k]}$ is a regular embedding, let 
$\GHilb^k(\Gamma(f))=\im(\mathrm{hf}^{[k]})$ denote its image, which is the Hilbert scheme of the graph. Let $\pi^{[k]}: \GHilb^k(\Gamma(f)) \to \GHilb^k(X), \xi_{(I,I_{\Gamma(f)})} \mapsto \xi_I$ be the inverse of $\mathrm{hf}^{[k]}$.
The normal bundle of $\GHilb^k(\Gamma(f))$ in $\GHilb^k(X \times X)$ is $(\pi^{[k]})^*(f^*TX)^{[k]}$, whereas the normal bundle of $\GHilb^k(X)$ in $\GHilb^k(X \times X)$ is $TX^{[k]}$, and $\pi^{[k]}$ extends to an isomorphism of a small neighborhood $\pi^{[k]}_\nabla: \GHilb^k_{\nabla}(\Gamma(f)) \to \GHilb^k_\nabla(X)$ with Thom classes $\Thom(\Gamma(f))$ and $\Thom(X)$ respectively, such that $(\pi^{[k]}_\nabla)^*(\Thom(X))=\Thom(\Gamma(f))$.  we have 
\begin{equation}
\int_{\GHilb^k(X)}\Phi(V^{[k]})=\int_{\GHilb^k_\nabla(X)}\Phi(V^{[k]}) \cdot \Thom(X)=
\int_{\GHilb^k_\nabla(X)}\Phi(V^{[k]}) \cdot \pi^{[k]}_{\nabla *}\Thom(\Gamma(f))
\end{equation}
where 
\begin{equation}\label{thom}
\Thom(\Gamma(f))|_{\GHilb^k(\Gamma(f))}=\pi^{[k]*}\Euler((f^*TX)^{[k]})
\end{equation}
Recall the $f$-Hilbert scheme 
\[\GHilb^k(f)=\overline{\{\xi=\xi_1 \sqcup \ldots \sqcup \xi_k \in \GHilb^k(X): f(\xi_1)=\ldots =f(\xi_s) \in X\}}\]
as the set of subschemes supported on the fibers of $f$. 
\begin{proposition}[\cite{bsz2} Corollary 12.14] Let $f:X \to X$ be a stable Thom Boardman map. Then there is an embedding $\varphi_f: TX \hookrightarrow (f^*TX)^{[k]}$ such that the $k-1$-jet of $f$ induces a section $s_f$ of $f^*TX^{[k]}/\varphi_f(TX)$ presenting $\GHilb^k(f)$ as a local complete intersection.   
\[[\GHilb^k(f)]=\mathrm{Euler}(f^*TX^{[k]}/\varphi_f(TX))\]
\end{proposition}
This implies that the support of the Euler class sits in the $f$-Hilbert scheme:
\[\mathrm{supp}(\Euler(f^*TX)^{[k]}) \subset \GHilb^k(f).\]
and hence 
\begin{corollary}\label{cor:mostimportant} The class $\Phi(V^{[k]}) \cdot \Thom(X)$ can be represented by a form $\omega_{V,f}$ whose support satisfies
\[\mathrm{supp}(\omega_{V,f}|_{\GHilb^k(X)}) \subset \GHilb^k(f).\]
and hence 
\[\mathrm{supp}(\omega_{V,f}|_{\Hilb^{A_1,\ldots, A_s,\mu}(X)}) \subset \GHilb^k(f) \cap \Hilb^{A_1,\ldots, A_s,\mu}(X).\]
\end{corollary}
We define via Thom isomorphism the class 
\[\Phi(V^{[k]})_f=\frac{\omega_{V,f}|_{\GHilb^k(X)}}{\Euler(TX^{[k]}} \in \Omega^\bullet(\GHilb^k(X)),\]
such that 
\begin{equation}\label{stabilitytrick}
\int_{\Hilb^{A_1, \ldots A_s,\mu}(X)}\Phi(V^{[k]})=\int_{\Hilb^{A_1,\ldots, A_s;\mu}(X)} \Phi(V^{[k]})_f
\end{equation}

In short, any stable Thom-Boarman map $f$ sufficiently close to the identity defines a deformation $\Phi(V^{[k]})_f$ of $\Phi(V^{[k]})$ whose support has better geometric properties:
\begin{proposition}{\cite{berczitau2}}\label{prop:support} For stable Thom-Boardman map $f$ 
\[\GHilb^k(f) \cap \GHilb^k_0(X) \subset \CHilb^k(X),\]
that is, the punctual $k$-fold locus of a stable map sits in the curvilinear component, hence 
\begin{equation}\label{support1}
\mathrm{supp}(\Phi_f)\cap \GHilb^k_0(X) \subseteq \mathrm{supp}(\Euler((f^*TX)^{[k]}))\cap \GHilb^k_0(X) \subset \CHilb^k(X)
\end{equation}
\end{proposition}
%Moreover, by Proposition \ref{prop:hilbfloc},  $\GHilb^k(f)$ is locally irreducible at the curvilinear points, and locally isomorphic to the Haiman bundle $B$.
In short, $\Phi_f$ is properly supported in the following sense:
\begin{definition}\label{geocond} We say that the form $\Phi \in \Omega^\bullet(\GHilb^k(X))$ is properly supported if $\supp(\Phi)$  intersects the punctual Hilbert scheme only in the curvilinear component:
\begin{equation}\label{supportcond}
\mathrm{supp}(\Phi) \cap \GHilb^k_0(X) \subset \CHilb^k_0(X).
\end{equation}
\end{definition}

\subsection{Sieve on fully nested Hilbert schemes}\label{subsec:sieve}
Let $V$ be a rank $r$ vector bundle over the complex manifold $X$ of dimension $m$ with Chern roots $\theta_1,\ldots, \theta_r$, and $\Phi(c_1,\ldots ,c_{kr})$ be a Chern polynomial in the Chern roots of the tautological bundle $V^{[k]}$ over $\Hilb^{k}(X)$. The branched cover $\kappa: \GHilb^{[k]}(X) \to \GHilb^{k}(X)$ gives 
\begin{equation}\label{n!}
n! \int_{\GHilb^{k}(X)}\Phi = \int_{\GHilb^{[k]}(X)} \kappa^*\Phi
\end{equation}
one can work over ordered Hilbert schemes, and we keep the notation $\Phi$ for the pulled-back form. 

Let $(A_1,\ldots, A_s)$ be a regular tuple and fix a corresponding partition $\mu$ so that $\Hilb^{A_1,\ldots, A_s; \mu}(X)$ is the corresponding geometric subset of the ordered Hilbert scheme where the labeled points are arranged among the support according to $\mu$. The first step in our strategy, similarly to \cite{berczitau2}, is to pull-pack integration to the fully nested Hilbert scheme $N^{A_1,\ldots, A_s;\mu}(X)$, which admits plenty of approximating bundles (coming from the different factors) to build linear combinations with punctual supports. According to \cite{junli,rennemo}, these tautological bundles can be combined into a sieve formula,  which decomposes $\Phi$ as a sum 
\[\Phi=\sum_{\a \in \Pi(s)}\Phi^\a\]
of forms indexed by partitions of $\{1, \ldots s\}$. For any partition $\{1,\ldots, s\}=\a_1 \sqcup \ldots \sqcup \a_t$ the form $\Phi^\a$ is supported on the approximating punctual subset  
\[\supp(\Phi^\a)=N^{[\mu|\a]}(X)=\HC^{-1}(\Delta_{\mu|\a})\]
and $\Phi^\a$ is a linear combination of the classes $\Phi(V^{[\b]})$ for those partitions $\b \in \Pi(s)$ which are refinements of $\a$, that is, $\b \le \a$ holds.  Recall that $\mu|\a=(\beta_1,\ldots \b_t) \in \Pi(k)$ is the partition of $\{1,\ldots, k\}$ obtained by merging labels in $\mu_i$ and $\mu_j$ iff $i,j$ sit in the same $\a_l$:
\[\b_i=\cup_{j\in \a_i} \mu_j \text{ for } 1\le i \le t.\]

\begin{definition} Let $\a \in \Pi(s)$ be a partition and $\Phi$ a homogeneous symmetric polynomial in the Chern roots of $V^{[k]}$. Define the class $\Phi^{\mu|\a} \in H^*(N^{k}(X))$ inductively by putting $\Phi^{\mu|([1],\ldots [s])}=\Phi(V^{[\mu_1],\ldots [\mu_s]})=\Phi(V^{[\mu_1]}\oplus \ldots \oplus V^{[\mu_s]})$ and for $\a>([1],\ldots [s])$  
\[\Phi^{\mu|\a}=\Phi(V^{[\mu|\a]})-\sum_{\b < \a}\Phi^{\mu|\b}.\]
\end{definition}

\begin{remark}\label{philambda}
Let $\Lambda=[1,\ldots, s]$ be the trivial partition. The formula above gives us 
\[\Phi^{\mu|\Lambda}(V^{[k]})=\sum_{\b \in \Pi(s)}(-1)^{|\b|-1}(|\b|-1)!\Phi(V^{\mu|\b}).\]
\end{remark}

\begin{proposition}\label{thm:support} The restriction of $\Phi^{\mu|\a}$ to $\N^{A_1,\ldots, A_s;\mu}(X) \setminus \N^{[\mu|\a]}_0(X)$ vanishes, that is, the support of $\Phi^{\mu|\a}$ in the geometric subset $\N^{A_1,\ldots, A_s;\mu}(X)$ is the diagonal part $\N^{[\mu|\a]}_0(X)=\HC^{-1}(\Delta_{\mu|\a})\subset \N^{k}(X)$.
\end{proposition}

\begin{proof} This follows from easy inclusion-exclusion and induction, for the proof see \cite{rennemo}.
\end{proof}

In particular, the support of $\Phi^{\mu|[1],\ldots, [s]}$ is the full geometric subset $\N^{A_1,\ldots, A_s;\mu}(X)$ and the support of $\Phi^{[1,\ldots, s]}$ is the full punctual geometric subset 
\[\N^{A_1,\ldots, A_s;\mu}_0(X)=\pi_\Lambda^{-1}(\Hilb^{k}_0(X))=\N^{A_1,\ldots, A_s;\mu}(X) \cap \Hilb^{[k]}_0(X).\]
 The curvilinear Hilbert scheme and curvilinear fully nested Hilbert scheme sit in the punctual part: 
 \begin{equation}\label{diagramseven}
\xymatrix{
\CN^k(X) \ar[d]^{\pi_\Lambda}  & \CN^{A_1+\ldots +A_s}(X) \ar@{^{(}->}[r] \ar@{^{(}->}[l] \ar[d]^{\pi_\Lambda} & \N^{A_1,\ldots ,A_s;\mu}_0(X) \ar@{^{(}->}[r] \ar[d]^{\pi_{\Lambda}}  & \N^{k}_0(X) \ar[d]^{\pi_\Lambda}  \\  
\CHilb^{[k]}(X)  & \Hilb^{A_1+\ldots +A_s}(X) \ar@{^{(}->}[r]^{\text{irr. comp.}} \ar@{^{(}->}[l] & \Hilb^{A_1,\ldots, A_s;\mu}_0(X) \ar@{^{(}->}[r]  & \Hilb^{[k]}_0(X)}
\end{equation}
Note that here $\Hilb^{A_1+\ldots +A_s}(X)$ is an irreducible component of $\Hilb^{A_1,\ldots, A_s;\mu}_0(X)$ and $\CHilb^{[k]}(X)$ is an irreducible component of $\Hilb^{[k]}_0(X)$, but $\CN^{A_1+\ldots +A_s}(X)$ might have several components in $\N^{A_1,\ldots ,A_s;\mu}_0(X)$. 

\begin{lemma}\label{lemma:support} Let  $\Phi$ be a properly supported Chern polynomial in the sense of Definition \ref{geocond}. Then 
\begin{enumerate}
\item  $\pi_\Lambda^*\Phi$ is also properly supported, that is, it is represented by a form whose support intersects the punctual part $\N^{k}_0(X)$ only in the curvilinear part $\CN^{k}(X)=\pi_\L^{-1}(\CHilb^{k}(X))$. 
\item More generally, for a partition $\a=(\a_1,\ldots, \a_t) \in \Pi(s)$ of the support the form $\Phi^{\mu|\a}$ is represented by a form whose support intersects the punctual part $\N^{k}_0(X)$ only in the curvilinear part
\[\CN^{A_1+\ldots+ A_s;\mu|\a}(X)=\pi_\a^{-1}(\CHilb^{A_1,\ldots, A_s;\mu|\a_1}(X)\times \ldots \times \CHilb^{A_1,\ldots, A_s;\mu|\a_s]}(X))\]
When $(A_1,\ldots, A_s)$ is a curvilinear regular tuple as in Definition \ref{def:curvilinear} and Definition \ref{def:regular}, and $A_{\a_i}=\sum_{j\in \a_i} A_j$ stands for $1\le i \le t$ then  
\[\CN^{A_1+\ldots+ A_s;\mu|\a}(X)=\pi_{\a_1}^{-1}(\CHilb^{A_{\a_1}}(X)) \times \ldots \times \pi_{\a_t}^{-1}(\CHilb^{A_{\a_t}}(X)).\]
\end{enumerate}
\end{lemma}

\begin{proof}
If $\omega$ is a properly supported form representing $\Phi$, then $\pi^*\omega$ represents $\pi_\Lambda^*\Phi$, and the statement follows from the fact that support of the pull-back form under a proper map is equal to the pre-image of the support.  
\end{proof}

The first step in proving Theorem \ref{main2} is to pull back the integral over the fully nested Hilbert scheme, and apply the sieve formula:
\begin{equation}\label{motivic1}
\int_{\Hilb^{A_1,\ldots, A_s,\mu}(X)}\Phi(V^{[k]})=\int_{\N^{A_1,\ldots, A_s;\mu}(X)}\pi^* \Phi=\sum_{\a \in \Pi(s)} \int_{\N^{A_1,\ldots, A_s;\mu}(X)} \Phi^{\mu|\a}
\end{equation}

\subsection{Reduction to equivariant integration on $X=\CC^n$}\label{subsec:reduction}

The sieve formula \eqref{motivic1} reduces integration over $\Hilb^{A_1,\ldots, A_s;\mu}(X) \subset \GHilb^k(X)$ to small neighborhoods of the punctual nested Hilbert scheme sitting over various diagonals under the Hilbert-Chow morphism.    
Take a small neighborhood $\widehat{\N}^{k}(\bU)$ of $\widehat{\CN}^{k}(X)$ as in diagram \eqref{diagramfive}, which fibers over $\flag_{k-1}(TX)$ and hence fibers over $X$. This bundle can be pulled back along a classifying map $\tau: X \to \mathrm{BGL}(m)$ from the universal bundle 
\[\mathbf{E}=\mathrm{BGL}(n) \times_{\GL(n)} \widehat{\BN}^{k}(\CC^n)\]
where with the notations of diagram \eqref{diagramfive} 
\[\widehat{\BN}^{k}(\CC^n)=(\mu \circ \rho \circ \pi_\Lambda)^{-1}(0)=\pi_\Lambda^{-1}(\widehat{\BHilb}^{k}(\CC^n))\]
is the fiber over the origin $0\in X=\CC^n$, which is the balanced fully nested Hilbert scheme supported at the origin of $\CC^n$. We get a commutative diagram 
\begin{equation*}
\xymatrix{\widehat{\CN}^{k}(X) \ar[rd]^\pi \ar@{^{(}->}[r] & \widehat{\BN}^{k}(\bU) \ar[r] \ar[d]^{\pi} & \mathbf{E}  \ar[d] \\
 & X \ar[r]^{\tau}  &  \mathrm{BGL}(n) } 
\end{equation*}
which induces a diagram of cohomology maps
\begin{equation*}
\xymatrix{H^*(\widehat{\BN}^{k}(\bU))  \ar[d]^{\pi_*} &  H^*(\mathbf{E}) \ar[l] \ar[d]^{\res} \\
 H^*(X)   &  H^*(\mathrm{BGL(n)}) \ar[l]^{\mathrm{Sub}}} 
\end{equation*}
Here 
\begin{itemize}
\item $\res$ is the equivariant push-forward (integration) map along the fiber $\BN^{k}(\CC^n)$, and in the next section we develop an iterated residue formula derived from equivariant localisation.
\item $\mathrm{Sub}$ is the Chern-Weil map, which is the substitution of the Chern roots of $X$ into the generators $\lambda_1,\ldots, \lambda_n$ of $H^*(\mathrm{BGL}(n))=H_{\GL(n)}^*(pt)=\QQ[\l_1,\ldots, \l_n]$,
\end{itemize}

Commutativity of the diagram tells us that for any form $\omega$ supported on some neighborhood $\widehat{\BN}^{k}(\bU)$ we have  
\[\int_{\N^{k}(X)}\omega =\int_{X} \int_{\widehat{\BN}^{k}(\CC^n)} \omega |_{\{\l_1,\ldots, \l_n\} \to \text{Chern roots of } TX}\]
We apply this formula for the form $\Phi^\a$ coming from the sieve. According to Proposition \ref{thm:support} and Lemma \ref{lemma:support} for a partition $\a=(\a_1,\ldots, \a_t) \in \Pi(s)$ of the support the form $\Phi^{\mu|\a}$ is compactly supported in a neighborhood of the curvilinear $\mu|\a$ locus
\[\CN^{A_1+\ldots+ A_s;\mu|\a}(X)=\pi_{\a_1}^{-1}(\CHilb^{A_{\a_1}}(X)) \times \ldots \times \pi_{\a_t}^{-1}(\CHilb^{A_{\a_t}}(X))\]
and hence
\begin{equation}\label{reduceintegraltoaffinespace}
\int_{\Hilb^{A_1,\ldots, A_s;\mu}(X)}\Phi(V^{[k]})=\int_X \sum_{\a \in \Pi(s)} \int_{\widehat{\BN}^{A_1,\ldots, A_s;\mu|\a}(\CC^n)} \Phi^{\mu|\a} |_{\{\l_1,\ldots, \l_n\} \to \text{Chern roots of TX}}
\end{equation}
where 
\[\widehat{\BN}^{A_1,\ldots, A_s;\mu|\a}(\CC^n)=\widehat{\N}^{A_1,\ldots, A_s;\mu}(\CC^n) \cap \widehat{\BN}^{\mu|\a}(\CC^n)\]
where 
\[\widehat{\BN}^{\mu|\a}(\CC^n)=\widehat{\BN}^{\b_1}(\CC^n) \times \ldots \times \widehat{\BN}^{\b_t}(\CC^n))\]
with the notation $\mu|\a=(\b_1,\ldots \b_t)$. 
$\widehat{\BN}^{A_1,\ldots, A_s;\mu|\a}(\CC^n)$ is a small (smooth) neighborhood of $\CN^{A_1+\ldots+ A_s;\mu|\a}(X)$ in $\widehat{\N}^{A_1,\ldots, A_s;\mu}(\CC^n)$, and $\Phi^{\mu|\a}$ is supported in this neighborhood. 

Due to the product form of $\widehat{\BN}^{\mu|\a}(\CC^n)$, it is enough to study the deepest term $\a=\Lambda=[1,\ldots, s]$, where we used the simplified notation 
\[\widehat{\BN}^{\mu|\Lambda}(\CC^n)=\widehat{\BN}^{k}(\CC^n)\]
% \text{ and } \widehat{\CN}^{\mu|\Lambda}(\CC^n)=\widehat{\CN}^{A_1,\ldots, A_s;\mu|\Lambda}(\CC^n)\] 

%\begin{remark}
%The bundle $\Phi^{\mu_\Lambda}(V^{[k]})$ is defined over the full nested Hilbert scheme $\N^k(X)$, but we integrate over the geometric subset $\N^{A_1,\ldots, A_s,\mu}(X)$ and hence by abusing the notation, we will 
%\end{remark}

\subsection{Extending the integration domain} 
If $(A_1,\ldots, A_s)$ is curvilinear regular tuple as in Definition \ref{def:curvilinear} and Definition \ref{def:regular}, then 
\[\Hilb^{A_1+\ldots +A_s;\mu}(\CC^n) = \CHilb^k(\CC^n) \cap \Hilb^{A_1,\ldots, A_s,\mu}(\CC^n)\]
and for stable $f$, by Proposition \ref{prop:support}, $\supp(\Phi_f(V^{[k]}))$ intersects the punctual part $\Hilb^{A_1,\ldots ,A_s}_0(\CC^n)$ only points in $\Hilb^{A_1+\ldots +A_s;\mu}(\CC^n)$. Hence $\Phi^{\mu|\Lambda}(V^{[k]})$ is supported on 
\[\widehat{\CN}^{A_1+\ldots +A_s}(\CC^n)=\pi_{\Lambda}^{-1}(\widehat{\Hilb}^{A_1+\ldots +A_s}(\CC^n)).\]

Let 
\begin{equation}\label{emultb}
 \epd{A_1,\ldots, A_s} =\epd{\Hilb^{A_1,\ldots ,A_s;\mu}(\CC^n),\BHilb^k(\CC^n)}
\end{equation}
denote the equivariant dual of the possibly singular geometric subset in the neighborhood we constructed as the balanced Hilbert scheme. Since $\BHilb^k(\CC^n)$ might be singular, the global equivariant dual is not well-defined, but Diagram \eqref{diagramfive} provides a smooth ambient space 
\[\widehat{\BHilb}^k(\CC^n) \subset \widehat{\grass}_{k}(S^\bullet TX) \] 
and for a torus fixed point $F \in \Hilb^{A_1,\ldots ,A_s;\mu}(\CC^n)$ we let 
\[\emu_F[\Hilb^{A_1,\ldots ,A_s;\mu}(\CC^n),\BHilb^k(\CC^n)]=\frac{\emu_F[\widehat{\Hilb}|^{A_1,\ldots ,A_s;\mu}(\CC^n),\widehat{\grass}]}{\emu_F[\widehat{\BHilb}^k(\CC^n),\widehat{\grass}]}\]
denote the $T$-equivariant multidegree at $F$.
Then  
\begin{equation}\label{integralextend}
\int_{\widehat{\BN}^{A_1,\ldots, A_s;\mu|\Lambda}(\CC^n)}\Phi_f(V^{[k]})= \int_{\widehat{\BN}^{k}(\CC^n)} \tilde{\Phi}_f^{\mu|\Lambda}(V^{[k]}) 
\end{equation}
where we define the integral on the right hand side as an Atiyah-Bott localisation sum where $\tilde{\Phi}_f^{\mu|\Lambda}(V^{[k]})$ is a rational form whose restriction to any torus fixed point $F \in \tilde{\BN}^k(\CC^n)$ is
\begin{equation}\label{tildeatF}
\tilde{\Phi}_{f,F}^{\mu|\Lambda}(V^{[k]})=\Phi_{f,F}^{\mu|\Lambda}(V^{[k]}) \cdot \pi_\Lambda^* \emu_{\pi_\Lambda(F)}[\Hilb^{A_1,\ldots ,A_s;\mu}(\CC^n),\BHilb^k(\CC^n)].
\end{equation}

%\begin{definition} Let  
%\[\tilde{\Phi}_f^{\mu|\Lambda}=\Phi_f^{\mu|\Lambda} \cdot \pi_\Lambda^* \epd{A_1,\ldots, A_s}.\]
%For a point $p\in \CN^{A_1+\ldots +A_s;\mu}(\CC^n)$ let 
%\begin{equation}\label{emultd}
% \tilde{\Phi}_{f,p}^{\mu|\Lambda}= \emu_p[A_1,\ldots, A_s] =\emu_p[\CN^{A_1,\ldots ,A_s;\mu|\Lambda}(\CC^n),\widehat{\BN}^{k}(\CC^n)]
%\end{equation}
%denote the equivariant multiplicity at the curvilinear point $p$. 
%\end{definition}

The next crucial step in our argument is to reduce the equivariant integration of a form which is supported on a small neighborhood of $\widehat{\CN}^{A_1+\ldots +A_s;\mu}(\CC^n)$ to an integral over $\widehat{\CN}^{A_1+\ldots +A_s}(\CC^n)$ itself. 

\subsection{Localisation over the flag} 
Let $V$ be a rank $r$ T-equivariant bundle over $\CC^n$ with T-equivariant Chern roots $\theta_1,\ldots, \theta_r$.  
Let $\l_1,\ldots, \l_n\in \mathfrak{t}^*$ denote the torus weights for the diagonal $T \subset \GL(n)$ action on $\CC^n$ with respect to the basis $e_1,\ldots, e_n \in \CC^n$ and let 
\[\ff=(\Span(e_1) \subset \Span(e_1,e_2)\subset \ldots \subset \Span(e_1,\ldots,e_{k-1}) \subset \CC^n)\]
denote the standard flag in $\CC^n$ fixed by the parabolic $P_{n,k} \subset \GL(n)$.

The Atiyah-Bott-Berline-Vergne  localisation formula of Proposition \ref{abbv} over the flag $\flag_{k-1}(\CC^n)$ gives  
\begin{equation} \label{flagloc}
\int_{\widehat{\BN}^k(\CC^n)} \tilde{\Phi}^{\mu|\Lambda}(V^{[k]})= \sum_{\sigma\in S_n/S_{n-k+1}}
\frac{\tilde{\Phi}^{\mu|\Lambda}_{\sigma(\ff)}}{\prod_{1\leq j \leq
k-1}\prod_{i=j+1}^m(\lambda_{\sigma\cdot
    i}-\lambda_{\sigma\cdot j})},
\end{equation}
where 
\begin{itemize}
\item $\sigma$ runs over the ordered $k-1$-element subsets of $\{1,\ldots, n\}$ labeling the torus-fixed flags $\sigma(\ff)=(\Span(e_{\sigma(1)}) \subset \ldots \subset \Span(e_{\sigma(1)},\ldots, e_{\sigma(k-1)}) \subset \CC^n)$ in $\CC^n$.
\item $\prod_{1\leq j \leq k-1}\prod_{i=j+1}^n(\lambda_{\sigma(i)}-\lambda_{\sigma(j)})$ is the equivariant Euler class of the tangent space of $\flag_{k-1}(\CC^n)$ at $\s(\ff)$.
\item Let $\widehat{\BN}_{\sigma(\ff)}=\rho^{-1}(\sigma(\ff)) \subset \widehat{\BN}^{k}(\CC^n)$ denote the fiber over $\bff$. Then $\Phi^\Lambda_{\sigma(\ff)}=(\int_{\widehat{\BN}_{\sigma(\ff)}} \tilde{\Phi}^{\mu|\Lambda})^{[0]}(\sigma(\ff))\in S^\bullet \mathfrak{t}^*$ is the differential-form-degree-zero part evaluated at $\sigma(\ff)$.
\end{itemize}
The Chern roots of the tautological bundle over $\flag_{k-1}(\CC^n)$ at the fixed point $\sigma(\ff)$ are represented by $\l_{\s(1)}, \ldots ,\l_{\s(k-1)}\in \mathfrak{t}^*$ and therefore 
\begin{equation}\label{alphasigmaf}
\tilde{\Phi}^{\mu|\Lambda}_{\s(\ff)}=\sigma \cdot \tilde{\Phi}^{\mu|\Lambda}_\ff=\tilde{\Phi}^{\mu|\Lambda}_\ff(\l_{\s(1)}, \ldots ,\l_{\s(k-1)})\in S^\bullet \mathfrak{t}^*,
\end{equation}
is the $\sigma$-shift of the polynomial $\tilde{\Phi}^{\mu|\Lambda}_{\ff}=(\int_{\BN_\ff} \tilde{\Phi}^{\mu|\Lambda})^{[0]}(\ff)\in S^\bullet \mathfrak{t}^*$ corresponding to the distinguished fixed flag $\ff$. By \eqref{cdff} the restriction of $V^{[k]}$ to the curvilinear part $\widehat{\CN}^{k}(\CC^n)$ is the tensor product
\[V^{[k]}|_{\widehat{\CN}^{k}(\CC^n)}=V \otimes  \calo_{\CC^n}^{[k]}\]
The test curve model formulated in Theorem \ref{embedgrass} (3) then tells that  
\[V^{[k]}=V \otimes \calo_{\CC^n}^{[k]}=V \otimes \cale\]
where $\cale$ is the tautological bundle over $\grass_{k}(S^\bullet \CC^n)$. Hence the Chern roots of $V^{[k]}$ on $\widehat{\CN}^{k}(\CC^n)$, and hence at the torus fixed points, are  the pairwise sums formed from Chern roots of $V$ and Chern roots of $\cale$. This means that  
\begin{equation}\label{cdff}
\Phi^{\mu|\Lambda}(V^{[k]})=\Phi^{\mu|\Lambda}(\theta_j,\l_i+\theta_j: 1\le i \le m, 1\le j \le r)
\end{equation}
is a polynomial in the tautological bundle.

\subsection{Transforming the localisation formula into iterated residue}\label{subsec:transform}
In order to handle the complex fixed point data in the Atiyah-Bott localisation formula efficiently and prove the residue vanishing theorem, we follow a crucial step outlined in \cite{bsz}, which involves transforming the right-hand side of equation (35) into an iterated residue. This process condenses the symmetry of the fixed point data and simplifies the combinatorial complexity. 

Proposition \ref{prop:ABtoresidue} for $d=k-1$, together with \eqref{flagloc},\eqref{alphasigmaf} and \eqref{cdff} gives
\begin{corollary}\label{propflag} Let $k\le m$. Then  
\begin{equation*}
\int_{\widehat{\BN}^{k}(\CC^n)} \tilde{\Phi}^{\mu|\Lambda}(V^{[k]})=\sires
\frac{\prod_{1\leq i<j\leq k-1}(z_i-z_j)\, \tilde{\Phi}^{\mu|\Lambda}_\ff(z_1, \ldots ,z_{k-1})\dbz}
{\prod_{l=1}^{k-1}\prod_{i=1}^m(\lambda_i-z_l)}
\end{equation*}
where 
\begin{itemize}
\item $\tilde{\Phi}_\ff^{\mu|\Lambda}=\int_{\widehat{\BN}_{\ff}} \tilde{\Phi}^{\mu|\Lambda}$ is the integral over the 
fiber, which we will calculate using a second equivariant localisation.
\item $z_1,\ldots, z_{k-1}$ are the $T$-weights of the tautological bundle over $\flag_{k-1}(\sym^{\le k-1}\CC^n)$. Equivalently, the $\GL(n)$ action on $\Hom^\ff(\CC^{k-1},\CC^n)$ reduces to $\GL(\CC_{[k-1]}) \subset \GL(n)$, and $z_1,\ldots z_{k-1}$ are the weights of the $T^{k-1}_\bz \subset \GL(\CC_{[k-1]})$ action. 
\end{itemize}
\end{corollary} 
\subsection{Second equivariant localisation over the base flag}\label{subsec:secondlocalization} Next we proceed a second equivariant localisation on $\widehat{\BN}_\ff$ to compute $\tilde{\Phi}^\Lambda_\ff(z_1,\ldots ,z_{k-1})$.  
By Proposition \ref{prop:support} the torus fixed points sit in the curvilinear locus 
\[\widehat{\CN}_\ff^{A_1+\ldots +A_s;\mu}(\CC^n) \subset \widehat{\CN}_\ff\]
where the lower index $\ff$ denotes the fiber over the distinguished flag $\ff$. The curvilinear locus $\widehat{\CN}_\ff$ sits over 
\[\widehat{\CHilb}_\ff=\rho^{-1}(\ff)\simeq \overline{P_{n,k-1}\cdot p_{n,k-1}} \subset \grass_{k-1}(\sym^{\le {k-1}}\CC_{[k-1]})\]
where  
\[p_{n,k-1}=\Span(e_1, \ldots, \sum_{\tau \in \mathcal{P}(k-1)}e_{\tau}) \in \grass_{k-1}(\sym^{\le k}\CC_{[k-1]}),\] 
and $P_{n,k-1} \subset \GL_n$ is the parabolic subgroup which preserves $\ff$. Equivalently, 
\[\widehat{\CHilb}_\ff=\overline{\phi(\Hom^\ff(\CC^{k-1},\CC^n)}\]
 where 
 \[\Hom^\ff(\CC^{k-1},\CC^n)=\{\psi \in \Hom(\CC^{k-1},\CC^n): \psi(e_i)\subset \CC_{[i]} \text{ for } i=1,\ldots, k-1\}\]
and $\CC_{[i]} \subset \CC^n$ is the subspace spanned by $e_1,\ldots, e_i$. 

The torus fixed points in $\widehat{\CHilb}^{k}_\ff \subset \grass_{k-1}(\sym^{\le k-1}\CC_{[k-1]})$ are subspaces of the form
\[F_\tau=\mathrm{Span}(e_{\tau_1}, e_{\tau_2} , \ldots , e_{\tau_{k-1}})\]
parametrised by $k-1$-tuples $\tau=(\tau_1,\ldots, \tau_{k-1})$ where $\tau_i \neq \tau_j$ for $1\le i <j \le k-1$ and 
\begin{equation}\label{tauconditions}
\tau_i = \{t^i_1,\ldots, t^i_{|\tau_i|}\} \text { such that } 1\le t^i_j \le i \text{ and } \sum_{j=1}^{|\tau_i|} t^i_j \le i \text{ for all } 1\le i \le k-1 
\end{equation}
We call these admissible $k-1$-tuples, and those $\tau$'s which correspond to fixed points in $\widehat{\CHilb}^{k}_\ff$ are called curvilinear admissible subsets, the set of these is $\mathcal{P}(k-1)$. 
 
 A point in $\widehat{\CN}_\ff \subset \prod_{\a \in \Pi(k)} \CHilb^\a(\CC_{[k-1]})$ has the form  $\xi=(\xi_A: A\subset \{1,\ldots, 
 k\})$. If this is torus fixed, then $\phi^\grass(\xi_{\{1\ldots, 
 k\}})=W_\tau \subset \sym^{\le k-1}\CC_{[k-1]}$ is a $k-1$-dimensional subspace for some $\tau=(\tau_1, \ldots, \tau_{k-1})$ and 
$\phi^\grass(\xi_A)=F_{A,\tau} \subset F_\tau$
is a torus-fixed subset of dimension $|A|-1$, hence 
\[F_{A,\tau}=\mathrm{Span}(e_{\tau_j}:j\in \Lambda_A)\]
for some $\Lambda_A \subset \{1,\ldots, k\}$ with $|\Lambda_A|=|A|-1$. The $T_\bz^{k-1}$ weights on the tautological bundle $\cale$ over $\grass_{k-1}(\symdot)$ are $z_1,\ldots, z_{k-1}$, hence for a subset $A \subset \{1,\ldots, k\}$ the equivariant Chern roots (i.e. torus weights) of $F_{A,\tau}$ are $\{z_{\tau_j}: j \in \Lambda_{A}\}$
where $z_{\tau_j}=\sum_{i \in \tau_j} z_i$. A simple but crucial consequence of \eqref{tauconditions} is the following
\begin{lemma}\label{crucial1} \begin{enumerate}
\item If $\tau \neq \{[1],[2],\ldots, [k-1]\}$, then at least one integer $2\le i \le k-1$ does not appear in $\tau$. Hence the $T^{k-1}_\bz$-weight of $F_{A,\tau}$ does not depend on $z_i$ for all $A$. 
\item If $\tau_{k-1} \neq [k-1]$ then $\tau$ does not contain $k-1$. In other words, the torus fixed points in $\widehat{\CHilb}^{k}_\ff$ which contain the weight $z_{k-1}$ are of the form $\tau=(\tau_1,\ldots, \tau_{k-2},[k-1])$.  
\end{enumerate}
\end{lemma}

Let $\mathcal{F}_\ff$ denote the set of torus fixed points on $\widehat{\CN}^{A_1+\ldots +A_s;\mu}_\ff$, which naturally sits in $\widehat{\CN}^{k}_\ff$.  Diagram \eqref{diagramsix} provides a smooth ambient space 
\[\widehat{\BN}_\ff \subset \widehat{\grass}=\prod\limits_{(\a_1,\ldots, \a_s) \in \Pi(k)} \prod\limits_{i=1}^s \widehat{\grass}_{|\a_i|}(S^\bullet \CC_{[k-1]})\] 
and for $F \in \mathcal{F}_\ff$ let $\emu_F[\widehat{\BN}_\ff,\widehat{\grass}]$ denote the $T$-multidegree at $F$. The Rossman localisation formula with Corollary \ref{propflag} gives  
\begin{equation}\label{intnumberone} 
\int_{\widehat{\BN}^{k}(\CC^n)}\tilde{\Phi}^{\mu|\Lambda}(V^{[k]})=
\sum_{F\in \mathcal{F}_\ff} \sires \frac{\emu_F[\widehat{\BN}_\ff,\widehat{\grass}] \prod_{i<j}(z_i-z_j) \tilde{\Phi}^{\mu|\Lambda}_F(z_1,\ldots ,z_{k-1})}{
\Euler_F(\widehat{\grass})  \prod_{l=1}^{k-1}\prod_{i=1}^m(\lambda_i-z_l)} \,\dbz.
  \end{equation}
Now recall that
\[\Phi^{\mu|\Lambda}(V^{[k]})=\Phi(V^{[k]}) + \sum \limits_{\a \in \Pi(s)\setminus \Lambda} (-1)^\a \Phi(V^{\mu|\a}).\]
Let $F=(F_{A,\tau}: A \subset \{1,\ldots, {k-1})$ be a torus fixed point on $\widehat{\CN}_\ff$, where $\tau=(\tau_1,\ldots, \tau_{k-1})$. Next we identify $\tilde{\Phi}_F^{\mu|\Lambda}(\bz)$ in the formula. 
Recall from the introduction that $V(z)$ stands for the bundle $V$ tensored by the line $\mathbb{C}_z$, which is the representation of a torus $T$ with weight $z$. Hence its Chern roots are $z+\theta_1,\ldots ,z+\theta_r$. For $z=z_{\tau_j}$ the Chern roots of $V(z_{\tau_j})$ are $(\sum_{l\in \tau_j}z_l + \theta_i: 1\le i \le r)$ and for $A \subset \{1,\ldots, k-1\}$ we write  
\[V(\bz^{A,F})=V \oplus \left(\oplus_{j \in \Lambda_A} V(z_{\tau_j})\right),\]
which has rank $r|A|$.  
Let $\mu|\a=(\mu|\a_1,\ldots, \mu|\a_t) \in \Pi(k)$ be the partition of the $k$ points which we get by merging points according to the partition $\a \in \Pi(s)$ of the support, that is, $\mu|\a_i=\cup_{j\in \a_i}\mu_j$. Recall from \eqref{cdff} that on the curvilinear locus $V^{\mu|\a}$ is the tensor product 
\[V^{\mu|\a}=V^{[\mu|\a_1]}\oplus \ldots \oplus V^{[\mu|\a_t]}=(V \otimes \calo_{\CC^n}^{[\mu|\a_1]}) \oplus \ldots \oplus (V \otimes \calo_{\CC^n}^{[\mu|\a_t]})\]
and by the test curve model formulated in Theorem \ref{embedgrass} the 
tautological bundle is $\calo_{\CC^n}^{[k]}/\calo_\grass=\cale$, where 
$\cale$ is the tautological bundle over $\grass_{k-1}(\sym^{\le k-1}\CC^n)$. Hence the Chern roots of $V^{\mu|\a}=V^{[\mu|\a_1]}\oplus \ldots \oplus V^{[\mu|\a_t]}$ at the fixed point $F=(F_{A,\tau}: A\subset \{1,\ldots, k-1\})$ are 
\[\{\theta_\ell, \theta_\ell+z_{\tau_j}: 1\le \ell \le r, 1\le i \le t, j \in \Lambda_{\a_i}\}\]
Hence if we write
\[\Phi(V(\bz^{\mu|\a,F}))=\Phi(\oplus_{i=1}^t V(\bz^{\mu|\a_i,F}))\]
then 
\begin{equation}\label{phiviaz}
\Phi_F^{\mu|\Lambda}(\bz)=\Phi(V(\bz^{\mu|\Lambda,F}))+\sum \limits_{\a \in \Pi(s)\setminus \Lambda} (-1)^\a \Phi(V(\bz^{\mu|\a,F}))
\end{equation}
and hence 
\begin{equation}\label{tildephiviaz}
\tilde{\Phi}_F^{\mu|\Lambda}(\bz)=\pi_\Lambda^* \emu_{\pi_\Lambda(F)}[\Hilb^{A_1,\ldots ,A_s;\mu}(\CC^n),\BHilb^k(\CC^n)] \cdot \sum \limits_{\a \in \Pi(s)} (-1)^\a \Phi(V(\bz^{\mu|\a,F}))
\end{equation}
which we can substitute into \eqref{intnumberone}. We use the shorthand notation 
\begin{equation}\label{tildephi}
\tilde{\Phi}(V(\bz^{\mu|\a,F}))=\Phi(V(\bz^{\mu|\a,F}) \pi_\Lambda^* \emu_{\pi_\Lambda(F)}[\Hilb^{A_1,\ldots ,A_s;\mu}(\CC^n),\BHilb^k(\CC^n)]
\end{equation} 
for the term corresponding to $\a$.

\subsection{Residue Vanishing theorem on fully nested Hilbert schemes}
The price for using the fully nested Hilbert scheme and the sieve formula to reduce integration to the curvilinear part is that the curvilinear locus $\widehat{\CN_\ff}$ over the flag $\ff$ is not irreducible: besides the main component there are further extra components.  The following theorem say that the contribution of the fixed points sitting on extra components is zero. 

\begin{theorem}{\textbf{Residue Vanishing Theorem on the fully nested Hilbert scheme}}\label{vanishing1}  Let $\a=(\a_1,\ldots, \a_t)\in \Pi(s)$ be a partition with $t>1$. Then the corresponding sum in \eqref{intnumberthree} vanishes:
\begin{equation}\label{eqn:residue}
\sum_{F\in \mathcal{F}_\ff} \sires \frac{\emu_F[\widehat{\BN}_\ff,\widehat{\grass}] \prod_{i<j}(z_i-z_j) \tilde{\Phi}(V(\bz^{\mu|\a,F}))\dbz}{
\Euler_F(\widehat{\grass})  \prod_{l=1}^{k-1}\prod_{i=1}^m(\lambda_i-z_l)}=0.
\end{equation}
\end{theorem}

\begin{proof}
The argument of the proof of Theorem in \cite{berczitau2} works with small amendments. We denoted by $\widehat{\CN}_\ff^{main} \subset \widehat{\CN}_\ff$ the main component, which maps dominantly to $\widehat{\CHilb}^k_\ff$ under $\pi_\Lambda$. Let $F=(F_{A,\tau}: A \subset \{1,\ldots, k-1\}) \in \widehat{\CN}^{A_1,\ldots, A_s;\mu}_\ff \subset \widehat{\CN}_\ff$ be a torus fixed point with $\tau=([1], \tau_2,\ldots, \tau_{k-1})$. In \cite{berczitau2} we show that 
\begin{itemize}
\item[(i)] If $\tau_{k-1}\neq [k-1]$ or $\tau_{k-2}\neq [k-2]$ then $V(\bz^{\a,\tau})$ does not contain $z_{k-1}$, hence the corresponding residue is zero.
\item[(ii)] If $\a=(\a_1,\ldots, \a_t)$ with $t>1$ and $F\in \widehat{\CN}_\ff^{main}$ then $V(\bz^{\a,\tau})$ does not contain $z_{k-1}$, hence the corresponding residue is zero. 
\end{itemize}
Hence the fixed points $F$ where the residue does not vanish have the form $F=(F_{A,\tau}: A \subset \{1,\ldots, k-1\})$ with $\tau=([1], \tau_2,\ldots, \tau_{k-3},[k-2],[k-1])$, and hence they sit in the component 
\[\widehat{\CN}_\ff^{bound}=\pi_{\Lambda}^{-1}(\widehat{\CHilb}^{bound}_\ff)\]
where 
\[\widehat{\CHilb}^{bound}_\ff=\{p_{k-1} \wedge v_{k-1}: [v_{k-1}] \in \PP[\CC_{[k-1]}], p_{k-1} \in \widehat{\CHilb}_{\ff_{-1}}\} \subset \widehat{\CHilb}_\ff\]
is a boundary divisor of $\widehat{\CHilb}_\ff$. We show that 
\begin{itemize}
\item[(iii)] $\widehat{\CN}_\ff^{bound}$ is a component of $\widehat{\CN}_\ff$ which has codimension 1 in $\widehat{\BN}_\ff$, and normal bundle $N^\ff$.
\item[(iv)] The normal direction $N^\ff_F=\emu_F[\widehat{\CN}_\ff^{bound},\widehat{\BN}_\ff]$ at a fixed point $F \in \widehat{\CN}_\ff^{bound} \setminus \widehat{\CN}^{main}_\ff$ satisfying sits in $\CC_{[k-1]}$. In other works, the first order deformations of $F$ are not punctual, they are supported in at least two points.
\end{itemize}
Next, in \cite{berczitau2} we introduced the subset 
\[\widehat{\CHilb}^{e_{k-1}}_\ff=\{p_{k-1} \wedge e_{k-1}: p_{k-1} \in \widehat{\CHilb}_{\ff_{-1}}\} \subset \widehat{\CHilb}_\ff^{bound}\]
and 
\[\widehat{\CN}_\ff^{e_{k-1}}=\pi_{\Lambda}^{-1}(\widehat{\CHilb}^{e_{k-1}}_\ff) \subset \widehat{\CN}_\ff^{bound}\]
and finally 
\[\widehat{\CN}_\ff^{e_{k-1},\perp}=\pi_{\Lambda}^{-1}(\widehat{\CHilb}^{e_{k-1}}_\ff) \cap \pi_{[1k]}^{-1}(e_{k-1}) \subset \widehat{\CN}_\ff^{bound}\] 
These fit into the diagram 
\[\xymatrix{ \widehat{\CN}_\ff^{e_{k-1},\perp} \ar@{^{(}->}[r]^{\iota} & \widehat{\CN}_\ff^{e_{k-1}}  \ar@{^{(}->}[r]^j & \widehat{\CN}_\ff^{bound}}\] 
where 
\begin{itemize}
\item[(v)] The normal bundle of $\iota$ is $\PP^1$ and $\Phi(V(\bz^{\mu|\a,\tau}))$ is constant along the fibers of $\iota$.
\item[(vi)] The normal bundle of $j$ is $\prod_{i=1}^{k-2}(e_{k-1}-e_i)$.
\end{itemize}
The crucial point is that (v) is true with $\tilde{\Phi}(V(\bz^{\mu|\a,\tau}))$ too: the multidegree factor in \eqref{tildephi} is pulled back from $\widehat{\CHilb}^k_\ff$, and the normal of $\iota$ sits in the fiber of $\pi_\Lambda$. The rest of the proof is the same as in \cite{berczitau2}: the integration along the normal bundle of $\iota$ has the same form, and we conclude by iteration of the formula that the sum of the residues over the fixed points must vanish. 
\end{proof}

\subsection{Reducing integration from the fully nested to the curvilinear Hilbert scheme}

The residue vanishing theorem reduces the formula \eqref{intnumberone} to 
\begin{equation}\label{intnumberthree} 
\int_{\widehat{\BN}^{k}(\CC^n)}\tilde{\Phi}_f^{\mu|\Lambda}(V^{[k]})
=\sum_{F\in \mathcal{F}_\ff} \sires \frac{\emu_F[\widehat{\BN}_\ff,\widehat{\grass}] \prod\limits_{i<j} (z_i-z_j) \tilde{\Phi}(V(\bz^{\mu|\Lambda,F}))}{
\Euler_F(\widehat{\grass})  \prod_{l=1}^{k-1}\prod_{i=1}^m(\lambda_i-z_l)} 
  \end{equation}
But the bundle $V^{\mu|\Lambda}$ is pulled-back from a small neighborhood $\CHilb^{A_1+\ldots +A_s;\mu}_\nabla(\CC^n)$ of $\CHilb^{A_1+\ldots +A_s;\mu}(\CC^n)$ in $\Hilb^{A_1,\ldots, A_s;\mu}(\CC^n)$.  Hence the right hand side of \eqref{intnumberthree} is the localisation formula for the integral over the small neighborhood, which we formulate in the following corollary.
\begin{corollary}  
\begin{equation}\label{intnumberfour}
\int_{\widehat{\BN}^{k}(\CC^n)}\tilde{\Phi}_f^{\mu|\Lambda}(V^{[k]})=\int_{\CHilb^{A_1+\ldots +A_s;\mu}_\nabla(\CC^n)} \Phi_f(V^{[k]})
%\sum_{F\in \mathcal{P}_\ff} \sires \frac{
% \prod_{i<j}(z_i-z_j) c_{d,F}(V^{[k+1]})}{
%\Euler_F(\CHilb^{k+1}_\ff(\CC^n))  \prod_{l=1}^k\prod_{i=1}^m(\lambda_i-z_l)} \,\dbz
\end{equation}
%where $\mathcal{P}_\ff=\pi_\Lambda(\mathcal{F}_\ff)$ is the set of $T_\bz^k$-fixed points on $\CHilb^{k+1}_\ff=\rho^{-1}(\ff)$ where $\rho: \widehat{\CHilb}^{k+1}(\CC^n) \to 
%\flag_k(\CC^n)$ as in Diagram \eqref{diagramsix}.  
\end{corollary}

%\subsection{The local model for properly supported forms}

%$\Curv^{k}(\CC^n) \subset \CHilb^{k}(\CC^n)$ is the nonsingular open locus which parametrises curvilinear subschemes. Let $B \to \Curv^{k+1}(\CC^n)$ denote the normal bundle of $\Curv^{k}(\CC^n)$ in $\GHilb^{k}(\CC^n)$.
%\begin{proposition}{(Haiman \cite{haiman})} $B=\calo_{\CC^n}^{[k]}/\calo$ is a rank $k-1$ bundle.
%\end{proposition}
%Note that the tautological bundle $B$ extends over the whole $\GHilb^{k}(\CC^n)$, and in particular, over the closure $\CHilb^{k}(\CC^n)=\overline{\Curv^{k}(\CC^n)}$, and  

In the next sections we will identify the Haiman bundle $B$ above which can be considered as the normal bundle of $\Hilb^{A_1+\ldots +A_s}(\CC^n)$ in $\Hilb^{A_1,\ldots, A_s}(\CC^n)$, and the equivariant integration formula on the curvilinear geometric subset $\Hilb^{A_1+\ldots +A_s}(\CC^n)$ coming from the Kazarian model.

\subsection{The normal bundle of $\Hilb^{A_1+\ldots +A_s}(\CC^n)$ in $\Hilb^{A_1,\ldots, A_s}(\CC^n)$}\label{subsec:normalbundle}

Recall that $\Phi_f(V^{[k]}) \in \Omega^\bullet(\GHilb^{k+1}(\CC^n))$ is properly supported form as in Definition \ref{geocond}. This means that $\supp(\Phi_f(V^{[k]})$ is locally irreducible at every point and 
\[\supp(\Phi_f) \cap \GHilb^{k}_0(\CC^n) \subset \CHilb^{k}(\CC^n)\]
Recall also that for regular tuple $(A_1,\ldots, A_s)$ 
\[\CHilb^{k}(\CC^n) \cap \Hilb^{A_1,\ldots, A_s|\mu}(\CC^n)=\Hilb^{A_1+\ldots +A_s|\mu}(\CC^n).\]
and hence 
\[\supp(\Phi^{\mu|\Lambda}_f(V^{[k]}) \cap \GHilb^{k}_0(\CC^n) \subset \Hilb^{A_1+\ldots +A_s;\mu}(\CC^n)\]

\begin{definition} We say that $\supp(\Phi^{\mu|\Lambda})$ is locally modeled as a bundle $B$, if there is a topological isomorphism $\rho$:
\begin{equation}\label{localmodelgeneral}
\xymatrix{\supp(\Phi^{\mu|\Lambda}) \ar[d]  \ar[r]^\rho &  B  \ar[d]  \\
  \supp(\Phi^{\mu|\Lambda})\cap \GHilb^{k}_0(\CC^n)  \ar@{^{(}->}[r] &  \Hilb^{A_1+\ldots +A_s}(\CC^n)}
 \end{equation}
\end{definition}

In case of a local model $B$ one applies the Thom-isomorphism to obtain 
\begin{equation}
\int_{\CHilb^{A_1+\ldots +A_s;\mu}_\nabla(\CC^n)} \Phi(V^{\mu|\Lambda})=\int_{\Hilb^{A_1+\ldots +A_s;\mu}(\CC^n)} \frac{\Phi(V^{\mu|\Lambda})}{\Euler(B)}
\end{equation}
This formula reduces the integration over $\Hilb^{A_1,\ldots, A_s}(\CC^n)$ to integration over the curvilinear component $\Hilb^{A_1+\ldots +A_s}(\CC^n)$. Applying Theorem \ref{main1} with $A=A_1+\ldots +A_s$ we get the desired formula 
\[\int_{\Hilb^{A_1+\ldots +A_s;\mu}(\CC^n)} \frac{\Phi(V^{\mu|\Lambda})}{\Euler(B)}=\int_X \sires \frac{\prod_{1\le i<j \le k-1}(z_i-z_j)\mathrm{ePD}[Q(A) \subset \mathrm{Alg}_\bd(N_\bullet)]\Phi(V(\bz))d\bz}{\prod_{w(i)+w(j)\le w(m)}(z_i+z_j-z_m)(z_1\ldots z_{k-1})^n \Euler(B)(\bz)}\prod_{i=1}^{k-1} s_X\left(\frac{1}{z_i}\right).\]
for the deepest term in the sieve, which gives the proof of Theorem \ref{main2}

We can't prove that $\Phi^{\mu|\Lambda}_f(V^{[k]})$ is locally modeled for any tuple $A_1,\ldots, A_s$, but we prove   

\begin{proposition}
$\Phi_f^{\mu|\Lambda}(V^{[k]})$ is locally modeled when $\dim(A_1)=\ldots =\dim(A_s)$.
\end{proposition}

\begin{proof}
Let $\mu=(\mu_1,\ldots, \mu_s)$ where $|\mu_1|=\ldots =|\mu_s|=\dim(A_1)=\ldots =\dim(A_s)$. The diagram 
\begin{equation}\label{diagramseven}
\xymatrix{
\Hilb^{A_1+\ldots +A_s}(\CC^n) \ar@{^{(}->}[r] \ar[d]^{\supp} & \BHilb^{A_1,\ldots, A_s}(\CC^n) \ar[d]^{\supp} \\
\CHilb^s(\CC^n) \ar@{^{(}->}[r] & \BHilb^s(\CC^n)}
\end{equation}
which includes the balanced Hilbert schemes is torus-equivariant due to the equal size of the supporting algebras. The codimension of the top horizontal embedding is equal to the codimension of the bottom horizontal embedding (this codimension is $s-1$), and the spaces on the left hand side are irreducible. We conclude that the  $\Hilb^{A_1+\ldots +A_s}(\CC^n)$ in $\BHilb^{A_1,\ldots, A_s}(\CC^n)$ at a generic point is isomorphic to the local neighborhood of $\CHilb^s(\CC^n)$ in $\BHilb^s(\CC^n)$ at a generic point (generic means curvilinear, not boundary point). Hence $\Phi^{\mu|\Lambda}_f(V^{[k]})$ is locally modeled with $\supp^*(B^s)$ where 
\[B^s=\calo_{\CC^n}^{[s]}\calo_{\CC^n}\]
is the tautological bundle which is the local model for $\Phi_f(V^{[s]})$ for stable $f$. 
\end{proof}

Note that if $\Phi^{\mu|\Lambda}_f(V^{[k]})$ is locally modeled then 
\[\Euler(B)=\epd{\Hilb^{A_1+\ldots +A_s}(\CC^n),\Hilb^{A_1,\ldots ,A_s}(\CC^n)}\]
is the equivariant dual. We conjecture that the same integral formula holds with $\Euler(B)$ replaced by this equivariant dual even if our form is not locally modeled. In the following table we collected these duals for Nakajima classes which are geometric subsets with Morin algebras: if $A_d=\CC[t][/t^d$ denotes the Morin algebra of order $d$ then 
\[\mathrm{Nak}(d_1,\ldots, d_s)=\Hilb^{A_{d_1},\ldots, A_{d_s}}(\CC^n),\]
and 
\[\Hilb^{A_{d_1}+\ldots +A_{d_s}}(\CC^n)=\Hilb^{A_{d_1+\ldots +d_s}}(\CC^n)\]
by definition (the sum Morin algebras is Morin). 

\begin{center}
\begin{tabular}{|c|c|c|}
$d_1+\ldots +d_s$ & $(d_1, \ldots , d_s)$ & $\epd{\mathrm{Nak}(d_1+\ldots +d_s) \subset \mathrm{Nak}(d_1,\ldots, d_s)}$ \\
\hline
$2$ & $(1,1)$ & $z_1$\\
\hline
$3$ & $(1,2)$ & $z_2$\\
\hline
$3$ & $(1,1,1)$ & $z_1z_2$\\
\hline
$4$ & $(1,3)$ & $z_3$\\
\hline
$4$ & $(2,2)$ & $z_1$\\
\hline
$4$ & $(1,1,2)$ & $z_1z_2$\\
\hline
$4$ & $(1,1,1,1)$ & $z_1z_2z_3$\\
\hline
 $5$ & $(1,4)$ & $z_4$\\
\hline
 $5$ & $(2,3)$ & $z_1$\\
\hline
 $5$ & $(1,2,2)$ & $z_1z_2$\\
\hline
 $5$ & $(1,1,3)$ & $z_1z_3$\\
\hline
 $5$ & $(1,1,1,2)$ & $z_1z_2z_3$\\
\hline
 $5$ & $(1,1,1,1,1)$ & $z_1z_2z_3z_4$\\
\hline
 $5$ & $(1,5)$ & $z_1$\\
\hline
 $6$ & $(2,4)$ & ?\\
\hline
 $6$ & $(3,3)$ & $z_1$\\
\hline
 $6$ & $(2,2,2)$ & $z_1z_2$\\
 \hline
\end{tabular}
\end {center}

\section{Example: Severi degrees}\label{sec:severi}

Let $S$ be a nonsingular projective surface and $L$ a $5r$-ample line bundle on $S$. Let $N_{r}(L)$ denote the count of $r$-nodal hypersurfaces in a generic linear system $\PP^r \subset |L|$. According to \cite{kleimanpiene}, we can write $N_{r}=P_{r}/r !$ where $P_r$ satisfy the formal identity in $t$
\[\sum_{r \ge 0} \frac{P_r t^r}{r!}=\exp (\sum_{q\ge 1}\frac{a_qt^q}{q!})\]
for some integers $a_0, a_1, \ldots $. In particular, 
\[P_0=1,P_1=a_1,P_2=a_1^2+a_2, P_3 =a_1^3+3a_2a_1+a_3,\ldots .\]
where Kleiman and Piene compute $a_i$ for $i\le 8$, see page 2 in \cite{kleimanpiene}:
\[a_1=3L^2+2Lc_1(S)+c_2(S)\]
\[a_2=-42L^2-39Lc_1(S)-6c_1^2(S)-7c_2(S)\]
\[a_3=1380L^2+1576Lc_1(S)+376c_1^2(S) + 138c_2(S).\]
In the expression $P_r=\sum_{\bi \in \Pi_r}C_{\bi}a^{\bi}$ the terms are indexed by partitions $\bi=(i_1,\ldots, i_s)$ of $r$, hence $i_1+\ldots+i_s=r$. The coefficient $C_{\bi}$ counts the number of ways how $r$ marked points can be partitioned into piles of size $i_1,\ldots, i_s$, that is $C_{\bi}={r \choose i_1\ i_2\ \ldots \ i_s}$. So the term $C_{\bi}a^{\bi}$ corresponds to the $\a=(\a_1,\ldots, \a_s)$ summand in Theorem \ref{main2}, which is an integral over 
$\Hilb^{A_1,\ldots, A_s}(\CC^n)$ where
\[A_j=(i_j,2i_j)=\underbrace{\ydiagram{2, 4}}_{2i_j}.\]
Here $A=\CC[x_1,x_2]/(x^{2r},x^ry,y^2)$ has dimension $3r$. We need to choose a filtration
\[N = N_1 \supset N_2 \supset \ldots \supset N_{m+1}=0\] 
on the vector space 
\[N=\frakm/I\simeq \Span_\CC(z_{10},\ldots, z_{2r-1,0},z_{01},\ldots, z_{r-1,1})\] 
satisfying $N_i \cdot N_j \supset N_{i+j}$. The dimension of the subsequent quotients of this filtration is denoted by $d_i = \dim(N_i/N_{i+1})$ and $\bd=(d_1,\ldots , d_m)$ is the dimension vector which satisfies $d_1+\ldots +d_m=3r-1$.
The formula of Theorem \ref{main1} depends on this filtration, and in general the degree of the Kazarian equivariant dual in the numerator depends on our choice. We have several options in this example:  we can take the canonical filtration $N_i=\mathfrak{m^i}/I$, or the following refinement 
%However, in this example something magical happens: there is a filtration where the Kazarian dual has degree $0$, and hence equal to $1$, so we do not have to worry about it. This distinguished filtration has length $m=2r-1$ and $d_1=\ldots =d_{2r-1}=1$: 
where $N_i$ is spanned by the boxes indexed by $1,2,\ldots, 3r-i$ in the following tableau:
{\small\[\ytableausetup{centertableaux,boxsize=2.5em}
\mathcal{B}=\begin{ytableau}
   2r  &  & 3r-1 \\
     \cdot    & 1 & 2 &  &  &  & &  2r-1
\end{ytableau}\]}
%The order of the boxes is the following: we go from left to right until the next to last in the bottom row, followed by the top row from left to right, and finally adding the rightmost box in the bottom row. 
Then $d_1=\ldots =d_{3r-1}=1$ and the weights are given by the number in the Young tableou: $w(z_{i0})=i$ and $w_{i1}=i+2r$.

\begin{theorem}[\textbf{Severi degree formula}] Introduce the variables indexed by boxes 
\[\ytableausetup{centertableaux,boxsize=2em}
\mathcal{B}=\begin{ytableau}
 z_{01} & z_{11}  &  & z_{r-1,1} \\
     \cdot       & z_{10} &  & z_{r-10}& & & & z_{2r-1,0} 
\end{ytableau}\]
and weights $w(z_{i0})=i$ and $w_{i1}=i+2r$. Then  %for $r>1$ we have (for $r=1$ see the formula below)
\begin{equation}\nonumber
a_r=\sires \frac{\prod\limits_{w(z_{ab})<w(z_{a'b'})}(z_{ab}-z_{a'b'})c_{2r}(L,L+z_{ab}:(a,b)\in \mathcal{B})}{\prod\limits_{a+b\le c\le 2r-1}(z_{a0}+z_{b0}-z_{c0}) \prod\limits_{a+b \le c \le r-1}(z_{a0}+z_{b1}-z_{c1}}
\cdot \frac{\epd{Q(A_{(r,2r)},\mathrm{Alg}(N_\bullet)d\bz}}{(z_{10}\cdots z_{r-10})(\prod_{(a,b)\in \mathcal{B}}z_{ab})^2} \prod_{(a,b)\in \mathcal{B}}s_S\left(\frac{1}{z_{ab}}\right)
\end{equation}
where again 
\begin{itemize}
\item $c_{2r}(L,L+z_{ab}:(a,b)\in \mathcal{B})$ denotes the $2r^{th}$ elementary symmetric polynomial formed from the $3r$ formal Chern roots $L,L+z_{ab},(a,b)\in \mathcal{B}$.
\item $s_S=1/c_S$ is the total Segre class of $S$ and in particular $s_0=1,s_1=c_1(S),s_2=c_1^2-c_2$.
\item $\mathrm{Alg}(N_\bullet)$ is the vector space of filtered commutative algebra structures on the filtration $N_\bullet$, and $Q(A_{r,2r})$ is the subset of those algebras which are isomorphic to $\mathfrak{m}/(x^2r,x^ry,y^2)$, and $\epd{Q(A_{(r,2r)},\mathrm{Alg}(N_\bullet)d\bz}$ is the torus-equivariant dual, which is a homogeneous polynomial in $z_{ab}$ (also called the Kazarian dual).
\end{itemize}
\end{theorem}
The Kazarian dual is determined by the equations of $Q(A_{r,2r})$ in the ambient vector space $N_\bullet$. Recall from \S \ref{subsec:kazarianmodel} that we have torus-equivariant embeddings 
\[Q(A)=\mathrm{Assoc}_A(d,N) \subseteq \mathrm{Assoc}(N_{\bullet}) \subseteq \mathrm{Alg}(N_\bullet),\]
hence 
\[\epd{Q(A_{(r,2r)},\mathrm{Alg}(N_\bullet)d\bz}=P_{(r,2r)}(\bz) \cdot \epd{\mathrm{Assoc}(n_\bullet),\mathrm{Alg}(N_\bullet)d\bz}\]  
is independent of $A$, and determined by the associativity equations \eqref{assoceq}. Here $P_{(r,2r)}$ is a homogeneous polynomial determined by the equations of the algebra $A_{(r,2r)}$ in the space of all filtered commutative associative algebra on $N_\bullet$. We will discuss these equations and calculations of the Kazarian dual in more details in \cite{berczitau4}.

\begin{example}[$r=1$]
The residue variables are 
\[\ytableausetup{centertableaux}
\mathcal{B}=\begin{ytableau}
 z_{01}    \\
     \cdot       & z_{10}  \\
\end{ytableau}\]
We leave as an exercise to check that 
\[a_1=3L^2+2Lc_1(S)+c_2(S)=\sires \frac{(z_{10}-z_{01})^2c_2(L,L+z_{10},L+z_{01})d\bz}{2(z_{10}z_{01})^2}s(1/z_{10})s(1/z_{01})\]
\end{example}
where:
\begin{itemize}
\item $c_2(L,L+z_{10},L+z_{01})$ denotes the second elementary symmetric polynomial formes from the formal Chern roots $L,L+z_{10},L+z_{01}$ and 
\item $s_S=1/c_S$ is the total Segre class of $S$ and in particular $s_0=1,s_1=c_1(S),s_2=c_1^2-c_2$. 
\end{itemize}
\begin{example}[$r=2$] Here the test curve model, which we worked out in Example \ref{example:gottsche} gives a better formula, because the Grassmannian equivariant dual has degree $0$ from dimension count, has it is $1$, whereas calculation of the Kazarian dual involves some associativity equations. We will come back to this example in \cite{berczitau4}.
The residue variables are
\[\ytableausetup{centertableaux}
\mathcal{B}=\begin{ytableau}
 z_{01} & z_{11}   \\
     \cdot       & z_{10} & z_{20} & z_{30} \\
\end{ytableau}\]
The test curve model in this case gives the following identity, which can be checked using Maple:
\small
\begin{multline}\nonumber
-42L^2-39Lc_1(S)-6c_1^2(S)-7c_2(S)=a_2=\\
\sires \frac{\prod_{(a,b)<(a',b')\in \mathcal{B}}(z_{ab}-z_{a'b'})c_4(L,L+z_{10},L+z_{20},L+z_{30},L+z_{01},L+z_{11})\prod_{(a,b)\in \mathcal{B}}s_S\left(\frac{1}{z_{ab}}\right)d\bz}{6z_{10}(2z_{10}-z_{20})(z_{10}+z_{20}-z_{30})(2z_{10}-z_{30})(z_{10}+z_{01}-z_{30})(z_{10}+z_{01}-z_{11})(2z_{10}-z_{11})(z_{10}z_{20}z_{30}z_{01}z_{11})^2} 
\end{multline}
where again 
\begin{itemize}
\item $c_4(L,L+z_{10},L+z_{20},L+z_{30},L+z_{01},L+z_{11})$ denotes the fourth elementary symmetric polynomial formes from the formal Chern roots $L,L+z_{ab},(a,b)\in \mathcal{B}$.
\item$s_S=1/c_S$ is the total Segre class of $S$ and in particular $s_0=1,s_1=c_1(S),s_2=c_1^2-c_2$.
\end{itemize}
\end{example}

%This algorithm works for any monomial geometric subset $P(\lambda_1,\ldots, \lambda_s)$ on arbitrary projective toric variety $X$. The key observation is that the collision of two punctual monomial geometric subset can be controlled by a single monomial geometric subset, namely the one with maximal Wall-index among the possible degenerations along coordinate axes. If $\dim(X)=n$ then we have a \textbf{test $n$-fold model} which gives an embedding into some Grassmannian/partial flag manifold. Localisation, the Haiman-Szenes trick and the generalised residue vanishing theorem gives an iterated residue formula...

\bibliographystyle{abbrv}
\bibliography{thom.bib}

\end{document}